\documentclass[final,10pt]{siamart1116}

\usepackage{hyperref}
\usepackage{epsfig}
\usepackage{amsfonts}
\usepackage{algorithm}
\usepackage{algorithmic}
\usepackage{graphicx}
\usepackage{subfig}
\usepackage{array}
\usepackage[utf8]{inputenc}

\newcommand {\ri}	{{\mathrm i}}

\newtheorem{remark}[theorem]{Remark}

\allowdisplaybreaks

\title{\bf Large-Scale and Global Maximization of the Distance to Instability}
\author{
	Emre Mengi\thanks{%
		 Department of Mathematics,
		 Ko\c{c} University, Rumelifeneri Yolu, 34450 Sar{\i}yer, Istanbul, Turkey,
		{\tt emengi@ku.edu.tr}.}
}

\begin{document}

\maketitle

\begin{abstract}
The larger the distance to instability from a matrix is, the more robustly stable the associated
autonomous dynamical system is in the presence of uncertainties and typically the less severe
transient behavior its solution exhibits. Motivated by these issues, we consider the maximization 
of the distance to instability of a matrix dependent on several parameters, a nonconvex 
optimization problem that is likely to be nonsmooth. In the first part we propose a globally
convergent algorithm when the matrix is of small size and depends on a few parameters.
In the second part we deal with the problems involving large matrices. We tailor a subspace
framework that reduces the size of the matrix drastically. The strength of the tailored subspace
framework is proven with a global convergence result as the subspaces grow and a superlinear
rate-of-convergence result with respect to the subspace dimension.  \\[5pt]
\noindent
\textbf{Key words.}  Eigenvalue optimization, maximin optimization, distance to instability, 
robust stability, subspace framework, large-scale optimization, global optimization,
eigenvalue perturbation theory
\\[5pt]
\noindent
\textbf{AMS subject classifications.}: 65F15, 90C26, 93D09, 93D15, 49K35
\end{abstract}

\section{Introduction}
\label{sec:intro}

Our concern in this work is the maximization of the distance to instability of a matrix dependent on
parameters over a set of admissible parameter values. We assume throughout that the
matrix dependent on parameters can be represented of the form
\[
	A(x) :=	f_1(x) A_1 + f_2(x) A_2 +	\dots		+	f_\kappa(x) A_\kappa,
\]
for given $A_1, \dots, A_\kappa \in {\mathbb C}^{n\times n}$ and $f_1, \dots, f_\kappa : \Omega \rightarrow {\mathbb R}$
that are real analytic on their domain $\Omega$, which is an open subset of ${\mathbb R}^d$.
Moreover the distance to instability of such an $A(x)$ is defined by
\begin{equation*}
	\begin{split}
	{\mathcal D} (A(x))	&	\;\; := \;\;\;	
			\min \{ \| \Delta \|_2 \; | \; \exists z \in {\mathbb C}^+ 
				 			\;\; {\rm det}(A(x) + \Delta - z I) = 0 \}	\\
					&	\;\;\; = \;\;\;	\min_{z \in {\mathbb C}^+}		\sigma_{\min}(A(x) - z I)	\\
					&	\;\;\; = \;\;	\left\{
									\begin{array}{ll}
										\min_{\omega \in {\mathbb R}} \; \sigma_{\min} (A(x) - \omega \ri I)	
										\quad\quad	& {\rm if} \; \Lambda(A(x)) \cap {\mathbb C}^+ = \emptyset \\
										0	
										\quad\quad	& {\rm otherwise} \\
									\end{array}
								\right.	\; ,
	\end{split}
\end{equation*}
where ${\mathbb C}^+ := \{ z \in {\mathbb C} \; | \; {\rm Re}(z) \geq 0 \}$ denotes the closed right-half of 
the complex plane, $\sigma_{\min}(\cdot)$ and $\Lambda(\cdot)$ denote the smallest singular value and spectrum, respectively,
of their matrix arguments. The second equality above is a simple consequence of the Eckart-Young theorem
\cite[Theorem 2.5.3]{Golub1996}, whereas the third equality follows from the maximum modulus principle. 
The condition $\Lambda(A(x)) \cap {\mathbb C}^+ = \emptyset$ is an algebraic way of stating that the 
autonomous system $y' = A(x) y$ is asymptotically stable. If this condition is violated, the distance to 
instability of $A(x)$ is defined to be zero.

The problem at our hands, maximization of the distance to instability, can formally be expressed as
\begin{equation}\label{eq:prob_defn}
	\begin{split}
	\max_{x \in \widetilde{\Omega}}	\; {\mathcal D}(A(x))	  \;\; = \;\;		
				\max_{x \in \widetilde{\Omega}}	\:	\min_{z \in {\mathbb C}^+}	\:	\sigma_{\min}(A(x) - zI)	\hskip 30ex	\\
				\hskip 7ex		   \;\; = \;\;
				\max  \: 
				\left\{
					\min_{\omega \in {\mathbb R}} \; \sigma_{\min} (A(x) - \omega \ri I) \;\; \bigg| \;\; 
					{x \in \widetilde{\Omega}} \;\; {\rm s.t.} \;\; \Lambda(A(x)) \cap {\mathbb C}^+ = \emptyset
				\right\},				
	\end{split}
\end{equation}
with $\widetilde{\Omega}$ denoting a compact subset of $\Omega$.
One classical problem in control theory is stabilization by output feedback control, which, for given
$A \in {\mathbb C}^{n\times n}, B \in {\mathbb C}^{n\times m}, C \in {\mathbb C}^{p\times n}$, involves
finding a $K \in {\mathbb C}^{m \times p}$ such that all eigenvalues of $A + BKC$ are contained in the open
left-half of the complex plane. A more robust stabilization procedure would attempt to minimize
the spectral abscissa (i.e., the real part of the rightmost eigenvalue) of $A + BKC$ over all $K$. Even more robust 
approaches would aim to minimize the $\varepsilon$-pseudospectral abscissa (i.e., the real part of the rightmost
point among the eigenvalues of all matrices within an $\varepsilon$ neighborhood for instance with respect to 
the matrix 2-norm) of $A + BKC$ for a prescribed $\varepsilon > 0$, or maximize the 
distance  to instability of $A + BKC$ \cite{Burke2003}. The latter can be formally stated as
\begin{equation}\label{eq:stab_output}
	\max_{K \in {\mathbb C}^{m\times p}}	\; {\mathcal D}(A + BKC).	
\end{equation}
If the entries of $K$ are constrained to lie in prescribed boxes in the complex plane, then (\ref{eq:stab_output}) falls 
into the scope of (\ref{eq:prob_defn}) that we consider here.

\subsection{Literature}
Optimization of spectral abscissa has been an active field of research since the beginning of 2000s.
It is observed in \cite{Burke2001a, Burke2001} that the nonsmoothness at the optimal point is a common
phenomenon. In particular, it is shown in these works that, for a particular affine family of matrices, the optimal
matrix with the smallest spectral abscissa is a Jordan block, and this property remains to be true for
slightly perturbed problems. These observations led to a concentrated effort put into the development of numerical
algorithms for nonconvex and nonsmooth optimization. For this purpose, a gradient sampling algorithm
is introduced in \cite{Burke2002}, its convergence is analyzed in \cite{Burke2005}. An implementation of
a hybrid algorithm based on this gradient sampling algorithm as well as BFGS, called HANSO, is made
publicly available \cite{Hanso}. A special adaptation of this software HIFOO \cite{Burke2006, Gumussoy2009, Arzelier2011}
for fixed-order ${\mathcal H}_\infty$ controller design has found various applications. HIFOO can be
used to optimize spectral abscissa, pseudospectral abscissa and distance to instability, but it would
converge only to a locally optimal solution. The problem of stabilization by output feedback
is referred as one of the important open problems in control theory in \cite{Blondel1997}, indeed it is well known
that the problem of finding a $K$ with box constraints on the entries of $K$ such that $A +BKC$
has all of its eigenvalues on the open left-half of the complex plane is NP-hard \cite{Nemirovski1993}. 
HIFOO and the related ideas can be applied to find $K$ minimizing the spectral abscissa, pseudospectral abscissa 
or maximizing the distance to instability of $A + BKC$ \cite{Burke2003}, but again it would lead to a solution 
that is locally optimal. In a different direction, bundling techniques \cite{Apkarian2006, Apkarian2006b} and 
spectral bundle methods \cite{Apkarian2008} have been proposed for ${\mathcal H}_\infty$-synthesis and 
${\mathcal H}_\infty$-norm minimization. More recently, sequential quadratic and linear programming techniques 
that take all of the eigenvalues (instead of only the rightmost eigenvalue) into account have been employed for 
spectral abscissa minimization \cite{Kungurtsev2014}. All of these techniques in the literature are meant to find
a locally optimal solution even if there are only a few optimization parameters. Furthermore,
none of them is specifically designed for large-scale problems; for instance, none of these 
approaches is meant for problems on the order of a few thousands.

The current-state-of-the-art regarding the computation of the distance to instability, which for this work 
is the objective function to be maximized, is at a mature stage. There are very reliable numerical techniques 
that converge very quickly and that are meant for small- to medium-scale problems. All of these techniques 
\cite{Byers1988, Boyd1990, Bruinsma1990} to compute the distance to instability of an $n\times n$ matrix $A$ 
are based on repeatedly finding the level-sets of the singular value function $f(\omega) := \sigma_{\min}(A - \omega {\rm i} I)$ 
by extracting the imaginary eigenvalues of a $2n\times 2n$ Hamiltonian matrix. For larger problems, a fixed-point 
iteration is proposed in \cite{Guglielmi2013}, and a technique that operates on the roots of an implicitly 
defined determinant function is discussed in \cite{Freitag2014}. These numerical techniques are meant for larger problems 
and work directly on $n\times n$ problems, but they can get stagnated at local minimizers of the singular value function.
A subspace framework is proposed in \cite{Kressner2014} to cope with large-scale distance to instability computations,
and observed to converge quickly with respect to the subspace dimension.

\subsection{Outline and Contributions}
We first deal with the maximization of the distance to instability when the matrices $A_1, \dots, A_\kappa$ 
involved are of small size in Section \ref{sec:small_prob}. In particular, we discuss how the algorithm 
in \cite{Mengi2014} can be adapted for this purpose. This results in Algorithm \ref{small_alg}, which
is globally convergent, unlike the methods employed in the literature, but aims at problems depending 
on only a few parameters.  

The rest of the paper is devoted to large-scale problems when $A_1, \dots, A_\kappa$ are of large size. 
Subspace frameworks based on one-sided restrictions of the matrix-valued function are proposed. At 
every step of the subspace frameworks the distance to instability is maximized for such a restricted problem, 
then the subspace is expanded with the addition of a singular vector at the maximizing parameter value so 
that Hermite interpolation properties hold between the full and the restricted distance to instability functions 
at this parameter value. We first present a basic framework along these lines, namely Algorithm \ref{alg}
in Section \ref{sec:large_prob}, and later an extended version Algorithm \ref{alge} in Section \ref{sec:extended_sf},
which Hermite-interpolates not only at the optimal parameter values but also at nearby points.
A detailed convergence analysis for the subspace frameworks is carried out in Section \ref{sec:convergence};
remarkably the global convergence of the subspace frameworks is established, and
a superlinear rate-of-convergence with respect to the subspace dimension is deduced for the basic framework 
when $d=1$ and for the extended framework for every $d$. The practical implication of these convergence
results is that the frameworks are capable of computing global maximizers of ${\mathcal D}(A(x))$ 
over $x \in \widetilde{\Omega}$ with high accuracy by replacing $A(x)$ with matrix-valued functions of size 
much smaller. Efficient solutions of the restricted problems are addressed in Section \ref{sec:sf_implement}.
The proposed subspace frameworks in Sections \ref{sec:large_prob} and \ref{sec:convergence} perform the inner 
minimization for the restricted problems on the right-hand side of the complex plane. Section \ref{sec:sf_uniform_stab} 
argues that these inner minimization problems can rather be performed on the imaginary axis if $A(x)$ is
asymptotically stable for all $x \in \widetilde{\Omega}$.
Finally, the performance of the proposed subspace frameworks in practice are illustrated on
examples in Section \ref{sec:num_exps}.

The literature lacks studies that address maximin or minimax optimization problems
involving a prescribed eigenvalue of a large-scale matrix-valued function. To our knowledge,
this is the first work that proposes subspace frameworks to deal with large-scale nature
of such problems; what is striking is the strong convergence results that we deduce
for the subspace frameworks. The proposed subspace frameworks and their convergence
analyses set examples for various other minimax or maximin eigenvalue optimization
problems, including the minimization of the ${\mathcal H}_\infty$-norm and the 
$\varepsilon$-pseudospectral abscissa for a prescribed $\varepsilon$. 
The approach for the small-scale problem is based on \cite{Mengi2014} and
global piece-wise quadratic estimators, whereas the subspace frameworks for the 
large-scale problem are inspired from \cite{Kangal2015}. However, all those previous
works concern the minimization of the $J$th largest eigenvalue, while the problem
we deal here has the maximin structure which introduces additional challenges.

\section{Small-Scale Problems}\label{sec:small_prob}
The algorithm we discuss in this section is meant to compute a globally optimal solution when the problem has a few optimization 
parameters, e.g., $d = 1$ or $d=2$. If there are more than a few parameters, one can for instance resort to HIFOO \cite{Burke2006} 
and be content with a locally optimal solution. Let us first assume that $A(x)$ is asymptotically stable, i.e., 
$\Lambda(A(x)) \cap {\mathbb C}^+ = \emptyset$, for all $x \in \widetilde{\Omega}$. In this case, for the solution 
of (\ref{eq:prob_defn}), we could equivalently deal with
\[
	\max_{x \in \widetilde{\Omega}}	\; \left[ {\mathcal D}(A(x)) \right]^2	\;\; = \;\;		
				\max_{x \in \widetilde{\Omega}}	\: \min_{\omega \in {\mathbb R}} \; \lambda( x,\omega )
\]
where $\lambda(x,\omega)$ denotes the smallest eigenvalue of
$
	M(x,\omega)	:=	(A(x) -\ri \omega I)^\ast	(A(x) - \ri \omega I).
$
Let us define $\omega(x)$ implicitly by
\[
	\omega(x) :=	{\arg\min}_{\omega \in {\mathbb R}} \; \lambda ( x,\omega ).
\]
In the case the global minimizer of the problem on the right is not unique, we define $\omega(x)$ arbitrarily as any global 
minimizer. Hence, the objective that we would like to maximize is
\[
	 \left[ {\mathcal D}(A(x)) \right]^2	\;\; = \;\;		 \lambda ( x,\omega(x) ).
\]

Let us consider an $\widetilde{x}$ where the global minimizer $\omega(\widetilde{x})$ of $\lambda ( \widetilde{x},\omega )$
over all $\omega$ is unique and the eigenvalue $\lambda ( \widetilde{x},\omega(\widetilde{x}))$ is simple.
The eigenvalue function $\lambda ( x,\omega )$ is twice continuously differentiable at $(x,\omega) = (\widetilde{x}, \omega(\widetilde{x}))$.
Throughout this section, we assume that for 
every $\widetilde{x}$ such that the global minimizer $\omega(\widetilde{x})$ is unique and $\lambda ( \widetilde{x},\omega(\widetilde{x}))$ 
is simple, the property 
\begin{equation}\label{eq:2der_positivitiy}
	\partial^2 \lambda ( x, \omega ) / \partial \omega^2 \: |_{(x,\omega) = (\widetilde{x}, \omega(\widetilde{x}))} \neq 0
\end{equation}
holds. Observe that the second derivative above cannot be negative, since $\omega(\widetilde{x})$ is a minimizer of 
$\lambda ( \widetilde{x}, \omega )$ over $\omega \in {\mathbb R}$. Hence, under assumption (\ref{eq:2der_positivitiy}), 
we have
\[
	\partial^2 \lambda ( x, \omega ) / \partial \omega^2 \: |_{(x,\omega) = (\widetilde{x}, \omega(\widetilde{x}))} > 0.
\]
Now the implicit function theorem ensures that the function $\omega(x)$ 
is defined uniquely in a neighborhood of $\widetilde{x}$ and twice continuously differentiable in this neighborhood. 
This in turn implies that $[{\mathcal D}(A(x))]^2$ is twice continuously differentiable in this neighborhood of $\widetilde{x}$. 
Its derivatives satisfy the properties stated by the next theorem. 
For part (i) of the theorem, we refer to \cite[Lemma 5.1]{Mengi2014}, whereas part (ii) is immediate from the fact that 
$\partial \lambda ( x, \omega ) / \partial \omega \: |_{(x,\omega) = (\widetilde{x}, \omega(\widetilde{x}))} = 0$ and the
chain rule. Note that, in what follows, the derivatives $\lambda_x(\cdot), \omega'(x), \lambda_{x\omega}(\cdot)$ are row vectors,
whereas $\nabla^2_{xx} \lambda(\cdot)$ denotes the $d\times d$ Hessian of $\lambda(\cdot)$ with respect to $x$ only.
The notations $\Phi'_{-}(\widetilde{\upsilon}), \Phi'_{+}(\widetilde{\upsilon})$ stand for the left-hand, right-hand derivatives of 
the univariate function $\Phi(\upsilon)$ at $\widetilde{\upsilon}$.
\begin{theorem}\label{thm:der_D}
The following hold for every $\widetilde{x} \in {\mathbb R}^d$: 
	\begin{enumerate}
	\item[\bf (i)] For every $p \in {\mathbb R}^d$, letting 
	$\Phi : {\mathbb R} \rightarrow {\mathbb R}, \;\; \Phi(\upsilon) := [{\mathcal D}(A(\widetilde{x} + \upsilon p))]^2$, 
	we have $\Phi_{-}'(0) \geq \Phi_{+}'(0)$.
	\item[\bf (ii)] If $\widetilde{x}$ is such that the global minimizer $\omega(\widetilde{x})$ of $\lambda ( \widetilde{x}, \omega )$ over all $\omega$ 
	is unique and the eigenvalue $\lambda ( \widetilde{x},\omega(\widetilde{x}))$ is simple, then
		\begin{itemize}
			\item $
					\nabla [{\mathcal D}(A(\widetilde{x}))]^2 \;\; = \;\;   \lambda_x (  x, \omega )^T \: |_{(x,\omega) = (\widetilde{x}, \omega(\widetilde{x}))}$,
				and
			\item $\nabla^2 [{\mathcal D}(A(\widetilde{x}))]^2 \;\; = \;\;  \left\{ \nabla^2_{xx} \lambda ( x, \omega ) + \lambda_{x \omega} ( x,\omega )^T \omega'(x) \right\} |_{(x,\omega) = (\widetilde{x}, \omega(\widetilde{x}))}$.
		\end{itemize}
	\end{enumerate}
\end{theorem}

The next result follows from part (i) of Theorem \ref{thm:der_D}. 
Its proof is similar to that of Theorem 5.2 in \cite{Mengi2014}. Only here the result is in terms of upper envelopes for a smallest 
eigenvalue function, whereas the result in \cite{Mengi2014} introduces lower envelopes for a largest eigenvalue function.
Here and elsewhere, $\lambda_{\max}(\cdot)$ refers to the largest eigenvalue of a Hermitian matrix argument. 
\begin{theorem}[Upper Support Function]\label{thm:sup_func}
Let $x^{(k)}$ be a point such that the global minimizer $\omega(x^{(k)})$ of the eigenvalue function $\lambda ( x^{(k)}, \omega )$
over all $\omega$ is unique and $\lambda ( x^{(k)}, \omega( x^{(k)} ) )$ is simple. 
Furthermore, let $\gamma$ satisfy $\lambda_{\max} ( \nabla^2 [{\mathcal D}(A(x))]^2 )  \leq \gamma$ for all $x$ 
such that the global minimizer $\omega(x)$ is unique and $\lambda ( x,\omega(x))$ is simple. For every $x \in {\mathbb R}^d$, we have
\begin{equation}\label{eq:upper_support}
\begin{split}
	[ {\mathcal D}(A(x)) ]^2		 \;\;\; \leq \;\;		\hskip 63ex \\
		\hskip 4ex		q(x; x^{(k)})	:=	[ {\mathcal D}(A(x^{(k)})) ]^2	+	
						\left\{  \nabla [{\mathcal D}(A(x^{(k)}))]^2 \right\}^T  (x - x^{(k)})	+
															\frac{\gamma}{2} \| x - x^{(k)} \|^2_2.
\end{split}
\end{equation}
\end{theorem}
We refer the function $q(x; x^{(k)})$ in (\ref{eq:upper_support}) as the upper support function for $[ {\mathcal D}(A(x)) ]^2$
about $x^{(k)}$. In \cite{Mengi2014}, based on such support functions, a globally convergent optimization algorithm 
due to Breiman and Cutler \cite{Breiman1993} has been adopted for eigenvalue optimization. Here, we adopt that algorithm
to maximize $[ {\mathcal D}(A(x)) ]^2$ globally. 

\subsection{Analytical Deduction of $\gamma$}
Before spelling out the algorithm formally, let us elaborate on how one can obtain $\gamma$ as in Theorem (\ref{thm:sup_func}) analytically.
\begin{theorem}\label{thm:bound_sdev}
The function ${\mathcal D}(A(x))$ satisfies
\[
\lambda_{\max} ( \nabla^2 [{\mathcal D}(A(\widetilde{x}))]^2 )	
		\;\;	\leq	\;\;	
\lambda_{\max}
\left(  \nabla^2_{xx} M (\widetilde{x},\omega(\widetilde{x})) \right)
\]
for all $\widetilde{x}$ such that the global minimizer $\omega(\widetilde{x})$ is unique and
$\lambda(\widetilde{x}, \omega(\widetilde{x}))$ is simple, where
\[
	\nabla^2_{xx} M(x,\omega)
		\;\;	:=	\;\;	
	\left[
			\begin{array}{cccc}
				\frac{\partial^2 M(x,\omega)}{\partial x_1^2}
						&
				\frac{\partial^2 M(x,\omega)}{\partial x_1 \partial x_2}
						&
					\dots
						&
				\frac{\partial^2 M(x,\omega)}{\partial x_1 \partial x_d}		\\ 
				\frac{\partial^2 M(x,\omega)}{\partial x_2 \partial x_1}
						&
				\frac{\partial^2 M(x,\omega)}{\partial x_2^2}
						&
					\dots
						&
				\frac{\partial^2 M(x,\omega)}{\partial x_2 \partial x_d}		\\
						&
						&
						\ddots
						&												\\
				\frac{\partial^2 M(x,\omega)}{\partial x_d \partial x_1}
						&
				\frac{\partial^2 M(x,\omega)}{\partial x_d \partial x_2}
						&
					\dots
						&
				\frac{\partial^2 M(x,\omega)}{\partial x_d^2}
			\end{array}
	\right].
\]
\end{theorem}
\begin{proof}
The arguments are similar to those in the proof of Theorem 6.1 in \cite{Mengi2017}.
Let $\widetilde{x}$ be a point such that the minimum $\omega(\widetilde{x})$ is unique. 
By part (ii) of Theorem \ref{thm:der_D},
\begin{equation}\label{eq:exp_2der}
	 \nabla^2 [{\mathcal D}(A(\widetilde{x})) ]^2
		\;\; = \;\;
	\left\{ \nabla^2_{xx} \lambda ( x, \omega ) + 
	\lambda_{x \omega} ( x,\omega )^T \omega'(x) \right\} \bigg|_{(x,\omega) = (\widetilde{x}, \omega(\widetilde{x}))}.
\end{equation}
The implicitly-defined function $\omega(x)$ satisfies
\[
	\partial \lambda (  x, \omega(x) ) / \partial \omega  = 0
\]
for all $x$ in a neighborhood of $\widetilde{x}$. Differentiating this equation
with respect to $x$ at $(x,\omega(x)) = (\widetilde{x}, \omega(\widetilde{x}))$ yields
\[
	\omega'(\widetilde{x})		\;\; = \;\;	
	- \left\{  \frac{ \lambda_{x \omega } ( x,\omega ) }{  \lambda_{\omega \omega}( x,\omega ) } \right\}
	\bigg|_{(x,\omega) = (\widetilde{x}, \omega(\widetilde{x}))}
\]
where the derivative in the denominator is positive due to (\ref{eq:2der_positivitiy}). Now plug this expression
for $\omega(\widetilde{x})$ in (\ref{eq:exp_2der}) to obtain
\[
	\nabla^2 [{\mathcal D}(A(\widetilde{x})) ]^2
		\;\; = \;\;
	\left\{ \nabla^2_{xx} \lambda ( x, \omega )  \; - \;
	\frac{  \lambda_{x \omega}( x,\omega )^T \lambda_{x \omega}( x,\omega ) } { \lambda_{\omega \omega}(x,\omega) }
	\right\} \bigg|_{(x,\omega) = (\widetilde{x}, \omega(\widetilde{x}))}.
\]
The last expression
and the formulas \cite{Lancaster1964} for the second derivatives of $\lambda(x,\omega)$
imply
\begin{equation*}
	\begin{split}
			\lambda_{\max} \left(   \nabla^2 [ {\mathcal D}(A(\widetilde{x})) ]^2  \right)	
				\;\;	\leq	\;\;		
	\lambda_{\max} \left( \nabla^2_{xx} \lambda ( \widetilde{x}, \omega(\widetilde{x}) ) \right)	\hskip 32ex \\
	\hskip 5ex
			 =	\;	\lambda_{\max} \left\{ H(\widetilde{x},\omega(\widetilde{x}))	\; + \;
		2 \sum_{j=1}^{n-1} \frac{1}{\lambda ( \widetilde{x},\omega(\widetilde{x})) - \lambda_j( \widetilde{x},\omega(\widetilde{x}))} 
		   				\Re (H_j(\widetilde{x},\omega(\widetilde{x}))) \right\},
	\end{split}
\end{equation*}
where 
\begin{equation*}
	\begin{split}
	[H(\widetilde{x},\omega(\widetilde{x}))]_{k,\ell}	\;	& 	\;\; = \;\;		
			v^\ast	\left\{   \frac{\partial^2 M(\widetilde{x},\omega(\widetilde{x}))}{\partial x_k \partial x_\ell } \right\}	v,	\\
	[H_j(\widetilde{x},\omega(\widetilde{x}))]_{k,\ell}		& 	\;\; = \;\;	
			\left[
					v^\ast	
					\left\{ \frac{\partial M(\widetilde{x}, \omega(\widetilde{x}))}{\partial x_k } \right\}
					v_j
			\right]
			\left[
					v_j^\ast	
					\left\{ \frac{\partial M(\widetilde{x},\omega(\widetilde{x}))}{\partial x_\ell } \right\}
					v
			\right],
	\end{split}
\end{equation*}
and $\lambda_j( \widetilde{x},\omega(\widetilde{x}))$ denotes the $j$th largest eigenvalue of $M(\widetilde{x},\omega(\widetilde{x}))$,
whereas $v_j$, $v$ represent unit eigenvectors corresponding to $\lambda_j( \widetilde{x},\omega(\widetilde{x}))$,
$\lambda ( \widetilde{x},\omega(\widetilde{x}))$.
As shown in the proof of \cite[Theorem 6.1]{Mengi2017}, the term $\Re (H_j(\widetilde{x},\omega(\widetilde{x})))$ is positive semi-definite
implying
\[
	\lambda_{\max} \left(   \nabla^2 [ {\mathcal D}(A(\widetilde{x})) ]^2  \right)	
		\;\;	\leq	\;\;
	\lambda_{\max} ( H(\widetilde{x},\omega(\widetilde{x})) )
		\;\;	\leq	\;\;
		\lambda_{\max} \left( \nabla^2_{xx} M(\widetilde{x},\omega(\widetilde{x})) \right),
\] 
where for the last inequality we again refer to the proof of \cite[Theorem 6.1]{Mengi2017}.
\end{proof}

It follows from
\[
	\nabla^2_{xx} M(\widetilde{x},\omega(\widetilde{x}))
		\;\;	=	\;\;
	\nabla^2 \left[ A(\widetilde{x})^\ast A(\widetilde{x}) \right]
			\; + \;
\omega(\widetilde{x}) \cdot \nabla^2 \left[ \ri A(\widetilde{x})  -  \ri A^\ast (\widetilde{x}) \right],
\]
and the inequality $|\omega(\widetilde{x})| \leq 2 \| A(\widetilde{x}) \|_2$ (see \cite[Lemma 2.1]{VanLoan1985}) that
\begin{equation*}
	\begin{split}
	\lambda_{\max} 
	\left(
		\nabla^2_{xx} M(\widetilde{x},\omega(\widetilde{x}))
	\right)
			&	\;\;	\leq	\;\;\;
		\left\| \nabla^2 \left\{ \sum_{j=1}^\kappa  f_j(\widetilde{x}) A_j  \right\}^\ast \left\{ \sum_{j=1}^\kappa  f_j(\widetilde{x}) A_j  \right\}  \right\|_2
				\; +	\\
			&	\quad \quad
					4 \left\{ \sum_{j=1}^\kappa |f_j(\widetilde{x}) |  \| A_j \|_2 \right\}
					\left\{ \sum_{j=1}^\kappa \| \nabla^2 f_j(\widetilde{x}) \|_2 \| A_j \|_2 \right\}	\\
			&	\;\;	\leq	\;
			 2 \left\{ \sum_{j=1}^\kappa  \| \nabla f_j(\widetilde{x}) \|_2  \| A_j \|_2 \right\}^2 
				+	\\
			&	\quad \quad \;
					6 \left\{ \sum_{j=1}^\kappa |f_j(\widetilde{x}) |  \| A_j \|_2 \right\}
					\left\{ \sum_{j=1}^\kappa \| \nabla^2 f_j(\widetilde{x}) \|_2 \| A_j \|_2 \right\}.
	\end{split}
\end{equation*}
Hence, any upper bound on
\begin{equation*}
	\begin{split}
	\max_{x\in \Omega} \;\;	2 g_1(x)^2	+	6 g_0(x) g_2(x),	\quad
{\rm with}	\;\; g_2(x) := \sum_{j=1}^\kappa \| \nabla^2 f_j(x) \|_2 \| A_j \|_2, \hskip 15ex \\
	g_1(x) := \sum_{j=1}^\kappa \| \nabla f_j(x) \|_2 \| A_j \|_2, 	\;\;
	g_0(x) := \sum_{j=1}^\kappa | f_j(x) | \| A_j \|_2
	\end{split}
\end{equation*}
is a theoretically sound choice for $\gamma$ in Theorem \ref{thm:sup_func}.

The expression above may look complicated at first look, however, for instance, in the affine case 
when $A(x) = B_0 + \sum_{j=1}^d x_j B_j$, that is considered widely in the literature,
it leads to the conclusion that 
\[
	\gamma \;\; = \;\;  2 \left\{ \sum_{j=1}^d \| B_j \| \right\}^2
\] 
is a sound choice that can be used in Theorem \ref{thm:sup_func}. 
A slightly tighter bound in this affine case can be obtained by observing
\begin{equation}\label{eq:matrix_forgamma}
	\nabla^2_{xx} M(x,\omega)
		\;	=	\;
	\left[
		\begin{array}{cccc}
			2 B_1^\ast B_1		&	B_1^\ast B_2 + B_2^\ast B_1	&	\dots		&	B_1^\ast B_d  +  B_d^\ast B_1  \\
			B_2^\ast B_1 + B_1^\ast B_2	&	2 B_2^\ast B_2		&			&	B_2^\ast B_d + B_d^\ast B_2	\\
				\vdots							&		&	\ddots	&	\\
			B_d^\ast B_1  +  B_1^\ast B_d		&	B_d^\ast B_2 + B_2^\ast B_d	&		&	2 B_d^\ast B_d
		\end{array}
	\right],
\end{equation}
so $\gamma$ can also be chosen as the largest eigenvalue of the matrix on the right-hand side above.

\begin{remark}\label{rem:nonsmooth_rcase}
It is often assumed so far that $\lambda(\widetilde{x}, \omega)$ has $\omega(\widetilde{x})$ as its
sole global minimizer, which is needed to ensure the twice continuous differentiability of ${\mathcal D}(A(x))$
at $\widetilde{x}$. This assumption is always violated if the matrices $A_1, \dots, A_\kappa$, hence the 
matrix-valued function $A(x)$ for all $x \in \Omega$, are real; in this case the singular value function
has the symmetry $\sigma_{\min}(A(x) - \omega {\rm i} I) = \sigma_{\min}(A(x) + \omega {\rm i} I)$ for all $\omega \in {\mathbb R}$,
implying if $\omega(x)$ is a global minimizer, so is $-\omega(x)$. This symmetry in the global minimizers
in the real case does not cause nondifferentiability, as long as there is a unique nonnegative 
(and a unique nonpositive) global minimizer, as the symmetry is preserved for all $x$.
Hence, all of the discussions and results in this section carry over to the real matrix-valued setting provided 
$\lambda(\widetilde{x}, \omega)$ has a unique nonnegative global minimizer over $\omega$.
\end{remark}

\subsection{The Algorithm}
The algorithm that we employ for small-scale problems is borrowed from \cite{Breiman1993, Mengi2014}.
It is presented in Algorithm \ref{small_alg} below for completeness. 
At the $k$th iteration, the piecewise quadratic function $\min \{ q(x; x^{(j)}) \; | \; j = 1,\dots, k-1 \}$ 
that lies above $[{\mathcal D}(A(x))]^2$ globally is maximized. Then the piecewise quadratic function 
is refined with the addition of a new quadratic piece, namely the upper support function $q(x; x^{(k)})$ about 
the computed maximizer $x^{(k)}$. Every convergent subsequence of the sequence $\{ x^{(k)} \}$ is 
guaranteed to converge to a global maximizer of $[{\mathcal D}(A(x))]^2$ \cite[Theorem 8.1]{Mengi2014}.
We observe in practice that this convergence occurs  typically at a linear rate, but a formal proof of this observation is open.

Solution of the subproblem in line \ref{subproblem}, that is the maximization of the smallest of $k-1$ 
quadratic functions with constant curvature, turns out to be challenging. It is possible to pose this as a 
bunch of nonconvex quadratic programming problems. For a few parameters, these quadratic programming 
problems are tractable and can be solved efficiently \cite{Breiman1993, Mengi2014}. In practice, 
we use \texttt{eigopt} \cite{Mengi2018b}, which is a Matlab implementation of this algorithm.

\begin{algorithm}
 \begin{algorithmic}[1]
\REQUIRE{ The matrix-valued function $A(x)$ of the form (\ref{eq:prob_defn}) and the feasible region $\widetilde{\Omega}$. }
\ENSURE{ The sequence $\{ x^{(k)} \}$. }
\STATE $x^{(1)} \gets$ a random point in $\widetilde{\Omega}$.
\STATE Calculate $[{\mathcal D}(A(x^{(1)}))]^2$ and $\nabla [{\mathcal D}(A(x^{(1)}))]^2$. \label{fun_eval}
\FOR{$k \; = \; 2, \; 3, \; \dots$}
	\STATE $x^{(k)} \gets \arg \max_{x \in \widetilde{\Omega}} \; \min \{ q(x; x^{(j)}) \; | \; j = 1,\dots, k-1 \}$. \label{subproblem}
	\STATE Calculate $[{\mathcal D}(A(x^{(k)}))]^2$ and $\nabla [{\mathcal D}(A(x^{(k)}))]^2$. \label{fun_eval2}
\ENDFOR
 \end{algorithmic}
\caption{Solution of Small-Scale Problems}
\label{small_alg}
\end{algorithm}

\subsection{General Case Without Uniform Stability}
Consider the problem at our hands with the full generality, that is consider
\[		
				\max_{x \in \widetilde{\Omega}}	\;	[{\mathcal D}(A(x)) ]^2
					\;\;	=	\;\;
				\max_{x \in \widetilde{\Omega}}	\: 
				\min_{\omega \in {\mathbb R}} \;
				\left\{ 
					\begin{array}{ll} 
				\lambda ( x,\omega )		\quad\quad	&	{\rm if} \; \Lambda(A(x)) \cap {\mathbb C}^+ = \emptyset	\\
				0							\quad\quad	&	{\rm otherwise}	
					 \end{array} \;\; .
				\right.
\]
The global property that the left-hand derivatives are greater than or equal to the right-hand derivatives indicated
in part (i) of Theorem \ref{thm:der_D} is lost at $x$ values where $A(x)$ has none of the eigenvalues on the open
right-half of the complex plane, but one or more eigenvalues on the imaginary axis
(for an illustrative example see the top left of Figure \ref{fig:small_random}; see, in particular, the distance function 
over there depending on one parameter near 0.4, where the distance becomes zero). A consequence is that the quadratic 
function $q(x; \widetilde{x})$ constructed about $\widetilde{x}$ such that $A(\widetilde{x})$ is not asymptotically stable is 
not necessarily an upper support function, that is $[{\mathcal D}(A(x))]^2 \leq q(x; \widetilde{x})$ does not necessarily 
hold for all $x \in \widetilde{\Omega}$. However, such points where asymptotic stability is lost
are far away from global maximizers that we are seeking. Indeed, if $\gamma$ is chosen large enough, 
the quadratic functions about these points still bound the distance to instability function from above locally in a 
neighborhood of each global maximizer. This correct representation around the 
global maximizers is sufficient for the convergence of the algorithm to the globally maximal value. Formally, it can be 
shown that there exists $\gamma$ such that every convergent subsequence of the sequence $\{ x^{(k)} \}$ 
generated by Algorithm \ref{small_alg} converges to a global maximizer of  $[ {\mathcal D}(A(x)) ]^2$ over $x \in \widetilde{\Omega}$.
In the affine case, when $A(x) = B_0 + \sum_{j=1}^\kappa x_j B_j$, the choice of $\gamma$ set equal to the largest
eigenvalue of the matrix in (\ref{eq:matrix_forgamma}) works well in practice in our experience.

\subsection{Numerical Results}
We present numerical results on three sets of 
examples\footnote{Available at \url{http://home.ku.edu.tr/~emengi/software/max_di/Data_&_Updates.html}} 
all of which concern the stabilization by output feedback
control problem. These results are obtained by applications of a Matlab implementation of Algorithm \ref{small_alg}
that is publicly available \cite{Mengi2018}. 

The first set involves a $4\times 4$ random example with
\begin{equation}\label{eq:random_matrix}
	\begin{split}
	A
		& =
	\left[
		\begin{array}{rrrr}
		0.1377   &	 0.3188  &  3.5784  &  0.7254	\\
		1.8339  &	 -1.7077 &  2.7694  & -0.0631	\\
		-2.2588 &   -0.4336 &  -1.7499 &   0.7147	\\
		0.8622 &   0.3426 &   3.0349  & -0.6050	
		\end{array}
	\right],	\hskip 11ex \\
	\quad
	B
		& =
	\left[
		\begin{array}{rr}
			-0.1241	&	0.4889	\\
			1.4897	&	1.0347	\\
			1.4090	&	0.7269	\\
			1.4172	&	-0.3034	\\
		\end{array}
	\right],
	\quad
	C
		=
	\left[
		\begin{array}{crcr}
			0.6715	&	-1.2075	&	0.7172	&	1.6302		\\
			0.2939	&	-0.7873	&	0.8884	&	-1.1471		\\
		\end{array}
	\right].
	\end{split}
\end{equation}
Denoting the $j$th columns of $B, C^T$ with $b_j, c_j$,
we first maximize ${\mathcal D}(A + k b_1 c_1^T)$ over $k \in [-5, 5]$ using Algorithm \ref{small_alg}.
The matrix $A$ is stable, indeed ${\mathcal D}(A) = 0.2063$ and the rightmost eigenvalues of $A$
are $-0.3470 \pm 2.3375 {\rm i}$. On the other hand, the maximized distance is given by
${\mathcal D}(A + k^{(1)} b_1 c_1^T) = 0.8385$ for $k^{(1)} = -0.9025$, furthermore the rightmost 
eigenvalues of $A + k^{(1)} b_1 c_1^T$ are $-1.0664 \pm 3.3377 {\rm i}$. The distance to instability
${\mathcal D}(A + k b_1 c_1^T)$ is plotted as a function of $k$ on top left in Figure \ref{fig:small_random}.
The nonsmoothness of ${\mathcal D}(A + k b_1 c_1^T)$ at $k = k^{(1)}$ is evident in this figure, which is caused
by the fact that the distance to instability of $A + k^{(1)} b_1 c_1^T$ is attained at two distinct negative $\omega$ 
values\footnote{Note that, as the matrices are real, the singular value function 
$\sigma(\omega) := \sigma_{\min} (A + k b_1 c_1^T - \omega {\rm i} I)$ is symmetric with respect to the origin. 
This means that the minimizers of $\sigma(\omega)$ appear in plus, minus pairs. 
As discussed in Remark \ref{rem:nonsmooth_rcase}, this does not cause nonsmoothness. However, 
attainment of the minimum at two distinct negative (or positive) $\omega$ values causes nonsmoothness.}. 
The bottom left portion of Figure \ref{fig:small_random} provides a plot of 
$\sigma_{\min}(A + k^{(1)} b_1 c_1^T - \omega {\rm i} I)$ with respect to $\omega$, which confirms that the 
singular value function has two negative global minimizers.

Next we maximize ${\mathcal D}(A + k_1 b_1 c_1^T + k_2 b_2 c_2^T)$ over $k_1,k_2 \in [-5, 5]$.
The maximized distance ${\mathcal D}(A + k^{(2)}_1 b_1 c_1^T +  k^{(2)}_2 b_2 c_2^T) = 0.9654$
attained at $k^{(2)} = ( k^{(2)}_1, k^{(2)}_2 ) = (-1.4489,0.5353)$ is improved compared to the rank one case,
while the rightmost eigenvalues $-1.4150 \pm 3.9805{\rm i}$ of $A + k^{(2)}_1 b_1 c_1^T +  k^{(2)}_2 b_2 c_2^T$
are located further to the left. The contour diagram of ${\mathcal D}(A + k_1 b_1 c_1^T + k_2 b_2 c_2^T)$ over $(k_1, k_2)$ 
is depicted on top right in Figure \ref{fig:small_random}.
Once again the distance function is nonsmooth at the maximizer, since the 
smallest singular value of $A +  k^{(2)}_1 b_1 c_1^T + k^{(2)}_2 b_2 c_2^T - \omega {\rm i} I$ is minimized 
globally at two distinct negative $\omega$  values, which is shown at bottom right in Figure \ref{fig:small_random}.

The second set concerns the turbo-generator example in \cite[Appendix E]{Hung1982}, used as a test 
example also in \cite{Burke2003}. This example involves the robust stabilization of $A + BKC$ over 
$K \in {\mathbb R}^{2\times 2}$, where $A \in {\mathbb R}^{10\times 10}, B \in {\mathbb R}^{10\times 2}, C\in {\mathbb R}^{2\times 10}$.
We maximize 
\begin{enumerate}
	\item[\bf (i)] ${\mathcal D}(A + k b_2 c_2^T)$ over $k \in [-0.5, 0.1]$, and 
	\item[\bf (ii)] ${\mathcal D}(A + k_{1} b_2 c_2^T + k_{2} b_2 c_1^T)$
	over $k_{1}, k_{2} \in [-0.5,0.1]$. 
\end{enumerate}
The original matrix $A$ is stable with ${\mathcal D}(A) = 0.0077$, whereas ${\mathcal D}(A + k^{(1)} b_2 c_2^T) = 0.0430$
at the global maximizer $k^{(1)} = -0.3990$ of (i), and ${\mathcal D}(A + k^{(2)}_1 b_2 c_2^T + k^{(2)}_2 b_2 c_1^T) = 0.0722$
at the global maximizer $k^{(2)} = (k^{(2)}_1, k^{(2)}_2) = (-0.1847, -0.1644)$ of (ii). The rightmost eigenvalues of 
$A,  A + k^{(1)} b_2 c_2^T, A + k^{(2)}_1 b_2 c_2^T + k^{(2)}_2 b_2 c_1^T$ are located
at $-0.2345$, $-0.3583 \pm 6.6403$, $-0.5019$, respectively. The results are displayed in Figure \ref{fig:small_turbo}.
In the one parameter case, the figure on top left indicates the existence of two maximizers, both of which are nonsmooth, 
and only one of which is a global maximizer. The algorithm correctly converges to the global maximizer. 
The two parameter case is highly nonsmooth, indeed 
$\sigma_{\min}(A + k^{(2)}_1 b_2 c_2^T + k^{(2)}_2 b_2 c_1^T - \omega {\rm i} I)$
is minimized at three distinct nonpositive $\omega$ values, as shown at bottom right in Figure \ref{fig:small_turbo}. 
This is reflected into the contour diagram on top right as steep changes close to the global maximizer.

The nonsmoothness at the maximizer does not always occur. In Figure \ref{fig:small_random2} on the left,
for random $A \in {\mathbb R}^{200\times 200}$ and $b,c \in {\mathbb R}^{200}$, the plot of ${\mathcal D}(A + k b c^T)$
with respect to $k \in [-0.1, 0.1]$ is illustrated. The distance function is smooth at the computed maximizer
$k^{(1)} =  0.0144$. Indeed, as depicted on the right in Figure \ref{fig:small_random2}, the singular value function
$\sigma_{\min} (A + k^{(1)} b c^T - \omega {\rm i} I)$  attains its minimum at a unique negative $\omega$ value. 
Our numerical experiments indicate both the smooth 
and the nonsmooth maximizers are possible. Based on our numerical experiments, it is not possible to call 
one of these cases generic and the other non-generic.

In each of these examples, $\gamma$ is set equal to the largest eigenvalue of the matrix in 
(\ref{eq:matrix_forgamma}). The precise values are listed in Table \ref{tab:gamma}.

\begin{figure}
		\hskip -3ex
		\begin{tabular}{ll}
		\includegraphics[width=.49\textwidth]{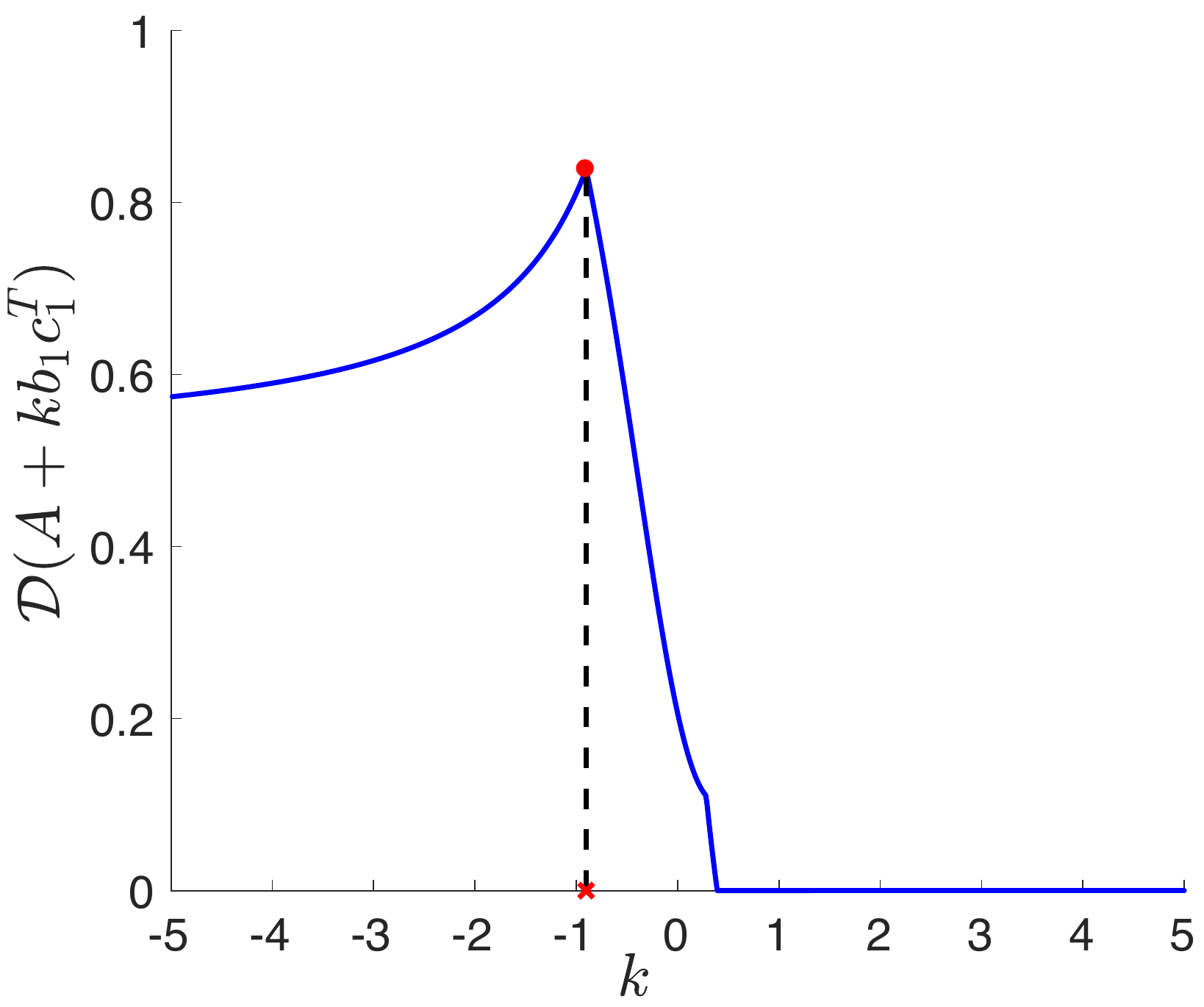} 	&
		\includegraphics[width=.49\textwidth]{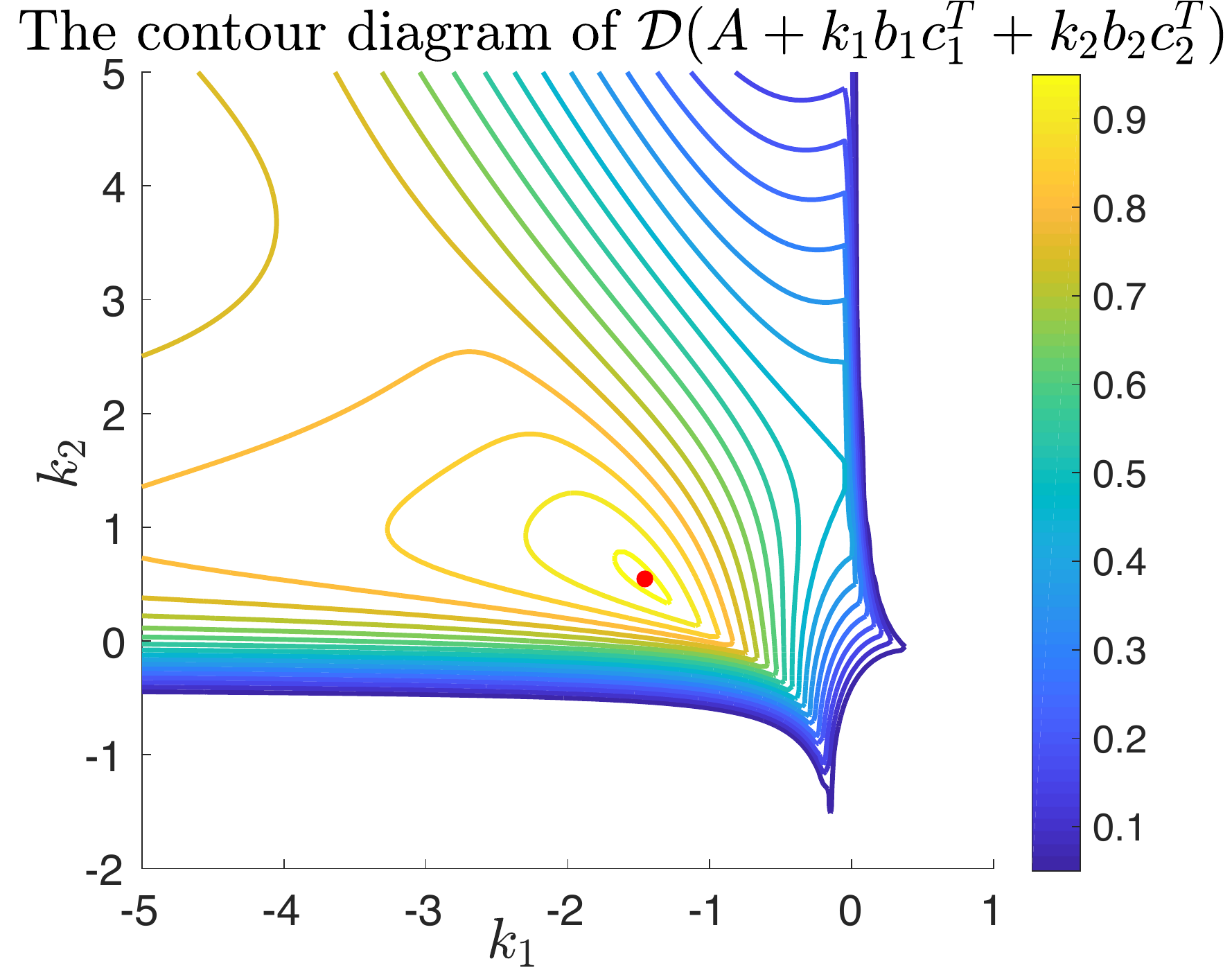} 	\\
		\includegraphics[width=.50\textwidth]{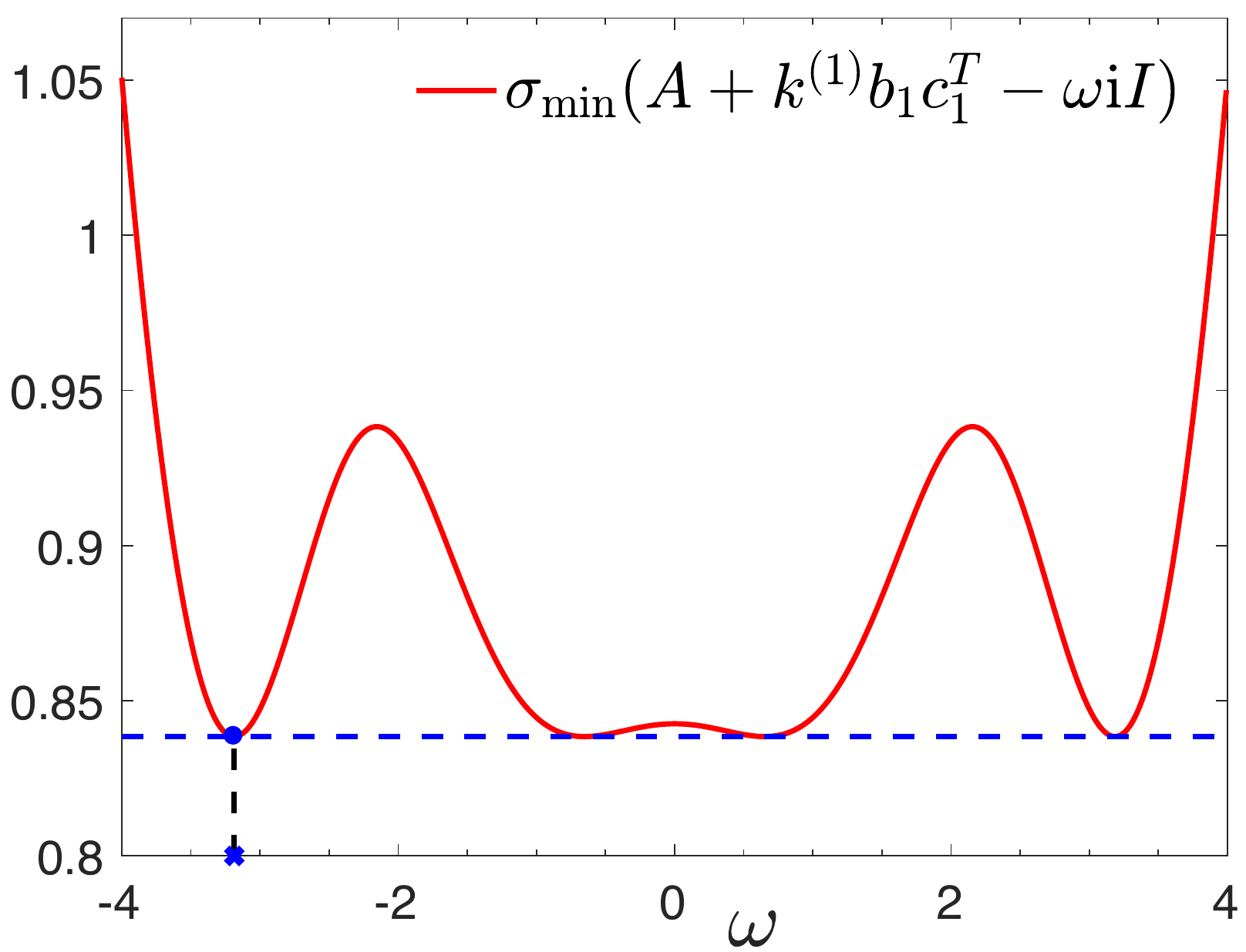} 	&
		\includegraphics[width=.47\textwidth]{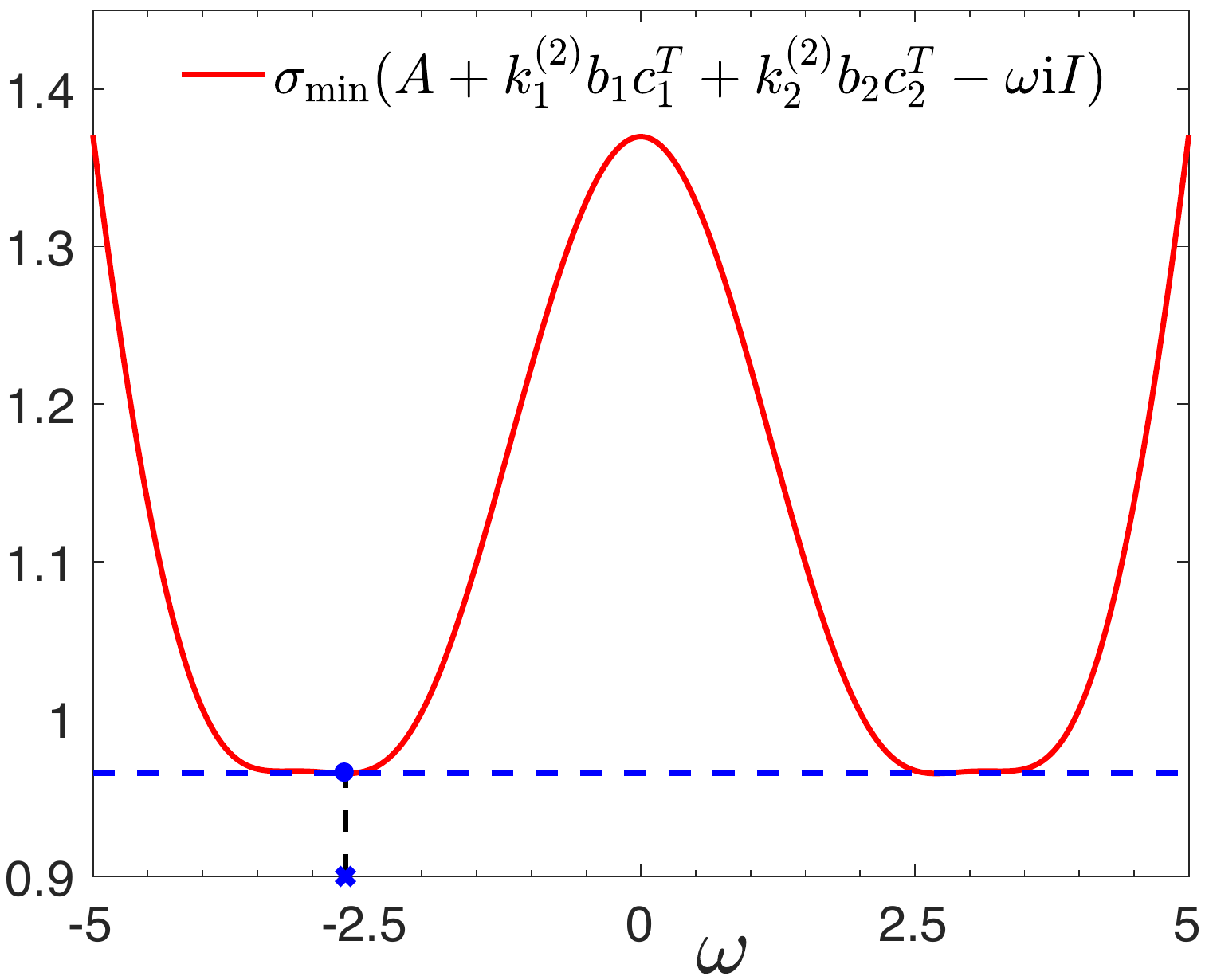}
		\end{tabular}
	  \caption{
	  		The plots concern the random example in (\ref{eq:random_matrix}).
			\textbf{\emph{(Top Left)}} The plot of ${\mathcal D}(A + k b_1 c_1^T)$ as a function of $k \in [-5,5]$.
			The cross at the horizontal axis marks the computed global
			maximizer $k^{(1)} =-0.9025$ by Algorithm \ref{small_alg}, whereas the dot marks 
			$(k^{(1)}, {\mathcal D}(A + k^{(1)} b_1 c_1^T))$.
			\textbf{\emph{(Bottom Left)}} The plot of the singular value $\sigma_{\min}(A + k^{(1)} b_1 c_1^T - \omega {\rm i} I)$
			as a function of $\omega$. The cross at the horizontal axis marks 
			$\omega_\ast = -3.1860$, one of the $\omega$ values minimizing $\sigma_{\min}(A + k^{(1)} b_1 c_1^T - \omega {\rm i} I)$ 
			globally, while the dot marks $(\omega_\ast, \sigma_{\min}(A + k^{(1)} b_1 c_1^T - \omega_\ast {\rm i} I))$.
			The dashed horizontal line is the highest horizontal line that bounds the graph of the singular value function from below.
			\textbf{\emph{(Top Right)}} The contour diagram of ${\mathcal D}(A + k_1 b_1 c_1^T + k_2 b_2 c_2^T)$ with respect to
			$k_1 \in [-5,5]$ in the horizontal axis and $k_2 \in [-5,5]$ in the vertical axis. 
			The global maximizer $( k^{(2)}_1, k^{(2)}_2 ) = (-1.4489,0.5353) $ is marked 
			with a dot.
			\textbf{\emph{(Bottom Right)}} The plot of the singular value 
			$\sigma_{\min}(A + k^{(2)}_1 b_1 c_1^T + k^{(2)}_2 b_2 c_2^T - \omega {\rm i} I)$
			as a function of $\omega$. The global minimizer marked with a cross in the horizontal 
			axis is $\omega_\ast = -2.6958$. The dot and the dashed line stand for the quantities
			same as in bottom left.
	   	     }
	     \label{fig:small_random}
\end{figure}

\begin{figure}
		\hskip -3ex
		\begin{tabular}{ll}
		\includegraphics[width=.49\textwidth]{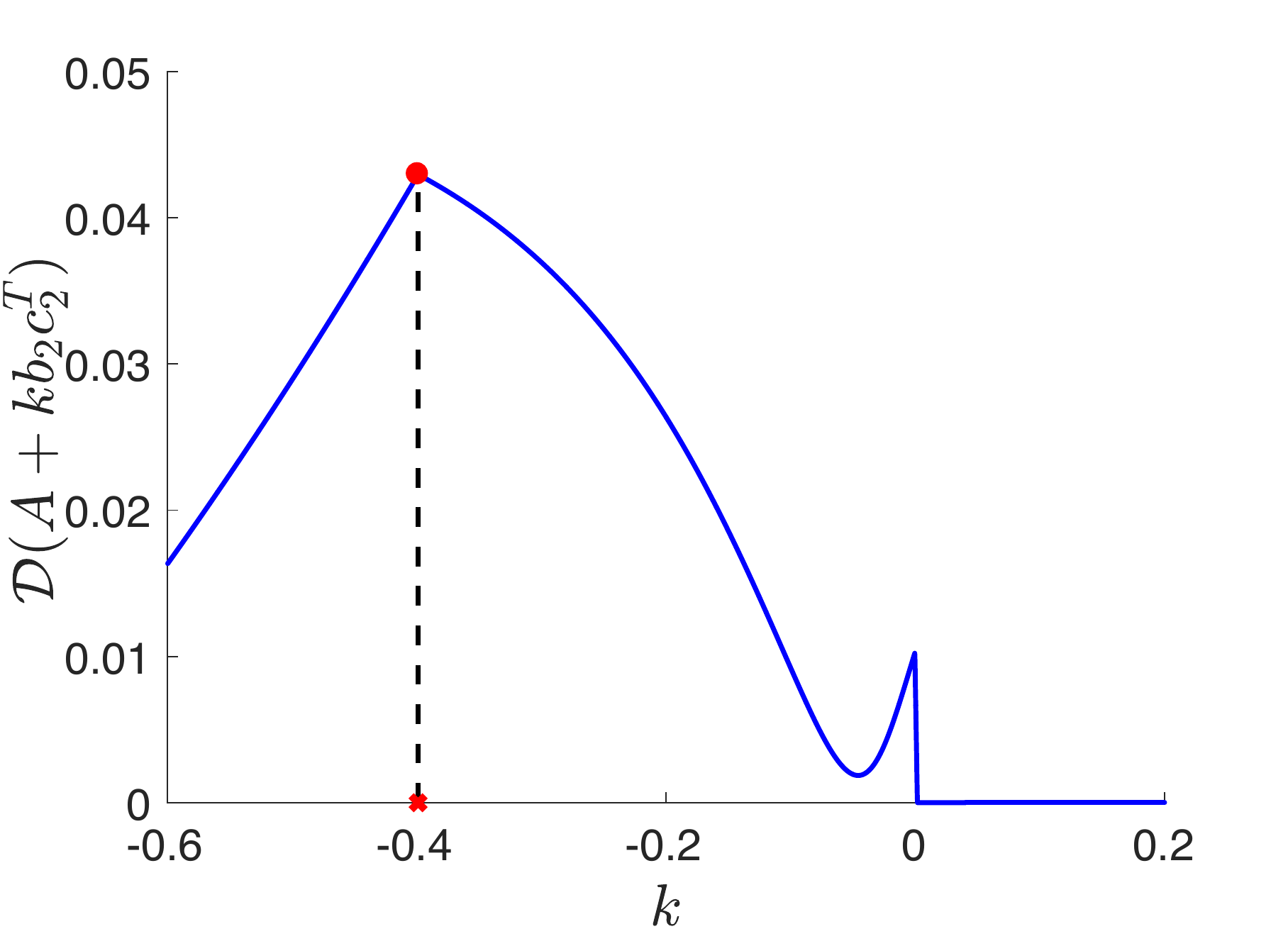} 	&
		\includegraphics[width=.49\textwidth]{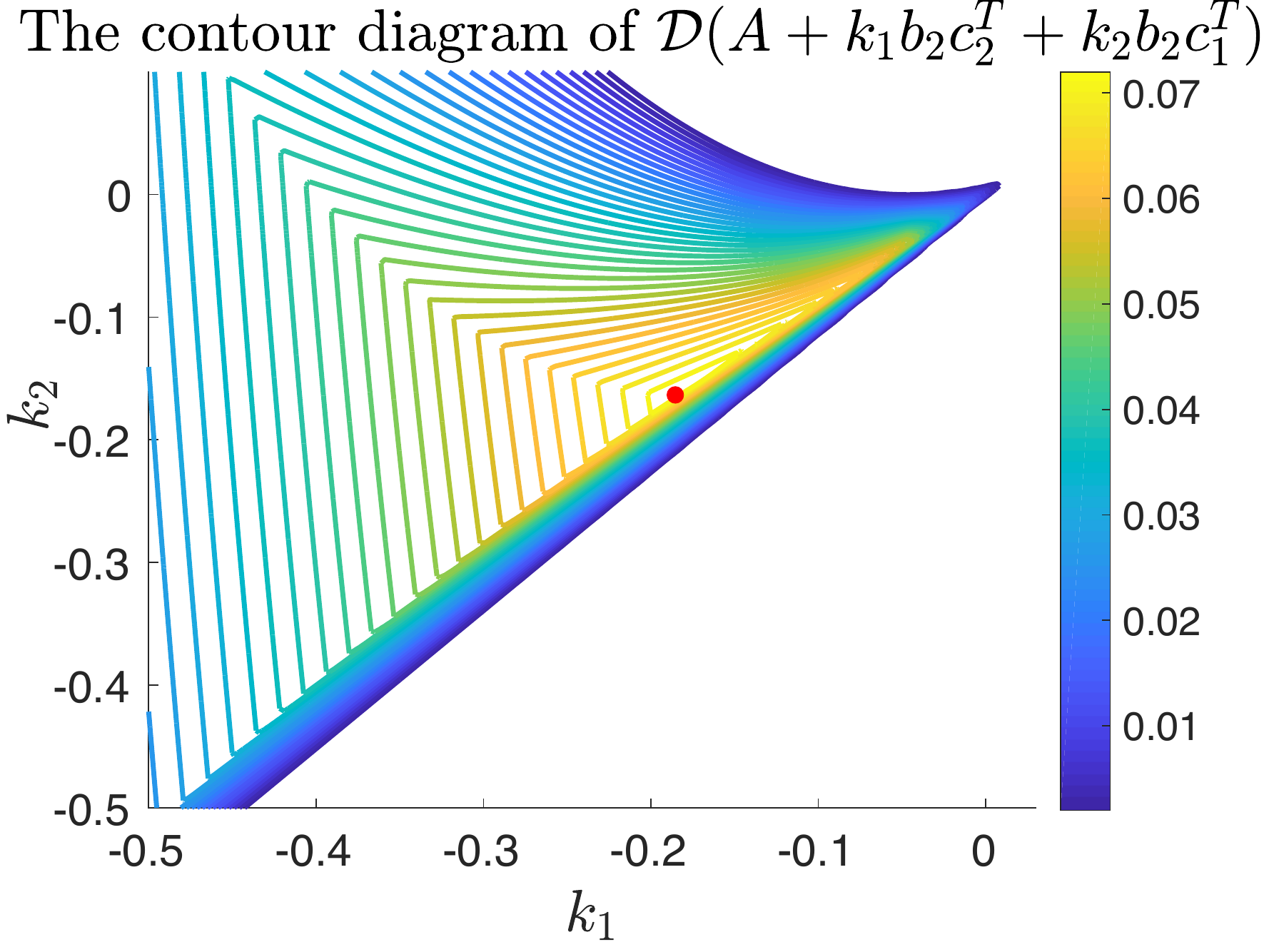} 	\\
		\includegraphics[width=.48\textwidth]{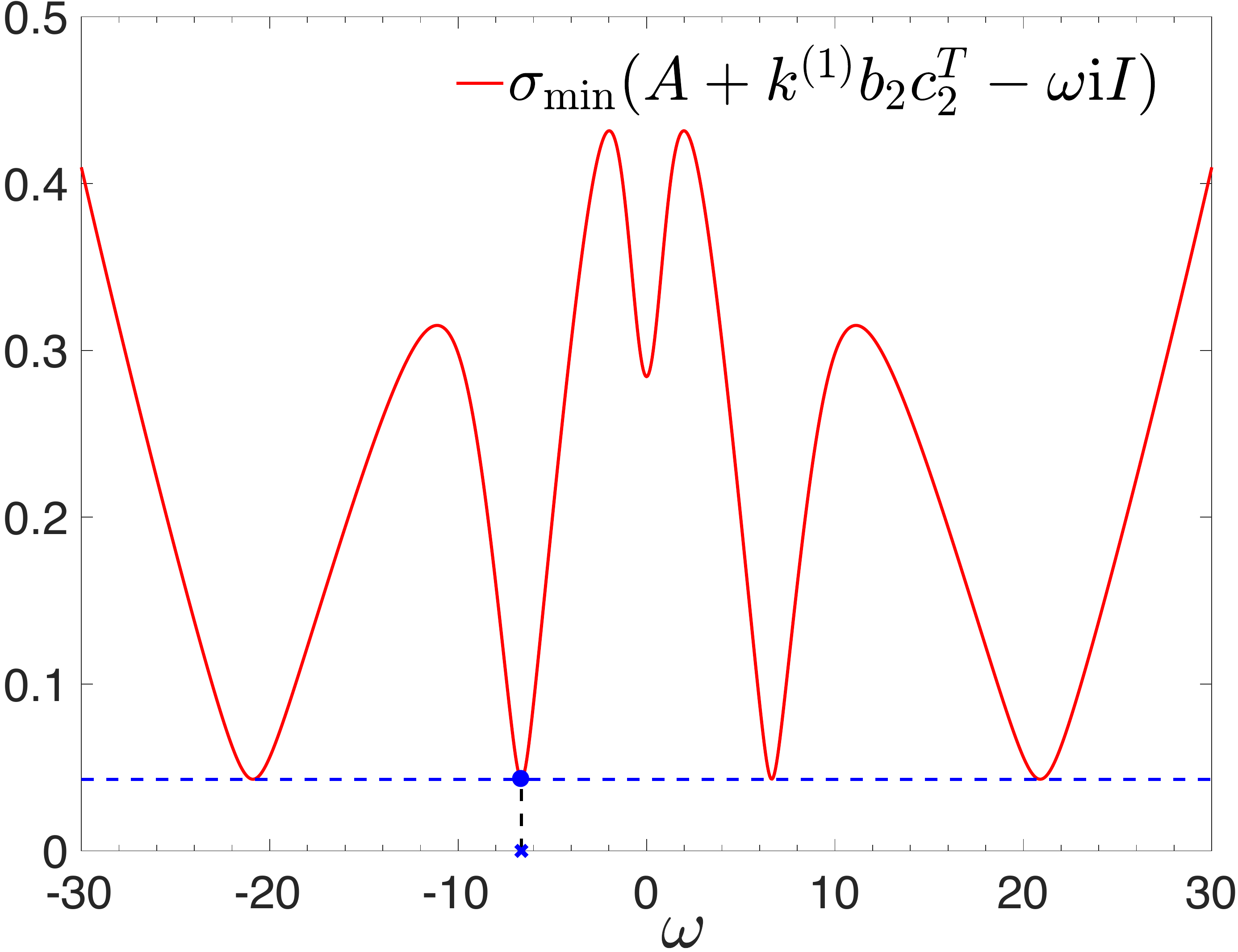} 	&
		\includegraphics[width=.49\textwidth]{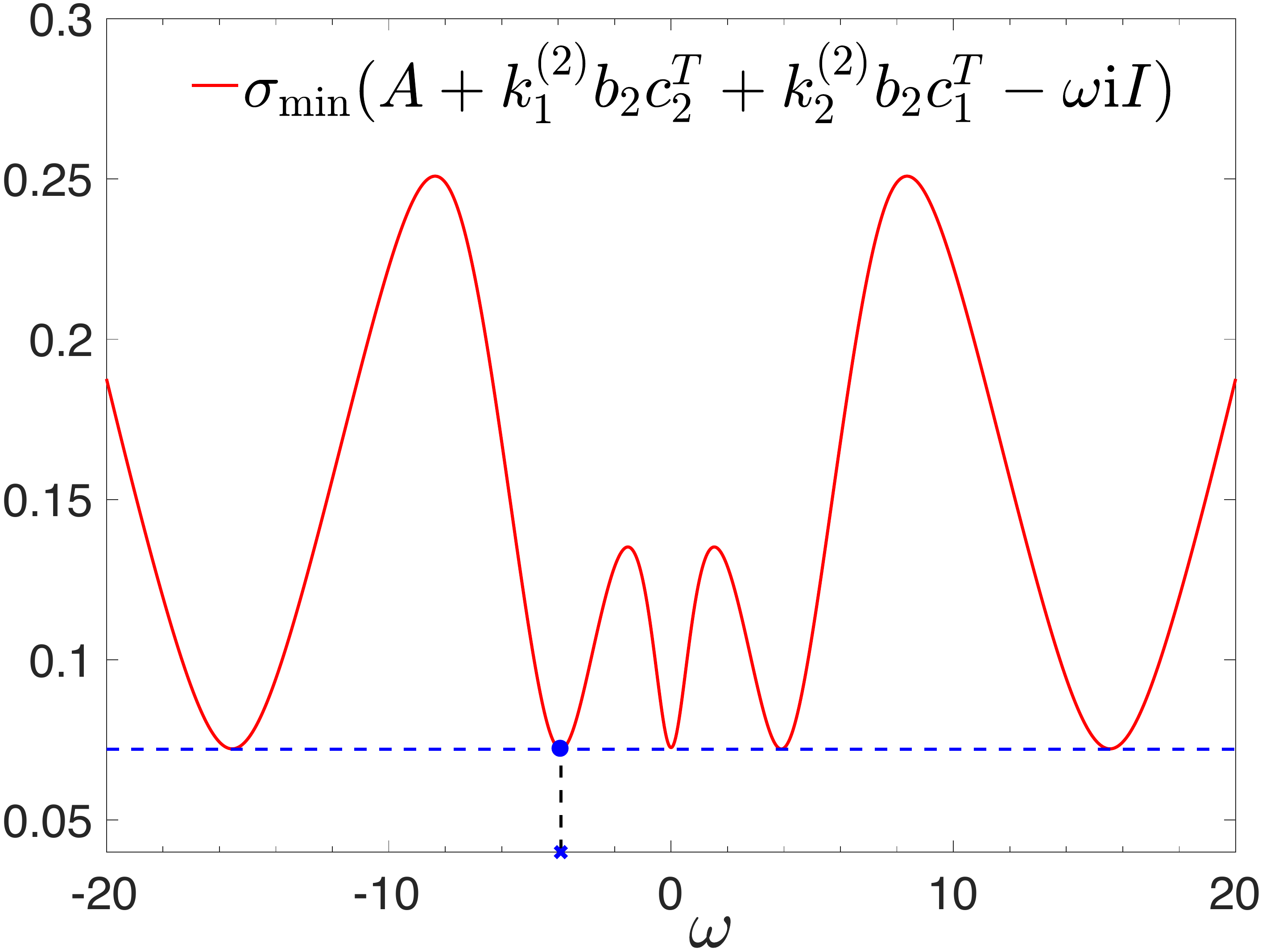}
		\end{tabular}
	  \caption{ 
	  		The plots are analogous to Figure \ref{fig:small_random}, but
	  		for the turbo-generator example \cite[Appendix E]{Hung1982}.
			Here, the one and two parameter cases concern the maximization of ${\mathcal D}(A + k b_2 c_2^T)$
			over $k \in [-0.5,0.1]$ and ${\mathcal D}(A + k_{1} b_2 c_2^T + k_{2} b_2 c_1^T)$ over 
			$k_{1}, k_{2} \in [-0.5,0.1]$, respectively. The maximizers for the top row are $k^{(1)} = -0.3990$ and 
			$( k^{(2)}_1, k^{(2)}_2 ) = (-0.1847, -0.1644)$, whereas the minimizers of the singular value 
			functions at the bottom row are $-6.6423$ and $-3.9057$ for the left-hand and right-hand plots, 
			respectively.
	   	     }
	     \label{fig:small_turbo}
\end{figure}

\begin{figure}
		\hskip -2ex
		\begin{tabular}{ll}
		\includegraphics[width=.49\textwidth]{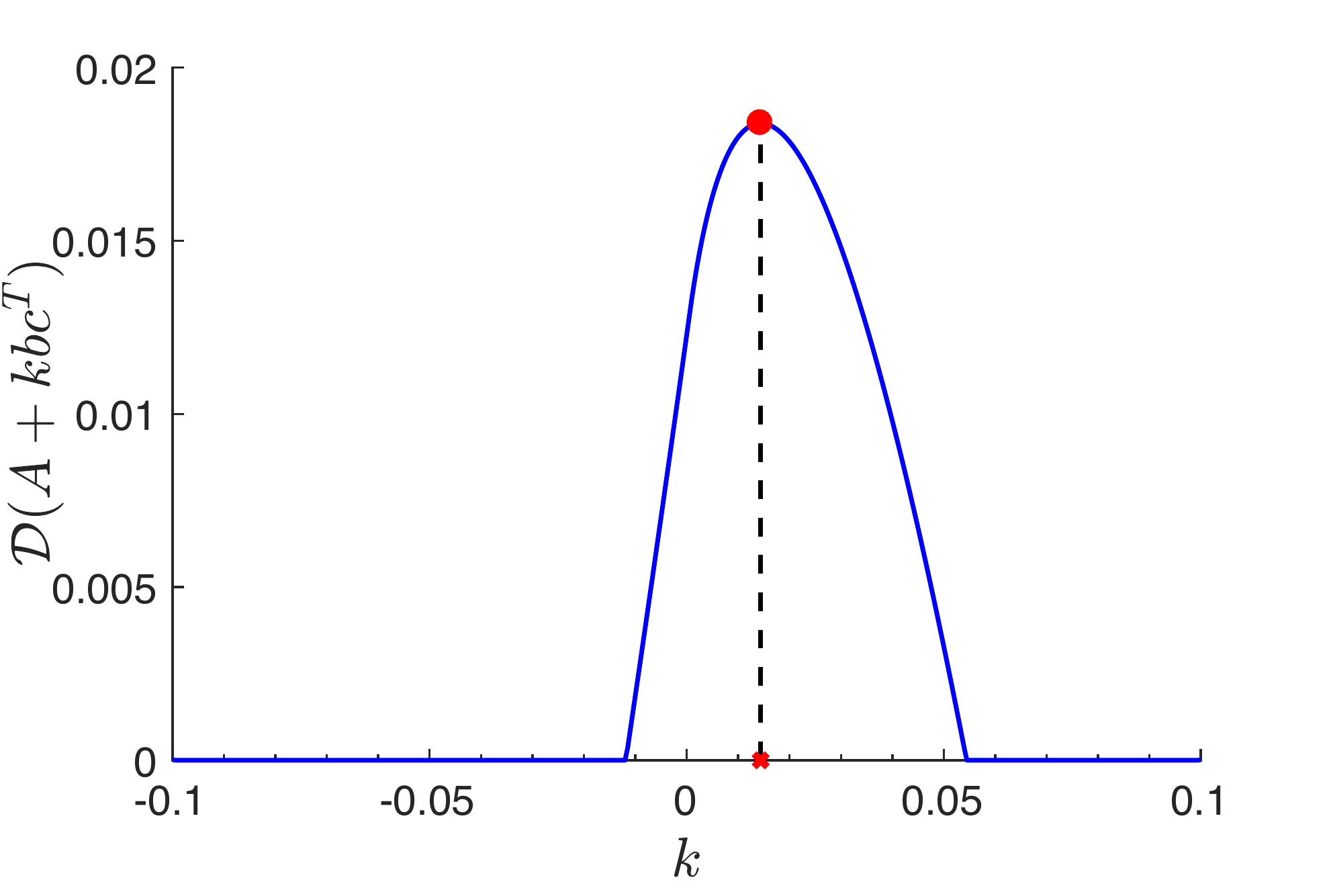} 	&
		\includegraphics[width=.49\textwidth]{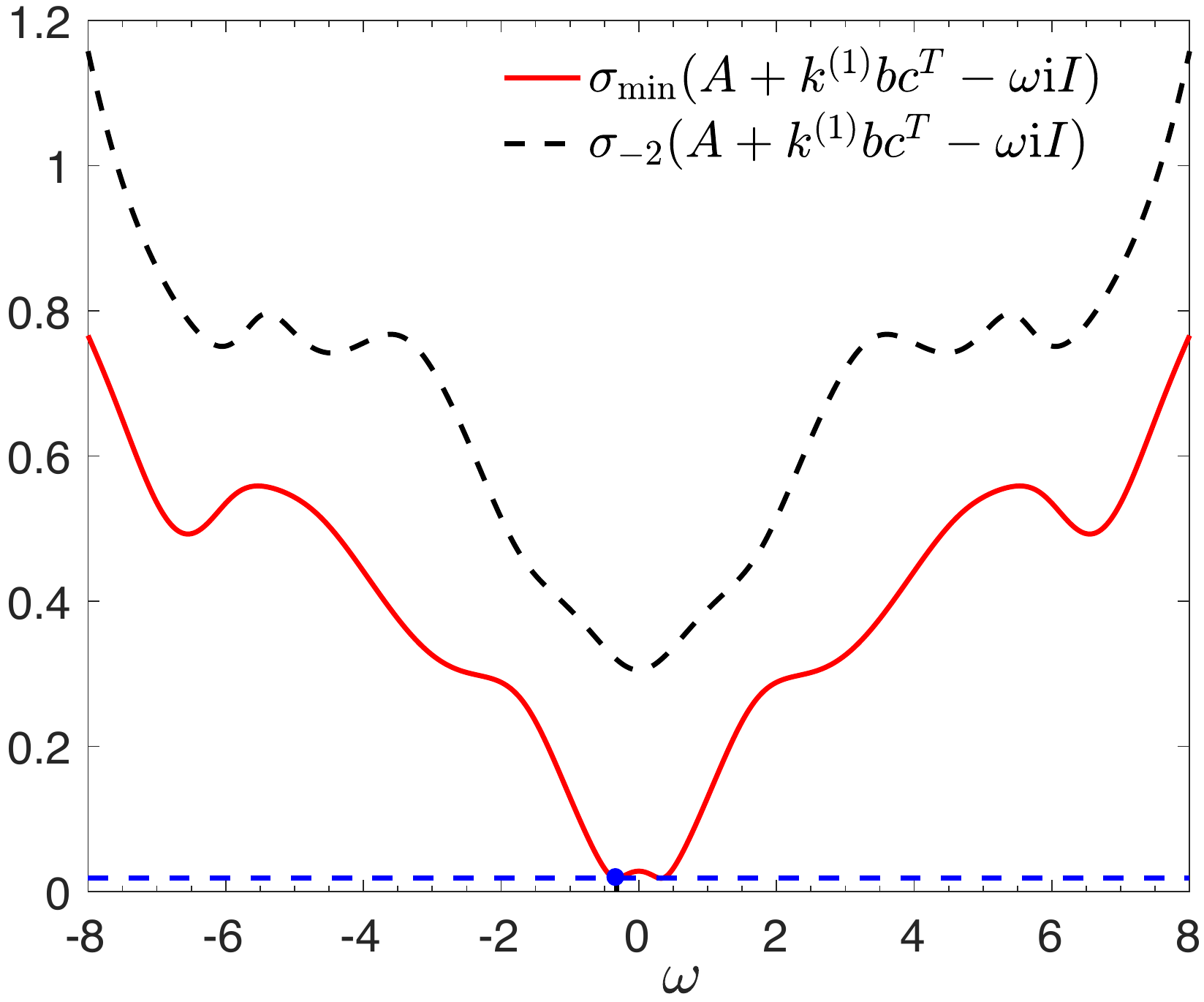} 	
		\end{tabular}
	  \caption{
	  		\textbf{\emph{(Left)}} The plot of the distance to instability ${\mathcal D}(A + k b c^T)$
			with respect to $k \in [-0.1,0.1]$ for a random $A \in {\mathbb R}^{200\times 200}$ and $b,c \in {\mathbb R}^{200}$.
			The distance function is smooth at the maximizer $k^{(1)} = 0.0144$ (marked with a cross on the
			horizontal axis), and attains the maximal value ${\mathcal D}(A + k^{(1)} b c^T) = 0.0184$. 
			\textbf{\emph{(Right)}} The plots of the smallest singular value function
			$\sigma_{\min}(A + k^{(1)} b c^T - \omega {\rm i} I)$ (solid curve)
			and the second smallest singular vale function
			$\sigma_{-2}(A + k^{(1)} b c^T - \omega {\rm i} I)$ (dashed curve) with respect to $\omega$.
			The point $\omega_\ast = -0.3231$ is the unique negative global minimizer of 
			the smallest singular value function, where the smallest singular value takes the value 0.0184.
	   	     }
	     \label{fig:small_random2}
\end{figure}

\begin{table}
	\begin{center}
	\begin{tabular}{|c||cc|}
		\hline
										&	$d = 1$	&	$d = 2$	\\
		\hline
		\hline
		Random $4\times 4$					&	63		&	64		\\
		Turbo-Generator					&	137		&	291		\\
		Random $200\times 200$				& 	72836	&	-		\\
		\hline
	\end{tabular}
	\end{center}
	\caption{
		$\gamma$ values for the random $4\times 4$ example (i.e., the example in (\ref{eq:random_matrix})), the turbo-generator example,
		and the random $200\times 200$ example.
	}
	\label{tab:gamma}
\end{table}


\section{A Subspace Framework for Large-Scale Problems}\label{sec:large_prob}

The ideas here are inspired by \cite{Kangal2015}. But the eigenvalue optimization problems considered in that work
concerns minimization or maximization of the $J$th largest eigenvalue, whereas the problem here, namely (\ref{eq:prob_defn}),
is a maximin problem. At first look, this looks challenging, but it turns out that the ideas over there can be extended
to our setting. 

Our projected reduced problems will be of the form
\begin{equation}\label{eq:red_prob}
	\max_{x \in \Omega}	\: \min_{z \in {\mathbb C}^+} \; 
					\sigma_{\min} (A^V(x) - z V)
\end{equation}
where
\begin{equation}\label{eq:red_mat_funs}
	A^V(x) \;\; := \;\; f_1(x) A_1 V + f_2(x) A_2 V +	\dots		+	f_\kappa(x) A_\kappa V
\end{equation}
and $V \in {\mathbb C}^{n\times \ell}$ with $n \gg \ell$  has orthonormal columns.
Hence we resort to one-sided projections. Since the problem at our hand is non-Hermitian, it may be tempting to use
two-sided projections. Unfortunately, it is difficult to conclude with strong convergence properties for such a two-sided
framework and its practical reliability is also in doubt. This is basically due to the loss of a monotonicity property discussed 
in the next section with the two-sided projections. 

The subspace framework is presented formally in Algorithm \ref{alg} below, where we use the notation
\begin{equation}\label{eq:red_inner}
	{\mathcal D}^{\mathcal V}(A(x)) \;\; :=  \;\;	
				\min_{z \in {\mathbb C}^+} \; 
					\sigma_{\min} (A^V(x) - z V),
\end{equation}
${\mathcal V} := {\rm Col}(V)$ (i.e., ${\mathcal V}$ is the column space of $V$),
and ${\rm orth}(C)$ stands for a matrix whose columns form an orthonormal basis for the column
space of the matrix $C$. Note that when ${\mathcal V} = {\mathbb C}^n$ and $A(x)$ is asymptotically stable 
(i.e., when $\Lambda(A(x)) \cap {\mathbb C}^+ = \emptyset$), the minimization problem in (\ref{eq:red_inner}) 
is attained on the imaginary axis due to the maximum modulus principle and we have 
${\mathcal D}^{\mathcal V}(A(x)) = {\mathcal D}(A(x))$. But this attainment property on the imaginary axis is 
not necessarily true in the rectangular case when ${\mathcal V}$ is a strict subspace of ${\mathbb C}^n$. 

At every iteration, the subspace framework first solves a projected small-scale problem. 
For this purpose, we benefit from the ideas in the previous section if there are only a few parameters; 
some details are discussed in Section \ref{sec:sf_implement}. 
If the problem depends on more than a few parameters, then a viable choice for the solution of the small-scale problem 
is HIFOO \cite{Burke2006}. Throughout the rest, we assume that the small-scale problem is solved globally. 
At a maximizer of this small-scale problem, if the full problem is asymptotically stable, then we compute the distance to 
instability of the full problem, in particular we retrieve the $\omega$ value where this distance is attained. We expand the 
subspace with the addition of a right singular vector corresponding to the smallest singular value at the optimal values of 
$x$ and $\omega$. On the other hand, if the full problem is not asymptotically stable at the maximizer of the small-scale
problem, then the full problem at the maximizer has an eigenvalue on the closed right-half of the complex plane.
The subspace is expanded with the inclusion of an eigenvector corresponding to such an eigenvalue.

\begin{algorithm}
 \begin{algorithmic}[1]
\REQUIRE{ The matrix-valued function $A(x)$ of the form (\ref{eq:prob_defn}) with the feasible region $\widetilde{\Omega}$.}
\ENSURE{ The sequences $\{x^{(\ell)} \}$, $\{ z^{(\ell)} \}$.}
\STATE $x^{(1)} \gets$ a random point in $\widetilde{\Omega}$.
\IF{$\Lambda(A(x^{(1)})) \cap {\mathbb C}^+ = \emptyset$}
	\STATE $z^{(1)} \gets {\rm i} \cdot \arg \min_{\omega \in {\mathbb R}} \; \sigma_{\min} (A(x^{(1)}) - \omega \ri I)$. \label{large_dinstab}
	\STATE $V_1	\;	\gets \;$ a unit right singular vector corresponding to $\sigma_{\min} (A(x^{(1)}) - z^{(1)} I)$. 
\ELSE
	\STATE $z^{(1)} \; \gets \;$ an eigenvalue in $\Lambda(A(x^{(1)})) \cap {\mathbb C}^+$.
	\STATE $V_{1}	\;	\gets \;$ a unit eigenvector corresponding to the eigenvalue $z^{(1)}$ of $A(x^{(1)})$.
\ENDIF
\STATE ${\mathcal V}_1 \; \gets \; {\rm span} \{ V_1 \}$.
\FOR{$\ell \; = \; 1, \; 2, \; \dots$}
	\STATE  $x^{(\ell+1)} \gets \arg \max_{x \in \widetilde{\Omega}}	\: {\mathcal D}^{{\mathcal V}_\ell}(A(x))$.
	\IF{$\Lambda(A(x^{(\ell+1)})) \cap {\mathbb C}^+ = \emptyset$}
		\STATE $z^{(\ell+1)} \gets {\rm i} \cdot \arg \min_{\omega \in {\mathbb R}} \; \sigma_{\min} (A(x^{(\ell+1)}) - \omega \ri I)$. \label{large_dinstab2}
		\STATE $v_{\ell+1}	\;	\gets \;$ a right singular vector corresponding to $\sigma_{\min} (A(x^{(\ell+1)}) - z^{(\ell+1)} I)$.
	\ELSE
		\STATE $z^{(\ell+1)} \; \gets \;$ an eigenvalue in $\Lambda(A(x^{(\ell+1)})) \cap {\mathbb C}^+$.
		\STATE $v_{\ell+1}	\;	\gets \;$ an eigenvector corresponding to the eigenvalue $z^{(\ell+1)}$ of $A(x^{(\ell+1)})$.
	\ENDIF
	\STATE $V_{\ell+1} 	\;	\gets	\; {\rm orth} 
		\left( \left[ \begin{array}{cc} V_{\ell} & v_{\ell+1} \end{array} \right] \right)$ 
		and ${\mathcal V}_{\ell+1} \; \gets \; {\rm Col} (V_{\ell+1})$.
\ENDFOR
 \end{algorithmic}
\caption{The Subspace Framework}
\label{alg}
\end{algorithm}

\begin{remark}
It may sound plausible to define the reduced problem (\ref{eq:red_prob}) so that the inner minimization problem 
is over the imaginary axis, that is ${\mathcal D}^{{\mathcal V}_\ell}(A(x))$ is defined by 
$\min_{\omega \in {\mathbb R}} \sigma_{\min} (A^{V_\ell}(x) - \ri \omega V_\ell)$.
Some difficulty arises with this approach if the full problem is unstable at the maximizer of the
reduced problem. As we shall see, Algorithm \ref{alg} is designed in a way so that interpolation
properties hold between the full and reduced problems at the maximizers of the reduced problems.
Without these interpolation properties, the essential features of the framework such as its global
and quick convergence are lost. 
It seems not possible to fulfill these interpolation properties at a maximizer of the reduced problem
where the full problem is unstable if the inner minimization problem is performed over the
imaginary axis. However, if the full problem $A(x)$ is uniformly stable at all $x \in \widetilde{\Omega}$,
then this idea of restricting the inner minimization to the imaginary axis works well; this is
discussed in more detail in Section \ref{sec:sf_uniform_stab}.
\end{remark}

\subsection{Theoretical Properties of the Subspace Framework}\label{sec:sfth}
We first present two results concerning the subspace framework that will play prominent roles
in the convergence analysis. The first one is a monotonicity result regarding ${\mathcal D}^{\mathcal V}(A(x))$
with respect to the subspace ${\mathcal V}$. We note once again that when ${\mathcal V} = {\mathbb C}^n$,
we have ${\mathcal D}^{\mathcal V}(A(x)) = {\mathcal D}(A(x))$.
\begin{theorem}[Monotonicity]\label{thm:mono}
For two subspaces ${\mathcal V}, {\mathcal W}$ of ${\mathbb C}^n$ such that ${\mathcal V} \subseteq {\mathcal W}$,
we have
\[
		\sigma_{\min}(A(x) - zI)	\;	\leq	\;	\sigma_{\min}(A^W(x) - zW)	\;	\leq	\;
													\sigma_{\min}(A^V(x) - zV)
													\;\;	\forall z \in {\mathbb C}^+, \: \forall x \in \Omega
\]
and
\[
				{\mathcal D}(A(x))	\;\;	\leq	\;\;	{\mathcal D}^{\mathcal W} (A(x))	\;\;	\leq	\;\;	{\mathcal D}^{\mathcal V}(A(x))	\quad	\forall x \in \Omega,
\]
where $V, W$ are matrices with columns forming orthonormal bases for ${\mathcal V}, {\mathcal W}$.
\end{theorem}
\begin{proof}
Recall that ${\mathcal D}^{\mathcal V}(A(x)) = \min_{z \in {\mathbb C}^+} \sigma_{\min} (A^V(x) - z V)$
and ${\mathcal D}^{\mathcal W}(A(x)) = \min_{z \in {\mathbb C}^+} \sigma_{\min} (A^W(x) - z W)$.
Additionally, since ${\mathcal W} \supseteq {\mathcal V}$, we have
\begin{equation*}
	\begin{split}
	\sigma_{\min}(A^W(x) - z W)	& \;\;	=	\;\;	\min \left\{  \| (A(x) - z I)w \|_2 \; | \; w \in {\mathcal W}, \: \| w \|_2 =1	\right\}	\\
			& \;\; \leq	\;\;	\min \left\{  \| (A(x) - z I)w \|_2 \; | \; w \in {\mathcal V}, \: \| w \|_2 = 1	\right\}	\\	
			& \;\; = \;\; \sigma_{\min}(A^V(x) - z V)
	\end{split}		
\end{equation*}
for all $z \in {\mathbb C}^+$. These assertions in turn imply ${\mathcal D}^{\mathcal W}(A(x)) \leq {\mathcal D}^{\mathcal V}(A(x))$.
In a similar way, since ${\mathbb C}^n \supseteq {\mathcal W}$, we have 
$\sigma_{\min}(A(x) - zI) \leq \sigma_{\min}(A^W(x) - z W)$ for all $z \in {\mathbb C}^+$ and
${\mathcal D}(A(x)) \leq {\mathcal D}^{\mathcal W}(A(x))$.
\end{proof}

The next result concerns Algorithm \ref{alg}. It establishes interpolation properties between 
${\mathcal D}(A(x))$ and ${\mathcal D}^{{\mathcal V}_\ell}(A(x))$. These interpolation properties also 
extend to first derivatives of these distance functions, but we elaborate on that issue later when 
we analyze the rate-of-convergence.
\begin{theorem}[Interpolation]\label{thm:interpolate}
The following hold regarding Algorithm \ref{alg} for a given $\ell \in {\mathbb Z}^+$ and $j = 1,\dots,\ell$:
\begin{equation}\label{eq:interpolation}
	\begin{split}
	{\mathcal D}(A(x^{(j)}))	& \;\;	=	\;\;	\sigma_{\min}(A(x^{(j)}) - z^{(j)} I)  \\
						& \;\; = \;\; \sigma_{\min}(A^{V_\ell}(x^{(j)}) - z^{(j)} V_\ell)	
								 \;\; = \;\; {\mathcal D}^{{\mathcal V}_\ell}(A(x^{(j)})).
	\end{split}
\end{equation}
In particular, if $A(x^{(j)})$ is asymptotically stable, then the minimum of $\sigma_{\min}(A^{V_\ell}(x^{(j)}) - z V_\ell)$ 
over all $z \in {\mathbb C}^+$ is attained on the imaginary axis.
\end{theorem}
\begin{proof}
By Theorem \ref{thm:mono}, we have 
$\sigma_{\min}(A(x^{(j)}) - z^{(j)}I)	\;	\leq	\;	\sigma_{\min}(A^{V_\ell}(x^{(j)}) - z^{(j)}V_\ell)$ and
${\mathcal D}(A(x^{(j)})) \; \leq \; {\mathcal D}^{{\mathcal V}_\ell}(A(x^{(j)}))$.
Additionally, since $v_j \in {\mathcal V}_\ell$, there exists $\alpha_j$ of unit length such that $v_j = V_\ell \alpha_j$,
which implies
\begin{equation}\label{eq:interpolation_ineqs}
	\begin{split}
		{\mathcal D}(A(x^{(j)})) & \;\;	=  \;\;		\sigma_{\min}(A(x^{(j)}) - z^{(j)} I)	\\
					& \;\; = \;\;	\| (A(x^{(j)}) -  z^{(j)} I) v_j \|_2	\;\; = \;\; 	\| (A^{V_\ell}(x^{(j)}) -  z^{(j)} V_\ell) \alpha_j \|_2	\\
							& \;\;	\geq	 \;\;	\sigma_{\min}(A^{V_\ell}(x^{(j)}) - z^{(j)} V_\ell)	
								\;\; \geq \;\; 	{\mathcal D}^{{\mathcal V}_\ell} ( A(x^{(j)}) ),		
	\end{split}
\end{equation}
regardless of whether $A(x^{(j)})$ is asymptotically stable or not.
Specifically, when $A(x^{(j)})$ is not asymptotically stable, then ${\mathcal D}(A(x^{(j)})) = 0$ by definition,
and all other terms in (\ref{eq:interpolation_ineqs}) do also vanish. This proves (\ref{eq:interpolation}).

If $A(x^{(j)})$ is asymptotically stable, then $z^{(j)}$ as defined in Algorithm \ref{alg} is purely imaginary
and (\ref{eq:interpolation_ineqs}) indicates that
\[
	{\mathcal D}^{{\mathcal V}_\ell} ( A(x^{(j)}) ) 
			\;\; = \;\;	\min_{z \in {\mathbb C}^+}
							\sigma_{\min}(A^{V_\ell}(x^{(j)}) - z V_\ell)
			\;\; = \;\; \sigma_{\min}(A^{V_\ell}(x^{(j)}) - z^{(j)} V_\ell).
\]
\end{proof}

\section{Convergence Analysis of the Subspace Framework}\label{sec:convergence}
We carry out the convergence analysis in the infinite dimensional setting. Formally, this means that 
$A_1, \dots, A_\kappa$ now become linear bounded operators acting on $\ell_2({\mathbb N})$,
the Hilbert space of square summable complex infinite sequences with the inner product
$\langle v, w \rangle := \sum_{j=0}^\infty \overline{v}_j w_j$ and the norm $\| v \| := \sqrt{ \sum_{j=0}^\infty | v_j |^2 }$.
Moreover some of the arguments below refer to the operator norms defined by
\begin{equation*}
	\begin{split}
	& \| A_j \|   \; := \;  
		\sup \{ \| A_j v \| \; | \; v \in \ell_2({\mathbb N}) \text{ such that } \| v \| = 1	\}	\;\; {\rm for} \; j = 1,\dots, \kappa	 \\
	& \| A(x) \|  \; := \; 
		\sup \{ \| A(x) v \| \; | \; v \in \ell_2({\mathbb N}) \text{ such that } \| v \| = 1	\}.
	\end{split}
\end{equation*}
We assume $A(x)$ for all $x$ has countably many singular values 
each with finite multiplicity, and 0 is not an accumulation point of 
these singular values. Compact operators have countably many singular values, but their
singular values accumulate at 0, indeed this is the sole accumulation point of the singular values. 
Hence, for instance, the singular values of $A(x) = \widetilde{A}(x) + \tau I$ with $\widetilde{A}(x)$ 
denoting any compact operator and $\tau$ any nonzero scalar satisfy the assumptions; the singular
values are countably many, do not accumulate at 0, additionally each singular value has finite multiplicity.

Throughout this section, we use the notations
\[
	\sigma^{\mathcal V}(x,z)	:=	\sigma_{\min}(A^V(x) - z V)
		\quad	{\rm and}	\quad
	\sigma(x,z)	:=	\sigma_{\min}(A(x) - z I)
\]
for $V$ whose columns form an orthonormal basis for the subspace ${\mathcal V}$, and
for $z \in {\mathbb C}^+$, $x \in \Omega$ so that 
${\mathcal D}^{\mathcal V}(A(x)) = \min_{z \in {\mathbb C}^+} \; \sigma^{\mathcal V}(x,z)$.

\subsection{Global Convergence}
The main result of this subsection, that is Theorem \ref{thm:global_conv}, establishes 
the global convergence of Algorithm \ref{alg}, which turns out to be a consequence of the monotonicity and the interpolation 
properties of the previous section\footnote{The monotonicity and interpolation results,
Theorems \ref{thm:mono} and \ref{thm:interpolate}, hold in the infinite dimensional setting
under consideration over the Hilbert space $\ell_2({\mathbb N})$. The arguments in the
convergence analyses make use of infinite dimensional versions of these theorem.},
as well as a Lipschitz continuity property. The latter is formally stated and 
proven below, where
\[
	\delta^{\mathcal V}_{\mathcal Z}(A(x)) \; := \; \inf \{ \sigma^{\mathcal V}(x,z) \; | \; z \in {\mathcal Z} \}
\]
for a given set ${\mathcal Z} \subseteq {\mathbb C}^+$.
\begin{lemma}[Uniform Lipschitz Continuity]\label{thm:Lip_cont}
There exists a constant $\zeta$ such that for all subspaces ${\mathcal V}$
the following hold:
\begin{enumerate}
\item[\bf (i)] For all $z \in {\mathbb C}^+$,
$ \;
| \sigma^{\mathcal V}(\widetilde{x},z) - \sigma^{\mathcal V}(\widehat{x},z) |	\; \leq \; \zeta \| \widetilde{x} - \widehat{x} \|_2
	\quad	\forall \widetilde{x}, \widehat{x} \in \Omega.
$
\item[\bf (ii)] For any given set ${\mathcal Z} \subseteq {\mathbb C}^+$,
$
	| \delta^{\mathcal V}_{\mathcal Z}(A(\widetilde{x})) - \delta^{\mathcal V}_{\mathcal Z}(A(\widehat{x})) |	\; \leq \; \zeta \| \widetilde{x} - \widehat{x} \|_2
	\quad	\forall \widetilde{x}, \widehat{x} \in \Omega.
$
\item[\bf (iii)] 
$
	| {\mathcal D}^{\mathcal V}(A(\widetilde{x})) - {\mathcal D}^{\mathcal V}(A(\widehat{x})) |	\; \leq \; \zeta \| \widetilde{x} - \widehat{x} \|_2
	\quad	\forall \widetilde{x}, \widehat{x} \in \Omega.
$
\end{enumerate}
\end{lemma}
\begin{proof} \textbf{(i)}
Let $V$ be such that its columns form an othonormal basis for ${\mathcal V}$.
Weyl's theorem \cite[Theorem 4.3.1]{Horn1985} implies
\begin{equation*}
	\begin{split}
		| \sigma^{\mathcal V}(\widetilde{x},z) - \sigma^{\mathcal V}(\widehat{x},z) |	& \leq \; 
									\| A^V(\widetilde{x})	-	A^V(\widehat{x}) \|	\\
																	& \leq \;
			\sum_{j=1}^\kappa	| f_j(\widetilde{x})  -  f_j(\widehat{x}) |   \| A_j \| 	
	\end{split}				
\end{equation*}
for all $\widetilde{x}, \widehat{x} \in \Omega$. But since each $f_j$ is real analytic, it is also Lipschitz continuous amounting to
the existence of constants $\gamma_j$ such that $| f_j(\widetilde{x})  -  f_j(\widehat{x}) | \leq \gamma_j \| \widetilde{x} - \widehat{x} \|_2$.
It follows that
\[
	| \sigma^{\mathcal V}(\widetilde{x},z) - \sigma^{\mathcal V}(\widehat{x},z) |
		\;	\leq	\;
	\left( \sum_{j=1}^\kappa	\gamma_j   \| A_j \| 	\right) \| \widetilde{x} - \widehat{x} \|_2
	\quad	\forall \widetilde{x}, \widehat{x} \in \Omega
\]
as desired.

\textbf{(ii)} 
Without loss of generality, we assume that the set ${\mathcal Z}$ is closed, as otherwise the argument below
applies by replacing ${\mathcal Z}$ with ${\rm Cl}({\mathcal Z})$, that is the closure of ${\mathcal Z}$, since 
$\delta^{\mathcal V}_{\mathcal Z}(A(x)) = \delta^{\mathcal V}_{{\rm Cl}({\mathcal Z})}(A(x))$. It turns out that, 
due to the assumption that ${\mathcal Z}$ is closed and the fact $\sigma^{\mathcal V}(x,z) \rightarrow \infty$ as 
$|z| \rightarrow \infty$, the quantity $\delta^ {\mathcal V}_{\mathcal Z}(A(x))$ must be attained at some
$z \in {\mathcal Z}$.

Now consider any two points $\widetilde{x}, \widehat{x} \in \Omega$ for which we must have
\[
	\delta^{\mathcal V}_{\mathcal Z}(A(\widetilde{x}))
		\;\; = \;\;
	\min_{z \in {\mathcal Z}} \: \sigma^{\mathcal V}(\widetilde{x},z)
	\quad\quad	{\rm and}	\quad\quad
	\delta^{\mathcal V}_{\mathcal Z}(A(\widehat{x}))
		\;\; = \;\;
	\min_{z \in {\mathcal Z}} \: \sigma^{\mathcal V}(\widehat{x},z).
\]
Without loss of generality, assume 
$\delta^{\mathcal V}_{\mathcal Z}(A(\widetilde{x})) \geq \delta^{\mathcal V}_{\mathcal Z}(A(\widehat{x}))$.
Also, let $\widetilde{z}, \widehat{z}$ be such that 
\[
	\delta^{\mathcal V}_{\mathcal Z}(A(\widetilde{x}))		\; = \;		\sigma^{\mathcal V}(\widetilde{x},\widetilde{z})
				\quad\quad	{\rm and}	\quad\quad
	\delta^{\mathcal V}_{\mathcal Z}(A(\widehat{x}))		\; = \;		\sigma^{\mathcal V}(\widehat{x},\widehat{z}).
\]
Finally, observe
\begin{equation*}
	\begin{split}
		| \delta^{\mathcal V}_{\mathcal Z}(A(\widetilde{x}))	 - 	\delta^{\mathcal V}_{\mathcal Z}(A(\widehat{x})) |	
					& \; = \;
		| \sigma^{\mathcal V}(\widetilde{x},\widetilde{z}) - \sigma^{\mathcal V}(\widehat{x},\widehat{z}) |	\\
					& \; \leq \;
		| \sigma^{\mathcal V}(\widetilde{x},\widehat{z}) - \sigma^{\mathcal V}(\widehat{x},\widehat{z}) |	\\
					& \; \leq \;
		\zeta \| \widetilde{x} - \widehat{x} \|_2,
	\end{split}
\end{equation*}
where the last inequality is due to part (i).

\textbf{(iii)} Observe that ${\mathcal D}^{\mathcal V}(A(x)) = \delta^{\mathcal V}_{\mathcal Z}(A(x))$ for
${\mathcal Z} = {\mathbb C}^+$, so this immediately follows from part (ii).
\end{proof}

Now we are ready for the global convergence result. We believe that an argument in support of global convergence in the 
finite dimensional setting based on the supposition that the subspace becomes the whole space eventually is not a proper 
argument. In \cite{Kangal2015}, it is observed that a subspace framework to maximize the largest eigenvalue stagnates 
at a local maximizer, that is not necessarily a global maximizer, after a few iterations.

\begin{theorem}[Global Convergence]\label{thm:global_conv}
Every convergent subsequence of the sequence $\{ x^{(\ell)} \}$ generated by Algorithm \ref{alg} in the infinite dimensional
setting converges to a global maximizer of ${\mathcal D}(A(x))$ over all $x \in \widetilde{\Omega}$. Furthermore,
\begin{equation}\label{eq:global_conv}
	\lim_{\ell\rightarrow \infty} \; {\mathcal D}^{{\mathcal V}_\ell}(A(x^{(\ell+1)}))
		\; = \;
	\lim_{\ell\rightarrow \infty} \; \max_{x \in \widetilde{\Omega}} \: {\mathcal D}^{{\mathcal V}_\ell} (A(x))
		\; = \;
	\max_{x \in \widetilde{\Omega}} \: {\mathcal D}(A(x)).
\end{equation}
\end{theorem}
\begin{proof}
To prove the claim that every convergent subsequence converges to a global maximizer of ${\mathcal D}(A(x))$, let us consider
a convergent subsequence $\{ x^{(j_\ell)} \}$ of $\{ x^{(\ell)} \}$. First observe that
\[
	\max_{x\in \widetilde{\Omega}} \: {\mathcal D}(A(x))	\;	\geq	\;	{\mathcal D}(A(x^{(j_\ell)}))	\; = \;		{\mathcal D}^{{\mathcal V}_{j_\ell}} (A(x^{(j_\ell)}))
\]
where the equality is due to the interpolation property (part (i) of Theorem \ref{thm:interpolate}), as well as
\[
	\max_{x\in \widetilde{\Omega}} \: {\mathcal D}(A(x))	\;	\leq	\;	\max_{x\in \widetilde{\Omega}} \: {\mathcal D}^{{\mathcal V}_{j_{\ell+1}-1}}(A(x))	
									\; = \;		{\mathcal D}^{{\mathcal V}_{j_{\ell+1}-1}}(A(x^{(j_{\ell+1})}))
									\; \leq \; {\mathcal D}^{{\mathcal V}_{j_\ell}}(A(x^{(j_{\ell+1})}))
\]
where both of the inequalities are due to the monotonicity property (Theorem \ref{thm:mono}). 
But by part (iii) of Lemma \ref{thm:Lip_cont}, for every subspace ${\mathcal V}$, we have
\[
		| {\mathcal D}^{\mathcal V}(A(\widetilde{x})) - {\mathcal D}^{\mathcal V}(A(\widehat{x})) |	\; \leq \; \zeta \| \widetilde{x} - \widehat{x} \|_2
		\quad	\forall \widetilde{x}, \widehat{x} \in \Omega,
\]
which implies
\[
	\lim_{\ell\rightarrow \infty} \left| {\mathcal D}^{{\mathcal V}_{j_\ell}}(A(x^{(j_{\ell+1})}))	-	{\mathcal D}^{{\mathcal V}_{j_\ell}} (A(x^{(j_\ell)})) \right|
			\;	\leq	\;
	\lim_{\ell\rightarrow \infty} \zeta \left\| x^{(j_{\ell+1})}  -  x^{(j_\ell)} \right\|_2
			\;	=	\;	0.
\]
Hence, employing the interpolation property once again,
\[
	\lim_{\ell \rightarrow \infty} {\mathcal D} (A(x^{(j_\ell)}))	\;	=	\;
	\lim_{\ell \rightarrow \infty} {\mathcal D}^{{\mathcal V}_{j_\ell}} (A(x^{(j_\ell)}))	\;	=	\;
	\max_{x\in \widetilde{\Omega}} \: {\mathcal D}(A(x)).
\]
It follows from the continuity of ${\mathcal D}(A(x))$ that the subsequence $\{ x^{(j_\ell)} \}$ converges to a point in 
$\arg \max_{x\in \widetilde{\Omega}} \: {\mathcal D}(A(x))$.

Finally, to deduce (\ref{eq:global_conv}), we proceed as in part (ii) of the proof of \cite[Theorem 3.1]{Kangal2015}.
Following similar arguments, the sequence $\{ {\mathcal D}^{{\mathcal V}_\ell}(A(x^{(\ell+1)})) \}$ can be shown to be monotonically
decreasing and bounded below by ${\mathcal D}_\ast := \max_{x\in \widetilde{\Omega}} {\mathcal D}(A(x))$, so it is
convergent. The proof is completed by constructing a subsequence of $\{ {\mathcal D}^{{\mathcal V}_\ell}(A(x^{(\ell+1)})) \}$
that converges to ${\mathcal D}_\ast$. In particular, for any convergent subsequence $\{ x^{(j_\ell)} \}$ of $\{ x^{(\ell)} \}$,
the sequence $\{ {\mathcal D}^{{\mathcal V}_{j_{\ell+1}-1}}(A(x^{(j_{\ell+1})})) \}$ is a subsequence of
$\{ {\mathcal D}^{{\mathcal V}_\ell}(A(x^{(\ell+1)})) \}$ and satisfies
\[
	 {\mathcal D}^{{\mathcal V}_{j_\ell}}(A(x^{(j_{\ell+1})})) 	\;	\geq	\;
	 {\mathcal D}^{{\mathcal V}_{j_{\ell+1}-1}}(A(x^{(j_{\ell+1})})) 	\;	\geq		\;	{\mathcal D}_\ast.
\]
Since we have
\[
	\lim_{\ell\rightarrow \infty} {\mathcal D}^{{\mathcal V}_{j_\ell}}(A(x^{(j_{\ell+1})}))	=	{\mathcal D}_\ast
\]
from the previous paragraph, $\{ {\mathcal D}^{{\mathcal V}_{j_{\ell+1}-1}}(A(x^{(j_{\ell+1})})) \}$ 
also converges to ${\mathcal D}_\ast$ as desired.
\end{proof}


\subsection{Local Rate-of-Convergence}
Now we assume that the sequence $\{ x^{(\ell)} \}$ itself is convergent. Theorem \ref{thm:global_conv} 
in this case ensures the convergence of the sequence $\{ x^{(\ell)} \}$ 
to a global maximizer $x_\ast$ of ${\mathcal D}(A(x))$ over $\widetilde{\Omega}$.
For instance, if ${\mathcal D}(A(x))$ has a unique global maximizer, due to assertion
(\ref{eq:global_conv}), the sequence $\{ x^{(\ell)} \}$ must converge to this unique global maximizer.
In this subsection we quantify the speed of this convergence. The main result establishes a local superlinear 
rate for the convergence of Algorithm \ref{alg} in the one parameter case (i.e., $d=1$) and of an extended version, 
namely Algorithm \ref{alge} below, in the multi-parameter case under the assumption that ${\mathcal D}(A(x))$ is smooth 
and its Hessian is invertible at the converged maximizer. 

Throughout this section $A(x)$ is assumed to be asymptotically stable at some $x \in \widetilde{\Omega}$. 
A consequence is that $A(x^{(\ell)})$ is asymptotically stable for all large $\ell$.

\subsubsection{Derivatives of Singular Value Functions}
We first present results that relate the
first and second derivatives of $\sigma(x,z)$ with those of $\sigma^{{\mathcal V}_\ell}(x,z)$.
From here on, the complex variable $z$ is written as $z = \alpha + {\rm i} \omega$
for $\alpha, \omega \in {\mathbb R}$. Hence, the notations $\sigma^{\mathcal V}_\alpha(x,z)$, 
$\sigma^{\mathcal V}_\omega(x,z)$ stand for differentiation of $\sigma^{\mathcal V}(\cdot)$ with respect 
to the real, imaginary parts of $z$.
\begin{theorem}[Hermite Interpolation of Singular Value Functions]\label{thm:Hermite_interpolate}
The following are satisfied by Algorithm \ref{alg} for all $\ell \in {\mathbb Z}^+$ and $j = 1,\dots,\ell$ 
such that $A(x^{(j)})$ is asymptotically stable:
\begin{enumerate}
\item[\bf (i)] If $u^{(j)}, v^{(j)}$ consist of a consistent pair of unit left, right singular vectors corresponding to $\sigma(x^{(j)},z^{(j)})$,
then $u^{(j)}, \vartheta^{(j)} := V_\ell^\ast v^{(j)}$ consist of a consistent pair of unit left, right singular vectors corresponding to
 $\sigma^{{\mathcal V}_\ell}(x^{(j)},z^{(j)})$.
\item[\bf (ii)] $\sigma_\omega (x^{(j)},z^{(j)}) = \sigma^{{\mathcal V}_\ell}_\omega (x^{(j)}, z^{(j)})$.
\item[\bf (iii)] If the singular value $\sigma(x^{(j)},z^{(j)})$ is simple, then 
$\sigma_\alpha (x^{(j)},z^{(j)}) = \sigma^{{\mathcal V}_\ell}_\alpha (x^{(j)}, z^{(j)})$.
\end{enumerate}
\end{theorem}
\begin{proof}
\textbf{(i)} Letting $\sigma := \sigma_{\min}(A(x^{(j)}) - z^{(j)} I)$ and $u^{(j)}, v^{(j)}$ be a corresponding
pair of consistent left, right singular vectors, this follows from the following line of reasoning:
\begin{equation*}
	\begin{split}
	\left( A(x^{(j)}) - z^{(j)} I \right) v^{(j)}  \; = \; \sigma u^{(j)}
	\quad {\rm and}	\quad
	(u^{(j)})^\ast \left( A(x^{(j)}) - z^{(j)} I \right) \; = \; \sigma (v^{(j)})^\ast
			\Longrightarrow		\hskip 5.5ex \\
	\left( A(x^{(j)}) - z^{(j)} I \right) V_\ell \vartheta^{(j)}  \; = \; \sigma u^{(j)}
		\quad {\rm and}	\quad
	(u^{(j)})^\ast \left( A(x^{(j)}) - z^{(j)} I \right) \; = \; \sigma (\vartheta^{(j)})^\ast V_\ell^\ast
			\Longrightarrow		\\
	\left( A^{V_\ell}(x^{(j)}) - z^{(j)} V_\ell \right) \vartheta^{(j)}  \; = \; \sigma u^{(j)}
		\quad {\rm and}	\quad
	(u^{(j)})^\ast \left( A^{V_\ell}(x^{(j)}) - z^{(j)} V_\ell \right) \; = \; \sigma (\vartheta^{(j)})^\ast. \hskip 3ex
	\end{split}
\end{equation*}
Hence, $\vartheta^{(j)}, u^{(j)}$ form a pair of unit right, left singular vectors of $A^{V_\ell}(x^{(j)}) - z^{(j)} V_\ell$
corresponding to $\sigma$.

\textbf{(ii)} We first remark that the singular value functions 
$\sigma(x,z)$ and $\sigma^{{\mathcal V}_\ell}(x,z)$ are differentiable at $(x,z) = (x^{(j)}, z^{(j)})$ with respect to the imaginary part of $z$.
This is due to the fact that both of the functions $\sigma(x^{(j)},\ri \omega)$ and $\sigma^{{\mathcal V}_\ell}(x^{(j)},\ri \omega)$ 
over $\omega \in {\mathbb R}$ have a local minimizer at $\omega^{(j)} \in {\mathbb R}$ such that $z^{(j)} = \ri \omega^{(j)}$, 
which is implied by Theorem \ref{thm:interpolate}, in particular equation (\ref{eq:interpolation}).
Using the analytical formulas for the derivatives of singular value
functions \cite{Bunse-Gerstner1991, Lancaster1964}, we obtain
\begin{equation*}
	\begin{split}
	\sigma_\omega (x^{(j)},z^{(j)})
				& \; = \;
	\Re
	\left(
		(u^{(j)})^\ast
		\frac{\partial \left\{ A(x^{(j)}) -  z^{(j)} I \right\}}{\partial \omega} 
		v^{(j)}
	\right) \\
			& \; = \;
	\Im
	\left(
		(u^{(j)})^\ast
		v^{(j)}
	\right)	
				 \; = \; 
	\Im
	\left(
		(u^{(j)})^\ast
		V_\ell \vartheta^{(j)}
	\right)	\\
			&	\; = \;
	\Re
	\left(
		(u^{(j)})^\ast
		\frac{\partial \left\{ A^{V_\ell}(x^{(j)}) - z^{(j)} V_\ell \right\}}{\partial \omega}
		\vartheta^{(j)}
	\right)
				\; = \;
	\sigma^{{\mathcal V}_\ell}_\omega (x^{(j)},z^{(j)}).
	\end{split}
\end{equation*}

\textbf{(iii)} It is an easy exercise to see that the simplicity of $\sigma(x^{(j)},z^{(j)})$ implies the 
simplicity of $\sigma^{{\mathcal V}_\ell}(x^{(j)}, z^{(j)})$, so both $\sigma(x,z)$ and $\sigma^{{\mathcal V}_\ell}(x,z)$
are differentiable with respect to the real part of $z$ at $(x,z) = (x^{(j)}, z^{(j)})$. Once again an application of
the analytical formulas for singular value functions yield
\begin{equation*}
	\begin{split}
	\sigma_\alpha (x^{(j)},z^{(j)})
				& \; = \;
	\Re
	\left(
		(u^{(j)})^\ast
		\frac{\partial \left\{ A(x^{(j)}) -  z^{(j)} I \right\}}{\partial \alpha} 
		v^{(j)}
	\right)	\\
				& \; = \;
	- \Re
	\left(
		(u^{(j)})^\ast
		v^{(j)}
	\right)	
			 \; = \; 
	- \Re
	\left(
		(u^{(j)})^\ast
		V_\ell \vartheta^{(j)}
	\right)	\\
				& \; = \;
	\Re
	\left(
		(u^{(j)})^\ast
		\frac{\partial \left\{ A^{V_\ell}(x^{(j)}) - z^{(j)} V_\ell \right\}}{\partial \alpha}
		\vartheta^{(j)}
	\right)
				\; = \;
	\sigma^{{\mathcal V}_\ell}_\alpha (x^{(j)},z^{(j)}).
	\end{split}
\end{equation*}
$\;\;$
\end{proof}

\begin{lemma}\label{thm:sval_sder_ineq}
For Algorithm \ref{alg} for all $\ell \in {\mathbb Z}^+$ such that $A(x^{(\ell)})$ is asymptotically stable, 
we have
\[
	 \sigma_{\omega \omega} (x^{(\ell)}, z^{(\ell)})	\;\; 	\leq	\;\;	
	 \sigma^{{\mathcal V}_\ell}_{\omega \omega} (x^{(\ell)}, z^{(\ell)}).
\]
\end{lemma}
\begin{proof}
The functions $\sigma (x, z)$, $\sigma^{{\mathcal V}_\ell} (x, z)$ are twice continuously differentiable
with respect to the imaginary part of $z$ at $(x,z) = (x^{(\ell)}, z^{(\ell)})$, since $\omega^{(\ell)} \in {\mathbb R}$
such that $z^{(\ell)} = \ri \omega^{(\ell)}$ is a minimizer of $\sigma (x^{(\ell)}, \ri \omega)$ and 
$\sigma^{{\mathcal V}_\ell} (x^{(\ell)}, \ri \omega)$ over $\omega \in {\mathbb R}$
as a consequence of equation (\ref{eq:interpolation}). Furthermore, the second derivatives are related by
\begin{equation*}
	\begin{split}
	\sigma_{\omega \omega} (x^{(\ell)}, z^{(\ell)})	\;	& =	\;	
	\lim_{h\rightarrow 0}	\:
	\frac{ \sigma (x^{(\ell)}, \ri (\omega^{(\ell)} + h)) - 2 \sigma (x^{(\ell)}, \ri \omega^{(\ell)}) +	\sigma (x^{(\ell)}, \ri (\omega^{(\ell)} - h)) }{h^2}	\\											\;	& \leq \;
	\lim_{h\rightarrow 0}	\:
	\frac{ \sigma^{{\mathcal V}_\ell} (x^{(\ell)}, \ri(\omega^{(\ell)} + h)) -	2 \sigma^{{\mathcal V}_\ell} (x^{(\ell)}, \ri \omega^{(\ell)}) + 
	\sigma^{{\mathcal V}_\ell} (x^{(\ell)}, \ri (\omega^{(\ell)} - h)) }{h^2} 	\\
								\;	& = \;
	\sigma_{\omega \omega}^{{\mathcal V}_\ell} (x^{(\ell)}, z^{(\ell)})
	\end{split}
\end{equation*}
where the inequality follows from 
$\sigma (x^{(\ell)}, \ri (\omega^{(\ell)} \pm h))  \leq \sigma^{{\mathcal V}_\ell} (x^{(\ell)}, \ri (\omega^{(\ell)} \pm h))$
due to monotonicity (Theorem \ref{thm:mono}) and 
$\sigma (x^{(\ell)}, \ri \omega^{(\ell)})	  =  \sigma^{{\mathcal V}_\ell} (x^{(\ell)}, \ri \omega^{(\ell)})$
due to the interpolation property (Theorem \ref{thm:interpolate}).
\end{proof}

\subsubsection{Derivatives of Distance Functions}
The key to our rate-of-convergence analysis is the interpolation properties between derivatives of 
${\mathcal D}(A(x))$ and ${\mathcal D}^{{\mathcal V}_\ell}(A(x))$. As a starting point for this analysis,
we extend the interpolation result of Theorem \ref{thm:interpolate} to first derivatives.
In what follows
$
	{\mathcal B}(y,\nu)  :=  \left\{ \widetilde{y} \; | \; \| \widetilde{y} - y \|_2 \leq \nu \right\}
$
refers to the closed ball (closed interval if $y$ is a scalar) of radius $\nu$ centered at $y$ either in 
a real Euclidean space or in a complex Euclidean space
depending on whether $y$ is real or complex.

\begin{theorem}[Hermite Interpolation of Distance Functions]\label{thm:Hermite_interpolate2}
The following hold regarding Algorithm \ref{alg} for every $\ell \in {\mathbb Z}^+$ and $j = 1,\dots,\ell$
such that $A(x^{(j)})$ is asymptotically stable: 
If ${\mathcal D}(A(x))$ and ${\mathcal D}^{{\mathcal V}_\ell}(A(x))$ are differentiable at $x = x^{(j)}$, then
		\[	
			\nabla {\mathcal D}(A(x^{(j)})) \;\; = \;\; \nabla {\mathcal D}^{{\mathcal V}_\ell}(A(x^{(j)})).
		\]
\end{theorem}
\begin{proof}
Suppose $v^{(j)}, u^{(j)}$ consist of a consistent pair of unit right, left singular vectors of $A(x^{(j)}) - z^{(j)} I$ 
corresponding to $\sigma_{\min}(A(x^{(j)}) - z^{(j)} I)$. By part (i) of Theorem \ref{thm:Hermite_interpolate}
the vectors $\vartheta^{(j)} := V_\ell^\ast v^{(j)}, u^{(j)}$ form a consistent pair of unit right, left singular vectors of 
$A^{V_\ell}(x^{(j)}) - z^{(j)} V_\ell$ corresponding to $\sigma_{\min}(A^{V_\ell}(x^{(j)}) - z^{(j)} V_\ell)$.
By employing the analytical formulas for the derivatives of singular value functions, we obtain
\begin{equation*}
	\begin{split}
	\frac{\partial {\mathcal D}(A(x^{(j)}))}{\partial x_s}
		 =
	\frac{\partial \sigma_{\min}(A(x^{(j)}) - z^{(j)} I)}{\partial x_s}
		=
	\Re
	\left(
		(u^{(j)})^\ast
		\frac{\partial \left\{ A(x^{(j)}) - z^{(j)} I \right\}}{\partial x_s}
		v^{(j)}
	\right) =	\hskip 6ex \\
	\Re
	\left(
		(u^{(j)})^\ast
		\frac{\partial \left\{ A^{V_\ell}(x^{(j)}) - z^{(j)} V_\ell \right\} }{\partial x_s}
		\vartheta^{(j)}
	\right)
		=
	\frac{\partial \sigma_{\min}(A^{V_\ell}(x^{(j)}) - z^{(j)} V_\ell)}{\partial x_s}	
		 =
	\frac{\partial {\mathcal D}^{{\mathcal V}_\ell}(A(x^{(j)}))}{\partial x_s}
	\end{split}
\end{equation*}
for $s = 1,\dots, d$. This completes the proof.
\end{proof}

\subsubsection{Extended Subspace Framework}\label{sec:extended_sf}
Before going into a formal rate-of-convergence analysis, we shall comment briefly on the extended
subspace framework, formally defined in Algorithm \ref{alge}. In the description of the algorithm, letting $e_p$ 
be the $p$th column of the $d\times d$ identity matrix, we employ the notations 
$e_{p,q} := 1/\sqrt{2} (e_p + e_q)$ if $p\neq q$ and $e_{p,p} := e_p$. The extended version adds 
the singular vectors not only at $x^{(\ell)}$, but also at the nearby points 
$x^{(\ell)} + h^{(\ell)} e_{p,q}$ for $p = 1,\dots, d, \; q = p,\dots,d$.
This obviously brings additional expenses, mainly the computation of the distance to instability at these nearby
points, as well as the computation of the corresponding right singular vectors. For instance, the cost of every
iteration for $d = 2$ is about four times that of the basic framework. Hence, this extended framework aims
to address the large-scale problems depending on a few parameters. 

The interpolation and Hermite interpolation results of Theorems \ref{thm:interpolate} and \ref{thm:Hermite_interpolate2}
do hold beyond $x^{(\ell)}$ also at the nearby points $x^{(\ell)} + h^{(\ell)} e_{p,q}$ for Algorithm \ref{alge}. These are 
formally stated in the next theorem. We omit its proof as the arguments are similar to those 
in the proofs of Theorems \ref{thm:interpolate} and \ref{thm:Hermite_interpolate2}.

\begin{theorem}\label{thm:Hermite_interpolate_ext}
The assertions of Theorems \ref{thm:interpolate} and \ref{thm:Hermite_interpolate2} are also satisfied
by the sequences generated by Algorithm \ref{alge}. Additionally, for Algorithm \ref{alge}, the following
hold for all $\ell \in {\mathbb Z}^+$, $j = 1,\dots, \ell$, $p = 1,\dots,d$ and $q = p, \dots, d$:
\begin{enumerate}
	\item[\bf (i)]
		$ 
			{\mathcal D}(A(x^{(j)}_{p,q}))	  =   \sigma_{\min}(A(x^{(j)}_{p,q}) - z^{(j)}_{p,q} I)  
									  =  \sigma_{\min}(A^{V_\ell}(x^{(j)}_{p,q}) - z^{(j)}_{p,q} V_\ell)	
									 = {\mathcal D}^{{\mathcal V}_\ell}(A(x^{(j)}_{p,q})).
		$ 
	\item[\bf (ii)] If ${\mathcal D}(A(x))$ and ${\mathcal D}^{{\mathcal V}_\ell}(A(x))$ are
			    differentiable at $x = x^{(j)}_{p,q}$, then
				\[	
					\nabla {\mathcal D}(A(x^{(j)}_{p,q})) \;\; = \;\; \nabla {\mathcal D}^{{\mathcal V}_\ell}(A(x^{(j)}_{p,q})).
				\]
\end{enumerate}
\end{theorem}
Additionally, we remark that the global convergence result of Theorem \ref{thm:global_conv}
also applies to the sequence $\{ x^{(\ell)} \}$ generated by Algorithm \ref{alge}.

\begin{algorithm}
 \begin{algorithmic}[1]
\REQUIRE{ The matrix-valued function $A(x)$ of the form (\ref{eq:prob_defn}) with the feasible region $\widetilde{\Omega}$.}
\ENSURE{ The sequences $\{ x^{(\ell)} \}$, $\{ z^{(\ell)} \}$. }
\STATE $x^{(1)} \gets$ a random point in $\widetilde{\Omega}$.
\IF{$\Lambda(A(x^{(1)})) \cap {\mathbb C}^+ = \emptyset$}
	\STATE $z^{(1)} \gets {\rm i} \cdot \arg \min_{\omega \in {\mathbb R}} \; \sigma_{\min} (A(x^{(1)}) - \omega \ri I)$.
	\STATE $V_1	\;	\gets \;$ a unit right singular vector corresponding to $\sigma_{\min} (A(x^{(1)}) - z^{(1)} I)$. 
\ELSE
	\STATE $z^{(1)} \; \gets \;$ an eigenvalue in $\Lambda(A(x^{(1)})) \cap {\mathbb C}^+$.
	\STATE $V_{1}	\;	\gets \;$ a unit eigenvector corresponding to the eigenvalue $z^{(1)}$ of $A(x^{(1)})$.
\ENDIF
\STATE ${\mathcal V}_1 \; \gets \; {\rm span} \{ V_1 \}$.
\FOR{$\ell \; = \; 1, \; 2, \; \dots$}
	\STATE  $x^{(\ell+1)} \gets \arg \max_{x \in \widetilde{\Omega}}	\: {\mathcal D}^{{\mathcal V}_\ell} (A(x))$.
	\IF{$\Lambda(A(x^{(\ell+1)})) \cap {\mathbb C}^+ = \emptyset$}
		\STATE $z^{(\ell+1)} \gets {\rm i} \cdot \arg \min_{\omega \in {\mathbb R}} \; \sigma_{\min} (A(x^{(\ell+1)}) - \omega \ri I)$.
		\STATE $v_{\ell+1}	\;	\gets \;$ a right singular vector corresponding to $\sigma_{\min} (A(x^{(\ell+1)}) - z^{(\ell+1)} I)$. 
	\ELSE
		\STATE $z^{(\ell+1)} \; \gets \;$ an eigenvalue in $\Lambda(A(x^{(\ell+1)})) \cap {\mathbb C}^+$.
		\STATE $v_{\ell+1}	\;	\gets \;$ an eigenvector corresponding to the eigenvalue $z^{(\ell+1)}$ of $A(x^{(\ell+1)})$.
	\ENDIF
	\STATE $h^{(\ell+1)} \; \gets \| x^{(\ell+1)}  -  x^{(\ell)}  \|_2$
	\FOR{$p = 1, \dots, d$}
		\FOR{$q = p, \dots, d$}
			\STATE $x^{(\ell+1)}_{p,q} \; \gets \; x^{(\ell+1)} + h^{(\ell+1)} e_{p,q}$.
			\IF{$\Lambda(A(x^{(\ell+1)}_{p,q})) \cap {\mathbb C}^+ = \emptyset$}
				\STATE $z^{(\ell+1)}_{p,q} \; \gets \; \ri \cdot \arg \min_{\omega \in {\mathbb R}} \; \sigma_{\min} (A(x^{(\ell+1)}_{p,q}) - \omega \ri I)$.
				\STATE $v^{(\ell+1)}_{p,q}	\;	\gets \;$ a right singular vector \\
					\hskip 25ex		corresponding to $\sigma_{\min} (A(x^{(\ell+1)}_{p,q}) - z^{(\ell+1)}_{p,q} I)$.
			\ELSE
				\STATE $z^{(\ell+1)}_{p,q} \; \gets \;$ an eigenvalue in $\Lambda(A(x^{(\ell+1)}_{p,q})) \cap {\mathbb C}^+$.
				\STATE $v^{(\ell+1)}_{p,q}	\;	\gets \;$ an eigenvector \\
					\hskip 18ex		corresponding to the eigenvalue $z^{(\ell+1)}_{p,q}$ of $A(x^{(\ell+1)}_{p,q})$.
			\ENDIF
		\ENDFOR
	\ENDFOR
	\STATE $V_{\ell+1} \gets {\rm orth} 
			\left(\left[ \begin{array}{cccccccc} V_{\ell} & v_{\ell+1} & v^{(\ell+1)}_{1,1} & \dots & v^{(\ell+1)}_{1,d} & v^{(\ell+1)}_{2,2} & \dots & v^{(\ell+1)}_{d,d} \end{array} \right] \right)$.
	\STATE ${\mathcal V}_{\ell+1}  \gets  {\rm Col} (V_{\ell+1})$.
\ENDFOR
 \end{algorithmic}
\caption{The Extended Subspace Framework}
\label{alge}
\end{algorithm}

\subsubsection{Rate-of-Convergence Analysis}\label{sec:rate_convergence}

The rate-of-convergence analysis here is inspired by \cite{Kangal2015}. Especially the Hermite
interpolation property (Theorems \ref{thm:Hermite_interpolate2} and \ref{thm:Hermite_interpolate_ext}) 
facilitates this. However, the fact that ${\mathcal D}(A(x)), {\mathcal D}^{{\mathcal V}_\ell}(A(x))$ are 
defined in terms of the global minimizers of $\sigma(x,z), \sigma^{{\mathcal V}_\ell}(x,z)$ over 
$z \in {\mathbb C}^+$ brings subtleties. Our first task is to show that 
${\mathcal D}(A(x)) = \sigma(x,\ri \omega(x))$, 
${\mathcal D}^{{\mathcal V}_\ell}(x) = \sigma^{{\mathcal V}_\ell}(x,\ri \omega^{{\mathcal V}_\ell}(x))$
for all $x$ close to $x_\ast := \lim_{\ell \rightarrow \infty} x^{(\ell)}$ (recall that $\{ x^{(\ell)} \}$ itself
is assumed to be convergent, and Theorem \ref{thm:global_conv} ensures that its limit $x_\ast$
is a global maximizer of ${\mathcal D}(A(x))$ over $x \in \widetilde{\Omega}$), where $\omega(x)$, 
$\omega^{{\mathcal V}_\ell}(x)$ are the real-valued functions defined implicitly by $\sigma_\omega(x,\ri \omega(x)) =0$, 
$\sigma^{{\mathcal V}_\ell}_\omega(x, \ri \omega^{{\mathcal V}_\ell}(x)) = 0$.
Throughout this section, for $x \in {\mathbb R}^d, z \in {\mathbb C}$ and $\epsilon_1, \epsilon_2 > 0$,
we employ the notations
\begin{equation*} 
	\begin{split}
	{\mathcal N}(x, z; \epsilon_1, \epsilon_2) := 
		\{ (\widetilde{x},\widetilde{z}) \in {\mathbb R}^d \times {\mathbb C}
			\; | \; \| \widetilde{x} - x \|_2 \leq \epsilon_1, \;\; | \widetilde{z} - z | \leq \epsilon_2 \},	\;\;	{\rm and}	\\
	{\mathcal N}(x, \Im z; \epsilon_1, \epsilon_2) := 
		\{ (\widetilde{x},\widetilde{\omega}) \in {\mathbb R}^d \times {\mathbb R}
		\; | \; \| \widetilde{x} - x \|_2 \leq \epsilon_1, \;\; | \widetilde{\omega} - \Im z | \leq \epsilon_2 \}	\hskip 6ex
	\end{split}
\end{equation*}
for neighborhoods of $(x,z)$ and $(x, \Im z)$.
Moreover, $\ri {\mathbb R}$ refers to the set of purely imaginary numbers.

\begin{lemma}\label{thm:lbounds_sder}
Suppose that $\sigma(x_\ast,z_\ast)$ is simple, where $z_\ast \in \ri {\mathbb R}$ is the global minimizer of
$\sigma(x_\ast,z)$ over $z \in {\mathbb C}^+$, which we assume is unique. Additionally, assume
$\: \sigma_{\omega\omega}(x_\ast,z_\ast) = \delta   >   0 \:$. There exist neighborhoods
	${\mathcal N}(x_\ast, z_\ast; \epsilon_1, \epsilon_2)$ 
		and
	${\mathcal N}(x_\ast, \Im z_\ast; \widetilde{\epsilon}_1, \widetilde{\epsilon}_2)$
of $(x_\ast, z_\ast)$ and $(x_\ast, \Im z_\ast)$ with $\: \widetilde{\epsilon}_1 \leq \epsilon_1$, 
$\: \widetilde{\epsilon}_2 \leq \epsilon_2 \:$such that
\begin{enumerate}
	\item[\bf (i)] the singular values $\sigma(x,z)$ and $\sigma^{{\mathcal V}_\ell}(x,z)$ for all $\ell$ large enough
	remain simple and their first three derivatives are bounded above by constants uniformly
	for all $(x,z) \in {\mathcal N}(x_\ast, z_\ast; \epsilon_1, \epsilon_2)$, where the constants are independent of $\ell$,
	\item[\bf (ii)] $\sigma_{\omega \omega} (\widetilde{x}, \ri \widetilde{\omega}) \geq 3\delta/4$	
	$\: \forall (\widetilde{x}, \widetilde{\omega}) \in {\mathcal N}(x_\ast, \Im z_\ast; \widetilde{\epsilon}_1, \widetilde{\epsilon}_2)$, and
	\item[\bf (iii)] $\forall \ell$ large enough,
	$\sigma^{{\mathcal V}_\ell}_{\omega \omega} (\widetilde{x}, \ri \widetilde{\omega}) \geq \delta/2$	
	$\: \forall (\widetilde{x}, \widetilde{\omega}) \in {\mathcal N}(x_\ast, \Im z_\ast; \widetilde{\epsilon}_1, \widetilde{\epsilon}_2)$.
\end{enumerate}
\end{lemma}
\begin{proof}
For assertion (i) we refer to \cite[Proposition 2.9]{Kangal2015}. By the boundedness 
of $\sigma_{\omega \omega \omega} (x, \ri \omega)$ in ${\mathcal N}(x_\ast,z_\ast; \epsilon_1,\epsilon_2)$ and 
$\sigma_{\omega\omega}(x_\ast,z_\ast) = \delta$, we infer a neighborhood 
$\widehat{\mathcal N} := 
{\mathcal N}(x_\ast, \Im z_\ast; \widehat{\epsilon}_1, \widehat{\epsilon}_2) \subseteq {\mathcal N}(x_\ast,z_\ast,\epsilon_1,\epsilon_2)$ 
of $(x_\ast, \Im z_\ast)$ such that $\sigma_{\omega \omega} (\widetilde{x}, \ri \widetilde{\omega}) \geq 3\delta/4$ for all 
$(\widetilde{x}, \widetilde{\omega}) \in \widehat{\mathcal N}$.

Now, since 
$\lim_{\ell\rightarrow \infty} \sigma(x^{(\ell)}, z^{(\ell)}) = \lim_{\ell\rightarrow \infty} {\mathcal D}(A(x^{(\ell)})) = \sigma(x_\ast, z_\ast)$, 
by the uniqueness of the global minimizer $z_\ast$ and the continuity of $\sigma(x,z)$,
we must have $z^{(\ell)} \rightarrow z_\ast$. Choose $\ell$ large enough so that 
$(x^{(\ell)}, \Im z^{(\ell)})$ is in $\widehat{\mathcal N}$. By Lemma \ref{thm:sval_sder_ineq}
\[
	 \sigma_{\omega \omega}^{{\mathcal V}_\ell} (x^{(\ell)}, z^{(\ell)}) 
	 			\geq 
		\sigma_{\omega \omega} (x^{(\ell)}, z^{(\ell)}) \geq 3\delta/4
\]
for all such large $\ell$. Finally, since $\sigma_{\omega \omega \omega}^{{\mathcal V}_\ell} (x, \ri \omega)$ 
are also uniformly bounded in $\widehat{\mathcal N}$ by a constant independent of $\ell$, there must
exist a neighborhood $\widetilde{\mathcal N} \subseteq \widehat{\mathcal N}$ of $(x_\ast, \Im z_\ast)$
such that $\sigma_{\omega\omega}(\widetilde{x}, \ri \widetilde{\omega}) \geq \delta/2$ for all 
$(\widetilde{x}, \widetilde{\omega}) \in \widetilde{\mathcal N}$ and all $\ell$ sufficiently large.
\end{proof}

The significance of the last lemma is that it implies the existence of the functions $\widetilde{\omega}(x)$ and 
$\widetilde{\omega}^{{\mathcal V}_\ell}(x)$ for $x$ near $x_\ast$ such that 
$\ri \widetilde{\omega}(x)$ and $\ri \widetilde{\omega}^{{\mathcal V}_\ell}(x)$
satisfy the first order optimality conditions to be a minimizer of $\sigma(x,z)$ and $\sigma^{{\mathcal V}_\ell}(x,z)$, respectively, 
over $z \in {\mathbb C}^+$. This is formally stated by the next result.
\begin{lemma}[Local Representations of the Minimizers]\label{thm:local_rep}
Suppose $(x_\ast, z_\ast)$ is as in Lemma \ref{thm:lbounds_sder} and satisfies the assumptions of that
lemma. Additionally, suppose $\sigma_{\alpha}(x_\ast,z_\ast)  >   0 \:$. For some $\eta_1, \eta_2 > 0$, the following hold:
\begin{enumerate}
	\item[\bf (i)] There exists a unique three times differentiable function 
	$\widetilde{\omega}(x) : {\mathcal B}(x_\ast,\eta_1) \rightarrow {\mathcal B}(\Im z_\ast,\eta_2)$ 
	such that $\widetilde{\omega}(x_\ast) = 	\Im z_\ast$, as well as
	\begin{equation}\label{eq:implicit_omega}
		\sigma_\omega(x,\ri \widetilde{\omega}(x)) = 0 \quad	{\rm and}	\quad 
		\sigma_{\omega \omega}(x,\ri \widetilde{\omega}(x)) \geq \delta/2	\quad	
		\forall  x \in {\mathcal B}(x_\ast, \eta_1).
	\end{equation}
	Furthermore, $\ri \widetilde{\omega}(x)$ is the unique point in the ball ${\mathcal B}(z_\ast,\eta_2)$
	that satisfies the first order optimality conditions to be a minimizer of $\sigma(x,z)$
	over $z \in {\mathbb C}^+$ for all $x \in {\mathcal B}(x_\ast, \eta_1)$;
	\item[\bf (ii)] For all $\ell$ large enough, in particular satisfying
	$(x^{(\ell)}, z^{(\ell)}) \in {\mathcal N}(x_\ast, z_\ast; \eta_1, \eta_2)$,
	there exists a unique three times differentiable function 
	$\widetilde{\omega}^{{\mathcal V}_\ell}(x) : {\mathcal B}(x_\ast,\eta_1) \rightarrow {\mathcal B}(\Im z_\ast,\eta_2)$   
	such that $\widetilde{\omega}^{{\mathcal V}_\ell}(x^{(\ell)}) = \Im z^{(\ell)}$, as well as
	\begin{equation}\label{eq:implciti_omegar}
		\sigma_\omega^{{\mathcal V}_\ell}(x, \ri \widetilde{\omega}^{{\mathcal V}_\ell}(x)) = 0 \quad		{\rm and}	\quad
		\sigma_{\omega\omega}^{{\mathcal V}_\ell}(x, \ri \widetilde{\omega}^{{\mathcal V}_\ell}(x)) \geq \delta/2	\quad
		\forall  x \in {\mathcal B}(x_\ast, \eta_1).
	\end{equation}
	Furthermore, for such $\ell$ and for all $x \in {\mathcal B}(x_\ast, \eta_1)$,
	$\ri \widetilde{\omega}^{{\mathcal V}_\ell}(x)$ is the unique point in the ball ${\mathcal B}(z_\ast,\eta_2)$
	that satisfies the first order optimality conditions to be a minimizer of $\sigma^{{\mathcal V}_\ell}(x,z)$
	over $z \in {\mathbb C}^+$.
%
\end{enumerate}
\end{lemma}
\begin{proof}
Consider the neighborhood ${\mathcal N}(x_\ast, \Im z_\ast; \widetilde{\epsilon}_1, \widetilde{\epsilon}_2)$ as in Lemma \ref{thm:lbounds_sder}.
For all $(x,\omega)$ in this neighborhood, the singular values $\sigma(x,\ri \omega)$ and $\sigma^{{\mathcal V}_\ell}(x,\ri \omega)$
for large $\ell$ are simple and their second derivatives with respect to $\omega$ are bounded below by $\delta/2$ uniformly. 
By the implicit function theorem,
for some $\eta_1 < \widetilde{\epsilon}_1$,  $\eta_2 < \widetilde{\epsilon}_2$, there exist a unique function $\widetilde{\omega}(x)$ satisfying 
(\ref{eq:implicit_omega}), and a unique function $\widetilde{\omega}^{{\mathcal V}_\ell}(x)$  satisfying (\ref{eq:implciti_omegar}) 
for all $\ell$ large enough, in particular $(x^{(\ell)}, z^{(\ell)}) \in {\mathcal N}(x_\ast,z_\ast; \eta_1, \eta_2)$.
Note that the uniformity of the radii $\eta_1, \eta_2$ over all such $\ell$ is due to the uniform lower bound $\delta/2$ on 
the second derivatives $\sigma_{\omega\omega}(x,\ri \omega)$ and $\sigma^{{\mathcal V}_\ell}_{\omega\omega}(x,\ri \omega)$
in ${\mathcal N}(x_\ast, \Im z_\ast; \widetilde{\epsilon}_1, \widetilde{\epsilon}_2)$.  We also remark that 
$\widetilde{\omega}(x)$ and $\widetilde{\omega}^{{\mathcal V}_\ell}(x)$ are three
times differentiable, because $\sigma(x,\ri \omega)$ and $\sigma^{{\mathcal V}_\ell}(x,\ri \omega)$ are simple, hence
real analytic, in ${\mathcal N}(x_\ast, \Im z_\ast; \widetilde{\epsilon}_1, \widetilde{\epsilon}_2)$.

We complete the proof by arguing that $\ri \widetilde{\omega}(x)$ and $\ri \widetilde{\omega}^{{\mathcal V}_\ell}(x)$ 
are the unique first order optimal points
in a ball in the complex plane around $z_\ast$. To this end, $\sigma_\alpha(x_\ast,z_\ast) > 0$ by assumption, and $\sigma_\alpha(x,z) > 0$
in a neighborhood $\widetilde{\mathcal N}$ of $(x_\ast, z_\ast)$ due to continuous differentiability of $\sigma(x,z)$ around $(x_\ast, z_\ast)$ 
as implied by part (i) of Lemma \ref{thm:lbounds_sder}. Choose $\ell$ even larger if necessary so that $(x^{(\ell)}, z^{(\ell)})$ is in this 
neighborhood and $\sigma_\alpha^{{\mathcal V}_\ell}(x^{(\ell)}, z^{(\ell)}) = \sigma_\alpha(x^{(\ell)},z^{(\ell)}) > 0$, where
we employ part (iii) of Theorem \ref{thm:Hermite_interpolate} for the equality. By Lemma \ref{thm:lbounds_sder}
once again, in particular due to analyticity of $\sigma^{{\mathcal V}_\ell}(x,z)$ and a 
uniform upper bound on $\sigma_{\alpha\alpha}^{{\mathcal V}_\ell}(x,z)$ independent of $\ell$,
there exists a neighborhood $\widehat{\mathcal N} \subseteq \widetilde{\mathcal N}$ of $(x_\ast, z_\ast)$ such that
\[
	\sigma_\alpha(x,z) > 0	\quad	{\rm and}	\quad	\sigma^{{\mathcal V}_\ell}_\alpha(x,z) >	0
	\;\;\;\;		\forall (x,z) \in \widehat{\mathcal N}
\]
for all $\ell$ large enough. Now $\eta_1, \eta_2$ as in the previous paragraph can be chosen small enough if necessary
so that ${\mathcal B}(x_\ast,\eta_1) \times {\mathcal B}(z_\ast, \eta_2) \subseteq \widehat{\mathcal N}$.
Consequently, for each $x \in {\mathcal B}(x_\ast,\eta_1)$, since 
$\ri \widetilde{\omega}(x), \ri \widetilde{\omega}^{{\mathcal V}_\ell}(x) \in {\mathcal B}(z_\ast, \eta_2)$, we have
\[
	\sigma_\omega(x,\ri \widetilde{\omega}(x)) = 0,	\;\;	\sigma_\alpha(x,\ri \widetilde{\omega}(x)) > 0
	\quad {\rm and} \quad
	\sigma_\omega^{{\mathcal V}_\ell}(x,\ri \widetilde{\omega}^{{\mathcal V}_\ell}(x)) = 0, \;\;	
	\sigma_\alpha^{{\mathcal V}_\ell}(x,\ri \widetilde{\omega}^{{\mathcal V}_\ell}(x)) > 0,
\]
that is $\ri \widetilde{\omega}^{{\mathcal V}_\ell}(x)$ and $\ri \widetilde{\omega}^{{\mathcal V}_\ell}(x)$ satisfy the first order optimality 
conditions to be a minimizer of $\sigma(x,z)$ and $\sigma^{{\mathcal V}_\ell}(x,z)$ over $z \in {\mathbb C}^+$.
The uniqueness of $\ri \widetilde{\omega}(x)$, $\ri \widetilde{\omega}^{{\mathcal V}_\ell}(x)$
as first order optimal points inside ${\mathcal B}(z_\ast, \eta_2)$ follow from the uniqueness of
$\widetilde{\omega}(x), \widetilde{\omega}^{{\mathcal V}_\ell}(x) \in {\mathcal B}(\Im z_\ast, \eta_2)$
satisfying $\sigma_\omega(x,\ri \widetilde{\omega}(x)) = 0$, $\sigma_\omega^{{\mathcal V}_\ell}(x,\ri \widetilde{\omega}^{{\mathcal V}_\ell}(x)) = 0$,
as well as the fact that $\sigma_\alpha(x,z) \neq 0$, $\sigma_\alpha^{{\mathcal V}_\ell}(x,z) \neq 0$
for $z \in {\mathcal B}(z_\ast,\eta_2)$ such that $\Re z > 0$.
\end{proof}

The next result has two important conclusions. First it shows that the unique first order optimal points of the previous
lemma are indeed global minimizers where ${\mathcal D}(A(x))$ and ${\mathcal D}^{{\mathcal V}_\ell}(A(x))$ 
are attained. Secondly, it establishes that the smoothness assumption on ${\mathcal D}(A(x))$
at $x_\ast$ implies the existence of a ball centered at $x_\ast$ in which ${\mathcal D}(A(x))$,
as well as ${\mathcal D}^{{\mathcal V}_\ell}(A(x))$ for all large $\ell$, are smooth. In this result and elsewhere,
$\sigma_{-2}(\cdot)$ refers to the second smallest singular value of its matrix argument.
\begin{lemma}[Uniform Smoothness]\label{thm:extend_smoothness}
Suppose that the sequence $\{ x^{(\ell)} \}$ by Algorithm \ref{alg} or Algorithm \ref{alge} converges 
to a point $x_\ast$ that is strictly in the interior of $\: \widetilde{\Omega}$ and that ${\mathcal D}(A(x_\ast))$
is attained at a unique $z_\ast$, the singular value $\sigma( x_\ast, z_\ast ) > 0$
is simple, $\sigma_{\omega\omega}( x_\ast, z_\ast) > 0$ and $\sigma_\alpha (x_\ast, z_\ast) > 0$. 
Let $\eta_1$, $\widetilde{\omega}(x)$, $\widetilde{\omega}^{{\mathcal V}_\ell}(x)$ for all $\ell$ large enough, 
say $\ell \geq \ell_1$, be as in Lemma \ref{thm:local_rep}. 

Then there exist an $\eta \leq \eta_1$ and an $\varepsilon > 0$ satisfying the following:
\begin{enumerate}
	\item[\bf (i)] We have $\min \{ \sigma(x,\ri \widetilde{\omega}(x)),\: 
							\sigma_{-2}(x,\ri \widetilde{\omega}(x)) - \sigma(x,\ri \widetilde{\omega}(x))\} \geq \varepsilon$,
	and ${\mathcal D}(A(x))$ is uniquely attained at $\ri \widetilde{\omega}(x)$ for all $x \in {\mathcal B}(x_\ast,\eta)$.
	\item[\bf (ii)] Additionally, for all $\ell$ large enough, say $\ell \geq \ell_2 \geq \ell_1$ and for all $x \in {\mathcal B}(x_\ast,\eta)$,
	$
		\min \{ \sigma^{{\mathcal V}_\ell}(x,\ri \widetilde{\omega}^{{\mathcal V}_\ell}(x)), \:$ 
	$	\sigma_{-2}^{{\mathcal V}_\ell}(x,\ri \widetilde{\omega}^{{\mathcal V}_\ell}(x)) - 
		\sigma^{{\mathcal V}_\ell}(x,\ri \widetilde{\omega}^{{\mathcal V}_\ell}(x))\} 
		\geq \varepsilon
	$,
	and
	${\mathcal D}^{{\mathcal V}_\ell}(A(x))$
	is attained at $\ri \widetilde{\omega}^{{\mathcal V}_\ell}(x)$ 
	uniquely.
\end{enumerate}
\end{lemma}
\begin{proof}
Following the arguments in \cite[Lemma 2.8]{Kangal2015}, due to the simplicity and positivity of $\sigma(x_\ast, z_\ast)$
and the interpolation properties $\sigma(x^{(\ell)}, z^{(\ell)}) = \sigma^{{\mathcal V}_\ell} (x^{(\ell)}, z^{(\ell)})$,
there exists an $\varepsilon$ and a neighborhood ${\mathcal N}(x_\ast, \Im z_\ast; \widetilde{\eta}_1, \widetilde{\eta}_2)$
of $(x_\ast, \Im z_\ast)$ such that
\begin{equation}\label{eq:uniform_gap}
 	\min \{ \sigma(x,\ri \omega),\: \sigma_{-2}(x,\ri \omega) - \sigma(x,\ri \omega)\} \geq \varepsilon
	\quad	\forall (x,\omega) \in {\mathcal N}(x_\ast, \Im z_\ast; \widetilde{\eta}_1, \widetilde{\eta}_2),
\end{equation}
and, for all $\ell$ large enough,
\begin{equation}\label{eq:uniform_gapr}
	\min \{ \sigma^{{\mathcal V}_\ell}(x,\ri \omega), \: \sigma_{-2}^{{\mathcal V}_\ell}(x,\ri \omega) 
											- 	\sigma^{{\mathcal V}_\ell}(x,\ri \omega)\} \geq \varepsilon 
	\quad	\forall (x,\omega) \in {\mathcal N}(x_\ast, \Im z_\ast; \widetilde{\eta}_1, \widetilde{\eta}_2).
\end{equation}
Furthermore, let $\widetilde{\omega}(x)$ and $\widetilde{\omega}^{{\mathcal V}_\ell}(x)$ for $\ell$ large enough, say 
$\ell \geq \ell_1$,  be as in Lemma \ref{thm:local_rep} satisfying its assertions (i) and (ii) for some 
$\eta_1 \leq \widetilde{\eta}_1$, $\eta_2 \leq \widetilde{\eta}_2$.

First we prove that ${\mathcal D}(A(x))$ is minimized uniquely at $\ri \widetilde{\omega}(x)$
for $x$ in a ball centered at $x_\ast$. To this end, for each $x \in {\mathcal B}(x_\ast,\eta_1)$, 
we define the functions
\begin{equation*}
	\begin{split}
	\delta_2(x) & \;\; := \;\; \inf \left\{ \sigma(x,z) \; | \; z \in {\mathbb C}^+ \backslash {\mathcal B}(z_\ast, \eta_2) \right\},
	\quad
	\delta_1(x) \;\; := \;\; \min \left\{ \sigma(x,z) \; | \;  z \in {\mathcal B}(z_\ast, \eta_2) \right\},	\\
	\delta(x) \;\:	 & \;\; := \;\;
				\delta_2(x)
							\; - \;
				\delta_1(x).	
	\end{split}
\end{equation*}
In particular, let 
$\;
	\delta_\ast  :=  \delta(x_\ast) 
			 :=   \delta_2(x_\ast)
						 -	\delta_1(x_\ast)
			 = 
					\delta_2(x_\ast)
						 -	{\mathcal D}(A(x_\ast)) \; > \; 0.
$
Consider the ball ${\mathcal B}(x_\ast,\delta_\ast/8\zeta)$, where the constant $\zeta$ is as in Lemma \ref{thm:Lip_cont}.
For each $x \in {\mathcal B}(x_\ast,\eta)$ with $\eta := \min \{ \delta_\ast/ (8 \zeta) , \eta_1 \}$, by part (ii) of Lemma \ref{thm:Lip_cont},
\[
	\delta_2(x_\ast) - \delta_2(x)  \leq \delta_\ast/8
	\quad	{\rm and}	\quad
	\delta_1(x) - \delta_1(x_\ast)  \leq \delta_\ast/8,
\]
implying
\begin{equation}\label{eq:gap_minimizers}
	\delta(x) = \delta_2(x) - \delta_1(x) \geq \delta_2(x_\ast) - \delta_1(x_\ast) - \delta_\ast/4 = 3 \delta_\ast /4.
\end{equation}
This means that $\sigma(x,z)$ is minimized over $z \in {\mathbb C}^+$ globally by some point in the interior of 
${\mathcal B}(z_\ast,\eta_2)$. This has to be $\ri \widetilde{\omega}(x)$, because 
Lemma \ref{thm:local_rep} asserts that all other points in ${\mathcal B}(z_\ast,\eta_2)$
violate the first order optimality conditions to be a minimizer of $\sigma(x,z)$ over $z \in {\mathbb C}^+$.
Note also that, since $\eta \leq \eta_1 \leq \widetilde{\eta}_1$ as well as
$\widetilde{\omega}(x) \in {\mathcal B}(\Im z_\ast,\eta_2)$ with $\eta_2 \leq \widetilde{\eta}_2$, 
it follows from (\ref{eq:uniform_gap}) that
\[
	\min \{ \sigma(x,\ri \widetilde{\omega}(x)),\: \sigma_{-2}(x,\ri \widetilde{\omega}(x)) - \sigma(x,\ri \widetilde{\omega}(x))\} \geq \varepsilon
	\quad	\forall x \in {\mathcal B}(x_\ast, \eta).
\]
This completes the proof of assertion (i).

Now there exists $\ell_2 \geq \ell_1$ such that $x^{(\ell)} \in {\mathcal B}(x_\ast, \eta)$ and
$z^{(\ell)} \in {\mathcal B} (z_\ast, \eta_2)$ for all $\ell \geq \ell_2$. 
We show the satisfaction of assertion (ii) for all $\ell \geq \ell_2$. For such an $\ell$ and for an $x \in {\mathcal B}(x_\ast,\eta)$,  define
\begin{equation*}
	\begin{split}
	\delta^{{\mathcal V}_\ell}_2(x) & \;\; := \;\; \inf \left\{ \sigma^{{\mathcal V}_\ell}(x,z) \; | \; z \in {\mathbb C}^+ \backslash {\mathcal B}(z_\ast, \eta_2) \right\} \\
	\delta^{{\mathcal V}_\ell}_1(x) & \;\; := \;\; \min \left\{ \sigma^{{\mathcal V}_\ell}(x,z) \; | \;  z \in {\mathcal B}(z_\ast, \eta_2) \right\},	\quad
	\delta^{{\mathcal V}_\ell}(x) \:	  \; := \;
				\delta^{{\mathcal V}_\ell}_2(x)
							\; - \;
				\delta^{{\mathcal V}_\ell}_1(x).	
	\end{split}
\end{equation*}
The monotonicity property $\sigma^{{\mathcal V}_\ell}(x,z) \geq \sigma(x,z)$ for all $z$
implies $\delta^{{\mathcal V}_\ell}_j(x) \geq \delta_j(x)$ for $j =1,2$. Additionally,
$\delta^{{\mathcal V}_\ell}_1(x^{(\ell)}) \leq \sigma^{{\mathcal V}_\ell}(x^{(\ell)},z^{(\ell)}) = \sigma(x^{(\ell)},z^{(\ell)}) = \delta_1(x^{(\ell)})$
implies $\delta^{{\mathcal V}_\ell}_1(x^{(\ell)}) = \delta_1 (x^{(\ell)})$. These observations lead us to
\begin{equation}\label{eq:intermediate_gap}
	\delta^{{\mathcal V}_\ell}(x^{(\ell)})    \;\; = \;\; 
					\delta^{{\mathcal V}_\ell}_2(x^{(\ell)})
							\; - \;
					\delta^{{\mathcal V}_\ell}_1(x^{(\ell)})	
							  \;\; \geq \;\;
					\delta_2(x^{(\ell)})
							\; - \;
					\delta_1(x^{(\ell)}) \; \geq \; 3\delta_\ast /4,
\end{equation}
where the last bound is due to (\ref{eq:gap_minimizers}). Now for each $x \in {\mathcal B}(x_\ast,\eta)$,
by part (ii) of Lemma \ref{thm:Lip_cont}, we have
\begin{equation*}
	\begin{split}
		\delta_2^{{\mathcal V}_\ell}(x^{(\ell)}) - \delta_2^{{\mathcal V}_\ell}(x) & \leq \zeta \| x^{(\ell)} - x \|_2 
								\leq \zeta (\| x^{(\ell)} - x_\ast \|_2 + \| x_\ast - x \|_2)
											\leq	\delta_\ast/4
		\quad	{\rm and}		\\
		\delta_1^{{\mathcal V}_\ell}(x) - \delta_1^{{\mathcal V}_\ell}(x^{(\ell)}) & \leq  \zeta \| x  -  x^{(\ell)} \|_2 
								\leq \zeta (\| x - x_\ast \|_2 + \| x_\ast - x^{(\ell)} \|_2)
											\leq \delta_\ast/4,
	\end{split}
\end{equation*}
that gives rise to
\[
	\delta^{{\mathcal V}_\ell}(x) \;\; = \;\; \delta_2^{{\mathcal V}_\ell}(x) - \delta_1^{{\mathcal V}_\ell}(x)
				\;\; \geq \;\; \delta_2^{{\mathcal V}_\ell}(x^{(\ell)}) - \delta_1^{{\mathcal V}_\ell}(x^{(\ell)}) - \delta_\ast/2 
				\;\; \geq \;\; \delta_\ast /4,
\]
where the last lower bound follows from (\ref{eq:intermediate_gap}). 
Hence, the global minimizer of $\sigma^{{\mathcal V}_\ell}(x,z)$ over $z \in {\mathbb C}^+$ also lies
strictly in the interior of ${\mathcal B}(z_\ast,\eta_2)$. Once again this global minimizer has to be 
$\ri \widetilde{\omega}^{{\mathcal V}_\ell}(x)$,
as the other points in ${\mathcal B}(z_\ast, \eta_2)$ violate the first order conditions to be a minimizer of 
$\sigma^{{\mathcal V}_\ell}(x,z)$ over $z \in {\mathbb C}^+$
due to Lemma \ref{thm:local_rep}. Finally, since $\eta \leq \widetilde{\eta}_1$ and
$\widetilde{\omega}^{{\mathcal V}_\ell}(x) \in {\mathcal B}(\Im z_\ast,\eta_2)$ with $\eta_2 \leq \widetilde{\eta}_2$, 
we deduce the following uniform gap from (\ref{eq:uniform_gapr}):
\[
 	\min \{ \sigma^{{\mathcal V}_\ell}(x,\ri \widetilde{\omega}^{{\mathcal V}_\ell}(x)), \: 
	\sigma_{-2}^{{\mathcal V}_\ell}(x,\ri \widetilde{\omega}^{{\mathcal V}_\ell}(x)) - 
	\sigma^{{\mathcal V}_\ell}(x,\ri \widetilde{\omega}^{{\mathcal V}_\ell}(x))\} \geq \varepsilon 
	\quad	\forall x \in {\mathcal B}(x_\ast, \eta),
\]
completing the proof of assertion (ii).
\end{proof}

Now we state two results regarding the second and third derivatives of ${\mathcal D}(A(x))$
and ${\mathcal D}^{{\mathcal V}_\ell}(A(x))$ for $x$ inside the ball 
${\mathcal B}(x_\ast, \eta)$ specified by Lemma \ref{thm:extend_smoothness}. Note that there 
are no interpolation properties between the second and higher order derivatives of these two 
distance to instability functions.

\begin{lemma}[Uniform Boundedness of the Third Derivatives]\label{thm:bound_thirdder}
Suppose that $\{ x^{(\ell)} \}$ by Algorithm \ref{alg} or Algorithm \ref{alge} converges 
to a point $x_\ast$ that satisfies the assumptions in Lemma \ref{thm:extend_smoothness}.
There exist $\eta \in {\mathbb R}^+$ and $\xi \in {\mathbb R}^+$ such that, for all 
$\ell$ large enough, we have
\[
	\left| \frac{ \partial^3 \left[ {\mathcal D}^{{\mathcal V}_\ell} (A(\widetilde{x})) \right] }{ \partial x_p \partial x_q \partial x_s } \right|
		\;	\leq	\;	\xi	\quad	\forall	\widetilde{x} \in {\mathcal B}(x_\ast,\eta)
\]
for $p,q,s = 1,\dots, d$.
\end{lemma}
\begin{proof}
By Lemma \ref{thm:extend_smoothness} there exists an $\eta$ such that, for all $\ell$ large enough
and for all $x \in {\mathcal B}(x_\ast,\eta)$,
${\mathcal D}^{{\mathcal V}_\ell}(A(x)) = \sigma(x, \ri \omega^{{\mathcal V}_\ell}(x))$ for some 
real-valued three times differentiable function $\omega^{{\mathcal V}_\ell}(x)$, and the singular
value $\sigma(x, \ri \omega^{{\mathcal V}_\ell}(x))$ is simple, bounded away from zero.
Consequently, ${\mathcal D}^{{\mathcal V}_\ell}(A(x))$ is three times differentiable on ${\mathcal B}(x_\ast,\eta)$.
Each third derivative of ${\mathcal D}^{{\mathcal V}_\ell}(A(x))$ at a given $\widetilde{x} \in {\mathcal B}(x_\ast,\eta)$ 
can be expressed in terms of the first three derivatives of $\sigma^{{\mathcal V}_\ell}(x,z)$ 
at $(x,z) = (\widetilde{x}, \ri \omega^{{\mathcal V}_\ell}(\widetilde{x}))$ and the first two derivatives of 
$\omega^{{\mathcal V}_\ell}(x)$ at $x = \widetilde{x}$, all of which can be bounded solely in terms of $\| A_j \|$
and $\max_{x \in {\mathcal B}(x_\ast, \eta)} \: | f_j(x) |$ for $j = 1,\dots,\kappa$, as well as the reciprocal of the 
uniform gap $\sigma^{{\mathcal V}_\ell}_{-2}(x,\ri \omega^{{\mathcal V}_\ell}(x)) - 
				\sigma^{{\mathcal V}_\ell}(x,\ri \omega^{{\mathcal V}_\ell}(x)) \geq \varepsilon > 0$ 
established in Lemma \ref{thm:extend_smoothness} for all $x \in {\mathcal B}(x_\ast,\eta)$.
\end{proof}

\begin{lemma}[Accuracy of the Second Derivatives]\label{thm:accuracy_secdev}
Suppose that $\{ x^{(\ell)} \}$ by Algorithm \ref{alg} when $d=1$ or Algorithm \ref{alge} converges 
to a point $x_\ast$ that satisfies the assumptions in Lemma \ref{thm:extend_smoothness}.
Then, for all $\ell$ large enough, we have
\[
	\left\| \nabla^2 {\mathcal D} (A(x^{(\ell)}))  -  \nabla^2 {\mathcal D}^{{\mathcal V}_\ell} (A(x^{(\ell)}))	\right\|_2	
			=	O(h^{(\ell)}),
\]
where we define $h^{(\ell)} := | x^{(\ell)} - x^{(\ell-1)} |$ for Algorithm \ref{alg} when $d=1$.
Furthermore, if $\nabla^2 {\mathcal D} (A(x_\ast))$ is invertible, then for all $\ell$ large enough
\[
	\left\| \left[ \nabla^2 {\mathcal D} (A(x^{(\ell)})) \right]^{-1}  -  \left[ \nabla^2 {\mathcal D}^{{\mathcal V}_\ell} (A(x^{(\ell)})) \right]^{-1}	\right\|_2	
			=	O(h^{(\ell)}).
\]
\end{lemma}
\begin{proof}
By Lemma \ref{thm:extend_smoothness} there exists an $\eta$ such that both ${\mathcal D}(A(x))$
as well as ${\mathcal D}^{{\mathcal V}_\ell}(A(x))$ for $\ell$ large enough are three times continuously differentiable
for all $x \in {\mathcal B}(x_\ast,\eta)$. Furthermore, we can choose $\ell$ even larger if necessary so 
that ${\mathcal B}(x^{(\ell)},h^{(\ell)}) \subseteq {\mathcal B}(x_\ast,\eta)$.

Now, for Algorithm \ref{alge} exploiting the properties
\begin{equation*}
	\begin{split}
	{\mathcal D}(A(x^{(\ell)}))		 =		{\mathcal D}^{{\mathcal V}_\ell} (A(x^{(\ell)})),
		\;\;\;\;
	\nabla {\mathcal D}(A(x^{(\ell)} ))		=		\nabla {\mathcal D}^{{\mathcal V}_\ell} (A(x^{(\ell)} ))  \quad	{\rm and} 
	\hskip 10ex \\
	{\mathcal D}(A(x^{(\ell)} + h^{(\ell)} e_{p,q}))		 =		{\mathcal D}^{{\mathcal V}_\ell} (A(x^{(\ell)} + h^{(\ell)} e_{p,q}))	
	\hskip 28.5ex
	\end{split}
\end{equation*}
for $p = 1,\dots,d$, $q = p, \dots, d$, and for Algorithm \ref{alg} when $d=1$ exploiting
\begin{equation*}
	\begin{split}
	{\mathcal D}(A(x^{(\ell)}))		=		{\mathcal D}^{{\mathcal V}_\ell} (A(x^{(\ell)} )),
		\;\;\;\;
	{\mathcal D}(A(x^{(\ell-1)}))		=		{\mathcal D}^{{\mathcal V}_\ell} (A(x^{(\ell-1)} ))
		\;\;\;\; 	{\rm and}	\hskip 8ex \\
	{\mathcal D}'(A(x^{(\ell)} ))		=		\left[ {\mathcal D}^{{\mathcal V}_\ell} \right]' (A(x^{(\ell)} )),	\hskip 43.5ex
	\end{split}
\end{equation*}
as well as the uniform boundedness of the third derivatives of ${\mathcal D}^{{\mathcal V}_\ell}(A(x))$
independent of $\ell$ (see Lemma \ref{thm:bound_thirdder} above), the results follow from an application
of Taylor's theorem with a third order remainder. For details we refer to the proof of
Lemma 2.8 in \cite{Kangal2015}; only ${\mathcal D}(A(x)), {\mathcal D}^{{\mathcal V}_\ell}(A(x))$
here takes the role of $\lambda_J(\omega), \lambda_J^{(\ell)}(\omega)$ in that proof. 
\end{proof}

The superlinear rate-of-convergence result below is a consequence of the Hermite interpolation properties,
as well as Lemma \ref{thm:accuracy_secdev} that relates the second derivatives of the distance
functions. Intuitively, an analogy can be made between the reduced function ${\mathcal D}^{{\mathcal V}_\ell}(A(x))$ 
and the model function used by a quasi-Newton method about $x^{(\ell)}$; the reduced function satisfies
\begin{equation*}
	\begin{split}
	{\mathcal D}^{{\mathcal V}_\ell}(A(x^{(\ell)})) = {\mathcal D}(A(x^{(\ell)})), \;
	\nabla {\mathcal D}^{{\mathcal V}_\ell}(A(x^{(\ell)})) = \nabla {\mathcal D}(A(x^{(\ell)})), \\
	\nabla^2 {\mathcal D}^{{\mathcal V}_\ell}(A(x^{(\ell)})) \approx \nabla^2 {\mathcal D}(A(x^{(\ell)}))
	\hskip 19ex
	\end{split}
\end{equation*}
where the gap between the Hessians decay to zero, moreover $x^{(\ell+1)}$ is defined as the maximizer 
of ${\mathcal D}^{{\mathcal V}_\ell}(A(x))$.
Formally, the proof of Theorem 3.3 in \cite{Kangal2015} applies identically with 
${\mathcal D}(A(x)), {\mathcal D}^{{\mathcal V}_\ell}(A(x))$ taking the role of 
$\lambda_J(\omega), \lambda_J^{(\ell)}(\omega)$ in that work
to deduce this superlinear rate-of-convergence result .
\begin{theorem}[Local Superlinear Convergence]\label{thm:convergence_rate}
Suppose that the sequence $\{x^{(\ell)} \}$ generated by Algorithm \ref{alg} when $d=1$ or
by Algorithm \ref{alge} converges to a point $x_\ast$ that is strictly in the interior of $\widetilde{\Omega}$
such that \textbf{(i)} ${\mathcal D}(A(x_\ast))$ is attained at $z_\ast \in \ri {\mathbb R}$ uniquely, 
$\sigma(x_\ast,z_\ast)$ is simple, $\sigma_{\omega\omega}(x_\ast, z_\ast) > 0$, $\sigma_{\alpha}(x_\ast, z_\ast) > 0$,
and \textbf{(ii)} $\nabla^2 {\mathcal D}(A(x_\ast))$
is invertible. Then, there exists a constant $\mu \in {\mathbb R}^+$ such that
\[
	\frac{ \| x^{(\ell+1)} - x_\ast \|_2 }{ \| x^{(\ell)} - x_\ast \|_2 \max \{ \| x^{(\ell)} - x_\ast \|_2, \| x^{(\ell-1)} - x_\ast \|_2 \} }	
	\;\; \leq \;\;		\mu
	\quad \;	\forall \ell \geq 2.
\]
\end{theorem}

\begin{remark}
As indicated in Remark \ref{rem:nonsmooth_rcase}, if $A(x)$ is real-valued, then the global minimizers of
$\sigma(x_\ast, {\rm i} \omega)$ over $\omega \in {\mathbb R}$ are in plus, minus pairs. This means that the uniqueness 
assumption on $z_\ast \in \ri {\mathbb R}$, the point where ${\mathcal D}(A(x_\ast))$ is attained, in the main rate-of-convergence 
result (Theorem \ref{thm:convergence_rate}) and the auxiliary results leading to this main result is no longer true.
In this case, this superlinear rate-of-convergence result still holds, but under the assumption that ${\mathcal D}(A(x_\ast))$
is attained uniquely over all purely imaginary numbers with nonnegative (or with nonpositive) imaginary parts.
It is straightforward to modify the analysis above to this setting by restricting $z$ to the first quadrant in the complex
plane, requiring $z^{(\ell)}$ to have nonnegative imaginary part. 
\end{remark}

\medskip

\section{Solutions of the Reduced Problems}\label{sec:sf_implement}

\subsection{Computation of ${\mathcal D}^{{\mathcal V}_\ell}(A(x))$}
 Algorithms \ref{alg} and \ref{alge} require the maximization of the distance ${\mathcal D}^{{\mathcal V}_\ell}(A(x))$ 
 over $x \in \widetilde{\Omega}$ at step $\ell$. For this maximization the objective
 \begin{equation}\label{eq:rect_obj}
 	{\mathcal D}^{{\mathcal V}_\ell}(A(x)) 
			\; = \;
		\min_{z \in {\mathbb C}^+} \; \sigma_{\min} (A(x) V_\ell  -  z V_\ell )
\end{equation}
needs to be computed at several $x$ in ${\mathbb R}^d$. Assuming the dimension of ${\mathcal V}_\ell$
is small, which is usually the case in practice as the framework converges at a superlinear
rate, it turns out this can be performed efficiently and reliably by means of simple extensions of the 
techniques to compute the distance to uncontrollability \cite{Gu2000, Gu2006} characterized by
\begin{equation}\label{eq:dist_uncont}
	\min_{z \in {\mathbb C}} \;
	\sigma_{\min}
	\left(
		\left[
			\begin{array}{cc}
				F - zI		&	G
			\end{array}
		\right]
	\right)
\end{equation}
for a given pair $F \in {\mathbb C}^{q\times q}, G \in {\mathbb C}^{q\times m}$ with $q \geq m$. 
Essential steps are outlined next.

We first compute a reduced QR factorization of the form
\[
	\left[
		\begin{array}{cc}
			V_\ell	&	A(x) V_\ell
		\end{array}
	\right]
			=
	\underbrace{
	\left[
		\begin{array}{cc}
			V_\ell	&	\widetilde{V}_\ell
		\end{array}
	\right]}_Q
	\underbrace{
	\left[
		\begin{array}{cc}
			I_\ell		&	R_{A}	\\
			0		&	R_{B}	\\
		\end{array}
	\right]}_R \: ,
\] 
where $Q \in {\mathbb C}^{n\times 2\ell}$ has orthonormal columns, $R \in {\mathbb R}^{2\ell \times 2\ell}$
is upper triangular, and $I_\ell$ denotes the identity matrix of size $\ell \times \ell$. Then the singular values 
of $A(x) V_\ell  -  z V_\ell$ and 
\[
	Q^T \left\{ A(x) V_\ell  -  z V_\ell \right\}
		\;	=	\;
	\left[
		\begin{array}{c}
			R_{A}  \\
			R_{B}
		\end{array}
	\right]
			-
	z
	\left[
		\begin{array}{c}
			I_\ell \\
			0
		\end{array}
	\right]
		\;	=	\;
	\left[
		\begin{array}{c}
			R_{A}  - z I_\ell \\
			R_{B}
		\end{array}
	\right]
\]
are the same for every $z$. Hence (\ref{eq:rect_obj}) reduces to
\begin{equation}\label{eq:dist_uncont+}
	{\mathcal D}^{{\mathcal V}_\ell}(A(x)) 
			\; = \;
		\min_{z \in {\mathbb C}^+} \: \sigma_{\min} \left(
								\left[
									\begin{array}{cc}
										R_A^T	- z I_\ell	  &	R_B^T
									\end{array}
								\right] \right)
\end{equation}
which is of the form (\ref{eq:dist_uncont}) except that the minimization has to be performed
over ${\mathbb C}^+$ rather than ${\mathbb C}$. Modifications of the techniques in
\cite{Gu2000, Gu2006} to solve (\ref{eq:dist_uncont+}), in particular to perform optimization 
over ${\mathbb C}^+$, are straightforward. 

In practice we adopt the approach in \cite{Gu2006} combined with BFGS (that employs
line searches ensuring the satisfaction of the weak Wolfe conditions by the iterates). First
BFGS converges to a local minimizer of (\ref{eq:dist_uncont+}), which is then subjected to
the trisection test from \cite{Gu2006} to determine whether the local minimizer
is a global minimizer or not. If it is a global minimizer, ${\mathcal D}^{{\mathcal V}_\ell}(A(x))$
is indeed computed. Otherwise, the trisection test provides a point in ${\mathbb C}^+$
where the singular value function takes a smaller value compared with the converged local
minimizer, and BFGS is restarted with this point. Each trisection test requires the extraction
of the real eigenvalues of either a $2\ell^2 \times 2\ell^2$ matrix or a $4\ell^2 \times 4\ell^2$
matrix pencil. In practice we employ the latter; even though these are larger generalized eigenvalue
problems, they turn out to be better conditioned and it is usually possible to solve them
reliably in the presence of rounding errors.

An alternative approach\footnote{We thank to Daniel Kressner for pointing out this alternative approach.} 
for the computation of ${\mathcal D}^{{\mathcal V}_\ell}(A(x))$ is based on the reduced QR factorization
\[
   	\left[
		\begin{array}{cccc}
			V_\ell	&	A_1 V_\ell		&	\dots		&	A_\kappa V_\ell
		\end{array}
	\right]
		\;\;	=	\;\;
	\widehat{Q}
	\left[
		\begin{array}{cc}
			I_\ell			&		\widehat{R}_A	\\
			0			&		\widehat{R}_B	
		\end{array}
	\right]
\]
where $\widehat{Q} \in {\mathbb C}^{n\times (\kappa+1) \ell}$, $\widehat{R} \in {\mathbb C}^{\ell \times \ell}$,
$\widehat{R}_A \in {\mathbb C}^{\ell \times \kappa \ell}$ and $\widehat{R}_B \in {\mathbb C}^{\kappa \ell \times \kappa \ell}$.
Partitioning
\begin{equation*}
	\widehat{R}_A
		 =
	\left[
		\begin{array}{ccc}
			\widehat{R}_{A,1}	&	\dots		&	\widehat{R}_{A,\kappa}
		\end{array}
	\right],	\quad\quad
	\widehat{R}_B
		 =
	\left[
		\begin{array}{ccc}
			\widehat{R}_{B,1}	&	\dots		&	\widehat{R}_{B,\kappa}
		\end{array}
	\right]
\end{equation*}
so that $\widehat{R}_{A,j} \in {\mathbb C}^{\ell \times \ell}$, 
$\widehat{R}_{B,j} \in {\mathbb C}^{\kappa \ell \times \ell}$, and little effort show that
\[
	{\mathcal D}^{{\mathcal V}_\ell}(A(x)) 
			\; = \;
		\min_{z \in {\mathbb C}^+} \: \sigma_{\min} \left(
								\left[
									\begin{array}{cc}
										\widetilde{R}_A(x)^T	- z I_\ell	  &	\widetilde{R}_B(x)^T
									\end{array}
								\right] \right)
\]
where $\widetilde{R}_A(x) = \sum_{j=1}^\kappa f_j(x) \widehat{R}_{A,j}$, 
$\widetilde{R}_B(x) = \sum_{j=1}^\kappa f_j(x) \widehat{R}_{B,j}$. This approach requires the computation of
only one QR factorization for the minimization of ${\mathcal D}^{{\mathcal V}_\ell}(A(x))$ over $x$ rather than one QR factorization
for each evaluation of  ${\mathcal D}^{{\mathcal V}_\ell}(A(x))$. But the difference it makes appears to be insignificant
in practice, because $V_\ell$ has typically a few columns and the computational cost of these QR factorizations is linear,
quite small compared to other ingredients of the algorithm.

\subsection{Maximization of ${\mathcal D}^{{\mathcal V}_\ell}(A(x))$}
Suppose first that ${\mathcal D}^{{\mathcal V}_\ell}(A(x))$ is attained on the imaginary axis 
for all $x \in \widetilde{\Omega}$. Then the regularity result  in part (i) of Theorem \ref{thm:der_D} 
remains to be true but with ${\mathcal D}^{{\mathcal V}_\ell}(A(x))$ taking the role of ${\mathcal D}(A(x))$.
A consequence is that upper support functions of Theorem \ref{thm:sup_func} extend to this
rectangular setting, that is, for a given $\widetilde{x}$ where ${\mathcal D}^{{\mathcal V}_\ell}(A(x))$
is differentiable, we have
\[
	[ {\mathcal D}^{{\mathcal V}_\ell}(A(x)) ]^2		 \; \leq \;		
		q(x; \widetilde{x})	
		:=	
		[ {\mathcal D}^{{\mathcal V}_\ell}(A(\widetilde{x})) ]^2		+	
		\nabla [{\mathcal D}^{{\mathcal V}_\ell}(A(\widetilde{x}))]^2	(x - \widetilde{x}) 	+	
		\frac{\gamma}{2} \| x - \widetilde{x} \|^2_2 \;\;\; \forall x,
\]
where $\gamma$ now satisfies 
$\lambda_{\max} ( \nabla^2 [{\mathcal D}^{{\mathcal V}_\ell}(A(x))]^2 )  \leq \gamma$ 
for all $x$ such that ${\mathcal D}^{{\mathcal V}_\ell}(A(x))$ is differentiable. Furthermore, estimation of 
such a $\gamma$ is facilitated by simple adaptations of Theorem \ref{thm:bound_sdev},
that is, defining 
$M^{{\mathcal V}_\ell} (x,z) := (A^{V_\ell}(x) - z V_\ell)^\ast (A^{V_\ell}(x) - z V_\ell)$
and denoting with $z(x)$ the point where ${\mathcal D}^{{\mathcal V}_\ell}(A(x))$ is attained,
the bound
\[
	\lambda_{\max} ( \nabla^2 [{\mathcal D}^{{\mathcal V}_\ell}(A(x))]^2 )	
			\;\;	\leq	\;\;	
	\lambda_{\max}
		\left(  \nabla^2_{xx} M^{{\mathcal V}_\ell} (x,z(x)) \right)
\]
holds for all $x$ where ${\mathcal D}^{{\mathcal V}_\ell}(A(x))$ is differentiable.
Specifically, for the affine case, we have 
$A^{V_\ell}(x) = B_0 V_\ell + \sum_{j = 1}^\kappa x_j B_j V_\ell$
and $\nabla^2_{xx} M^{{\mathcal V}_\ell}(x,z)$ is given by the right-hand side of (\ref{eq:matrix_forgamma})
by replacing $B_j$ with $B_j V_\ell$ for $j = 1, \dots, \kappa$. 

The remarks of the previous paragraph also applies if ${\mathcal D}^{{\mathcal V}_\ell}(A(x))$
is attained strictly on the right-hand side of the complex plane for all $x \in \widetilde{\Omega}$.
In this case the arguments deal with both the complex and the imaginary parts of the point
$z(x)$ where ${\mathcal D}^{{\mathcal V}_\ell}(A(x))$ is attained, which, at points of differentiability,
are defined implicitly by 
$\sigma^{{\mathcal V}_\ell}_\alpha (x,z(x)) = 0$, $\sigma^{{\mathcal V}_\ell}_\omega (x,z(x))) = 0$.

Hence, we employ Algorithm \ref{small_alg} for the maximization of 
${\mathcal D}^{{\mathcal V}_\ell}(x)$. But naturally $q(x; x^{(j)})$ is now defined in terms 
of ${\mathcal D}^{{\mathcal V}_\ell}(A(x^{(j)}))$, $\nabla {\mathcal D}^{{\mathcal V}_\ell}(A(x^{(j)}))$ 
and $\gamma$ is a global upper bound on  $\lambda_{\max} ( \nabla^2 [{\mathcal D}^{{\mathcal V}_\ell}(A(x))]^2 )$ 
for all $x \in \widetilde{\Omega}$. The algorithm is guaranteed to converge under the
assumptions of the previous paragraphs. 

A likely situation is that ${\mathcal D}^{{\mathcal V}_\ell}(A(x))$ 
is attained on the imaginary axis for some $x$, yet strictly on the right-hand side of the complex plane 
for some other values of $x$. Even in this case, the global convergence of the algorithm can be asserted provided
$\gamma$ is large enough. Some difficulty arises due to the possibility of points $\widetilde{x}$ such that 
${\mathcal D}^{{\mathcal V}_\ell}(A(\widetilde{x}))$ is attained at $\widetilde{z} \in \ri {\mathbb R}$ with 
$\sigma^{{\mathcal V}_\ell}_\alpha (\widetilde{x}, \widetilde{z}) = 0$. The left-hand derivative of 
${\mathcal D}^{{\mathcal V}_\ell}(A(x))$ in a particular direction at such a point can be smaller than the 
right-hand derivative, which means $q(x; \widetilde{x})$ is not an upper support function anymore. However, 
such a point $\widetilde{x}$ is away from maximizers of ${\mathcal D}^{{\mathcal V}_\ell}(A(x))$, and 
$q(x; \widetilde{x})$ is still an upper support function for ${\mathcal D}^{{\mathcal V}_\ell}(A(x))$ near a 
global maximizer locally provided $\gamma$ is large enough, from which the global convergence of the 
algorithm can be deduced.

In practice we choose $\gamma$ based on the largest eigenvalue of $\nabla^2_{xx} M^{{\mathcal V}_\ell} (x,z(x))$,
which seems to be working well.


\section{A Subspace Framework for Uniformly Stable Problems}\label{sec:sf_uniform_stab} 
If it is known in advance that $A(x)$ has all of its eigenvalues in the closed left-half of the complex
plane at all $x \in \widetilde{\Omega}$ (such is the case for instance for dissipative Hamiltonian
systems, which naturally arise from multi-body problems, circuit simulation, finite element
modeling of a disk brake \cite{Mehl2016}), 
then the reduced problems can be simplified. In particular, there is theoretical ground to perform 
the inner minimization problems over the imaginary axis. Formally, we operate on the 
reduced problems
\[
	\max_{x \in \widetilde{\Omega}}	\: \widetilde{{\mathcal D}}^{{\mathcal V}}(A(x))
	\quad\quad	{\rm with}	\quad
	\widetilde{{\mathcal D}}^{{\mathcal V}}(A(x))
		\; :=	\;
	\min_{\omega \in {\mathbb R}} \: \sigma_{\min}(A^V(x) - \omega \ri V),
\]
where $V$ is a matrix whose columns form an orthonormal basis for ${\mathcal V}$
and $A^V(x)$ is defined as in (\ref{eq:red_mat_funs}). 

The corresponding greedy framework is presented in Algorithm \ref{alg_uniform_stab}. As 
before at every iteration, we solve a reduced problem. We compute the distance to instability 
of the full problem at the maximizing $x$, in particular retrieve $\omega \in {\mathbb R}$ 
where the full distance is attained, and expand the subspace with the inclusion of a
right singular vector corresponding to $\sigma_{\min}(A(x) - \ri \omega I)$ at the
optimal $x$ and $\omega$.

Following the footsteps of the arguments in Sections \ref{sec:large_prob} and \ref{sec:convergence}, 
it is possible to deduce (i) the monotonicity, that is 
$
	{\mathcal V} \supseteq {\mathcal W} \; \Longrightarrow \;
		\widetilde{{\mathcal D}}^{{\mathcal V}}(A(x)) \leq \widetilde{{\mathcal D}}^{{\mathcal W}}(A(x)),
$
and (ii) Hermite interpolation, that is 
\begin{enumerate}
  \item[\bf (1)]
  $\widetilde{{\mathcal D}}^{{\mathcal V}_\ell}(A(x^{(j)}))	 \; = \; {\mathcal D}(A(x^{(j)}))$, as well as
  \item[\bf (2)]
  if ${\mathcal D}(A(x))$ and $\widetilde{{\mathcal D}}^{{\mathcal V}_\ell}(A(x))$ are differentiable at $x = x^{(j)}$, then
  $\nabla \widetilde{{\mathcal D}}^{{\mathcal V}_\ell}(A(x^{(j)}))	 \; = \; \nabla {\mathcal D}(A(x^{(j)}))$
\end{enumerate}
for $j = 1,\dots, \ell$. These properties, as before, pave the way for global convergence and
superlinear rate-of-convergence results presented formally below.
\begin{theorem}\label{thm:global_conv_uniform_stab}
Suppose that the eigenvalues of $A(x)$ are contained in the closed left-half of the complex plane at all $x \in \widetilde{\Omega}$.
Every convergent subsequence of the sequence $\{ x^{(\ell)} \}$ generated by Algorithm \ref{alg_uniform_stab} 
in the infinite dimensional setting converges to a global maximizer of ${\mathcal D}(A(x))$ over all 
$x \in \widetilde{\Omega}$. Furthermore,
\[
	\lim_{\ell\rightarrow \infty} \; \widetilde{{\mathcal D}}^{{\mathcal V}_\ell}(A(x^{(\ell+1)}))
		\; = \;
	\lim_{\ell\rightarrow \infty} \; \max_{x \in \widetilde{\Omega}} \: \widetilde{{\mathcal D}}^{{\mathcal V}_\ell} (A(x))
		\; = \;
	\max_{x \in \widetilde{\Omega}} \: {\mathcal D}(A(x)).
\]
\end{theorem}
\begin{theorem}\label{thm:convergence_rate_uniform_stab}
Suppose that the eigenvalues of $A(x)$ are contained in the closed left-half of the complex plane at all $x \in \widetilde{\Omega}$.
Suppose also that the sequence $\{x^{(\ell)} \}$ generated by Algorithm \ref{alg_uniform_stab} when $d=1$ 
converges to a point $x_\ast$ that is strictly in the interior of $\widetilde{\Omega}$ such that 
\textbf{(i)} ${\mathcal D}(A(x_\ast))$ is attained at $\omega_\ast$ uniquely, 
$\sigma_{\min}(A(x_\ast) - \omega_\ast {\rm i} I)$ is positive and simple, 
$\partial^2 \left[ \sigma_{\min}(A(x_\ast) - \omega_\ast {\rm i} I) \right] / \partial \omega^2 > 0$, as well as
\textbf{(ii)} ${\mathcal D}''(A(x_\ast)) \neq 0$. Then, there exists a constant
$\mu \in {\mathbb R}^+$ such that
\[
	\frac{ | x^{(\ell+1)} - x_\ast | }
		{ | x^{(\ell)} - x_\ast | \max \{ | x^{(\ell)} - x_\ast |, | x^{(\ell-1)} - x_\ast | \} }	\;\; \leq \;\;		\mu
	\quad \;	\forall \ell \geq 2.
\]
\end{theorem}
It is also possible to define an extended variant of Algorithm \ref{alg_uniform_stab},
reminiscent of Algorithm \ref{alge}, to which the global convergence result 
(Theorem \ref{thm:global_conv_uniform_stab}) and the superlinear convergence 
result (Theorem \ref{thm:convergence_rate_uniform_stab}) apply for every $d$.

\begin{algorithm}
 \begin{algorithmic}[1]
\REQUIRE{ The matrix-valued function $A(x)$ of the form (\ref{eq:prob_defn}) with the feasible region $\widetilde{\Omega}$.}
\ENSURE{ The sequences $\{x^{(\ell)} \}$, $\{ \omega^{(\ell)} \}$.}
\STATE $x^{(1)} \gets$ a random point in $\widetilde{\Omega}$.
\STATE $\omega^{(1)} \gets \arg \min_{\omega \in {\mathbb R}} \; \sigma_{\min} (A(x^{(1)}) - \omega \ri I)$.
\STATE $V_1	\;	\gets \;$ a unit right singular vector corresponding to $\sigma_{\min} (A(x^{(1)}) - \omega^{(1)} \ri I)$. 
\STATE ${\mathcal V}_1 \; \gets \; {\rm span} \{ V_1 \}$.
\FOR{$\ell \; = \; 1, \; 2, \; \dots$}
	\STATE  $x^{(\ell+1)} \gets \arg \max_{x \in \widetilde{\Omega}}	\: \widetilde{{\mathcal D}}^{{\mathcal V}_\ell}(A(x))$.
	\STATE $\omega^{(\ell+1)} \gets \arg \min_{\omega \in {\mathbb R}} \; \sigma_{\min} (A(x^{(\ell+1)}) - \omega \ri I)$.
	\STATE $v_{\ell+1}	\;	\gets \;$ a right singular vector corresponding to $\sigma_{\min} (A(x^{(\ell+1)}) - \omega^{(\ell+1)} \ri I)$.
	\STATE $V_{\ell+1} 	\;	\gets	\; {\rm orth} 
			\left( \left[ \begin{array}{cc} V_{\ell} & v_{\ell+1} \end{array} \right] \right)$ 
			and ${\mathcal V}_{\ell+1} \; \gets \; {\rm Col} (V_{\ell+1})$.
\ENDFOR
 \end{algorithmic}
\caption{The Subspace Framework for Uniformly Stable Problems}
\label{alg_uniform_stab}
\end{algorithm}

When solving the reduced problem at step $\ell$, the objective $\widetilde{{\mathcal D}}^{{\mathcal V}_\ell}(A(x))$ needs
to be computed at several points. To perform the calculation of this objective efficiently at a given $x \in {\mathbb R}^d$, 
inspired by the approach in \cite{Kressner2014} in the context of a subspace framework for the pseudospectral
abscissa, we first compute a reduced QR factorization
\[
		\left[
			\begin{array}{cc}
				V_\ell 	&	A(x) V_\ell
			\end{array}
		\right]
				\;	=	\;
		Q
		\left[
			\begin{array}{cc}
				R_1		&	R_2
			\end{array}
		\right] 
\]
where $Q \in {\mathbb C}^{n\times 2\ell}$, $R_1, R_2 \in {\mathbb C}^{2\ell \times \ell}$. It follows that the singular
values of $A(x) V_\ell - \omega \ri V_\ell$ and $R_2 - \omega \ri R_1$ are the same for all $\omega \in {\mathbb R}$,
so
\[
	\widetilde{{\mathcal D}}^{{\mathcal V}_\ell}(A(x))
		\; =	\;
	\min_{\omega \in {\mathbb R}} \; \sigma_{\min}(R_2 - \omega \ri R_1).
\]
We solve the singular value minimization problem on the right above by employing the globally convergent 
level-set algorithms for the distance to instability \cite{Boyd1990, Bruinsma1990}. These minimization problems
are cheap to solve, because they involve $2\ell \times \ell$ matrix-valued functions. 

Finally for the maximization of $\widetilde{{\mathcal D}}^{{\mathcal V}_\ell}(A(x))$ we apply Algorithm \ref{small_alg}
in Section \ref{sec:small_prob} with 
$\widetilde{{\mathcal D}}^{{\mathcal V}_\ell}(A(x))$ taking the role of ${\mathcal D}(A(x))$. 
Every convergent subsequence of the resulting sequence,
under the assumption that $A(x)$ is asymptotically stable for all $x \in \widetilde{\Omega}$,  
is guaranteed to converge to a global maximizer of $\widetilde{{\mathcal D}}^{{\mathcal V}_\ell}(A(x))$
provided $\gamma$ is chosen such that 
$\lambda_{\max}(\nabla^2 [ \widetilde{{\mathcal D}}^{{\mathcal V}_\ell}(A(x)) ]^2) \leq \gamma$
for all $x \in \widetilde{\Omega}$. 
In the affine case with $A^{V_\ell}(x) = B_0 V_\ell + \sum_{j=1}^\kappa x_j B_j V_\ell$, 
a theoretically sound choice for $\gamma$ is given by the largest eigenvalue of the matrix on the right-hand side
of (\ref{eq:matrix_forgamma}) but with blocks defined in terms of $B_j V_\ell$ rather than $B_j$.

\section{Numerical Results for the Subspace Framework}\label{sec:num_exps}
In this section we report the results obtained from several numerical experiments that are carried out 
with a publicly available Matlab implementation \cite{Mengi2018} of Algorithm \ref{alg} to maximize
${\mathcal D}(A(x))$ when $A(x)$ is affine and $x$ is subject to box-constraints.
We do not make an implementation of Algorithm \ref{alg_uniform_stab} available at this
point, as it typically does not converge to a desired maximizer unless the problem is uniformly stable.
But some numerical results illustrating the convergence of Algorithm \ref{alg_uniform_stab} is 
also reported at the end of Section \ref{sec:num_exp_illustrate}  

\subsection{Illustration of the Algorithm on a $200 \times 200$ Example}\label{sec:num_exp_illustrate}
We first demonstrate Algorithm \ref{alg} and its convergence on the example of Figure \ref{fig:small_random2}, 
which concerns the maximization of ${\mathcal D}(A + k b c^T)$ for a random 
$A \in {\mathbb R}^{200\times 200}$ and random $b,c \in {\mathbb R}^{200}$. Recall that the distance 
function is smooth at the global maximizer for this example.

To make the example more challenging we replace $A$ with $\widetilde{A} = A - 0.08 b c^T$,
which translates the graph of the distance function horizontally by 0.08. We maximize 
${\mathcal D}(\widetilde{A} + k b c^T)$ over $k \in [-0.2, 0.2]$. This distance function is
asymptotically stable only in a small subinterval of $[-0.2,0.2]$, and $\widetilde{A}$ itself
is unstable. The locally convergent methods in the literature would fail
and stagnate at an unstable point unless they are initiated with an asymptotically stable
point on this small interval, which is hard to know in advance. In contrast, Algorithm \ref{alg}
locates an asymptotically stable point that maximizes the distance to instability
regardless of the initial random point for Hermite interpolation. (As a convention we always
choose the initial point as the midpoint of the box or the interval in our implementation, 
so for this particular example we start with $0$. But any other point in the interval should 
work equally well.)

Figure \ref{fig:support} displays the full distance function (the solid curve) and the reduced distance 
functions (the dashed, dashed-dotted, dotted curves) at the third, fourth, fifth steps of the algorithm. 
The dashed curve representing the reduced function at the third step with a three 
dimensional subspace interpolates the full distance function at at $k=-0.2, 0, 0.2$ (marked with circles). 
This reduced function is maximized at $k = 0.10612$, so at the next iteration the 
full distance function is interpolated at this maximizer as well. This leads to the dashed-dotted curve in the figure, 
which represents the reduced function at the fourth step of the algorithm with a four dimensional subspace. 
Next the full distance function is interpolated at the maximizer $k = 0.08019$ of the dashed-dotted curve, in addition
to the existing interpolation points, leading to the dotted curve representing the reduced function at step five with 
a five dimensional subspace. Observe that, as the subspace dimension increases, the reduced functions tend 
to the full distance function. Observe also that, due to monotonicity, each reduced function is squeezed in 
between the full distance function that lies below, and the reduced functions with smaller subspaces that lie above. 

The iterates of Algorithm \ref{alg} seem to converge at least at a superlinear rate on this example; this is depicted
in Table \ref{table:rate_conv}, in particular see the decay of the errors of the maximal values of the reduced distance 
functions on the third column of the bottom row, as well as the errors of the maximizers of the reduced distance 
functions on the third column of the top row. 

The minimum of $\sigma^{{\mathcal V}_\ell} (x^{(\ell+1)}, z)$ over $z \in {\mathbb C}^+$ must be attained 
on the imaginary axis at the later steps of Algorithm \ref{alg} in theory. This is a consequence of the facts that
$\| x^{(\ell+1)} - x^{(\ell)} \| \rightarrow 0$, as well as $\sigma^{{\mathcal V}_\ell} (x^{(\ell)}, z)$ must be minimized 
at $z^{(\ell)}$ on the imaginary axis due to Theorem \ref{thm:interpolate}. This property holds to be true for all
the examples that we have experimented with. For the particular example, the global minimizer of
$\sigma^{{\mathcal V}_\ell} (x^{(\ell+1)}, z)$ over $z \in {\mathbb C}^+$ is listed with respect to $\ell$
on the fourth column in the top row of Table \ref{table:rate_conv}; notice that the minimizer is strictly on the 
right-hand side in the first step, but in all other steps the minimizer is always on the imaginary axis.

The $\gamma$ values employed for the solution of the reduced problems increase monotonically as the 
subspace dimension increases. Typically in the first steps with small dimensional subspaces we deal
with subproblems that require much smaller values of $\gamma$ compared to the full problem.
This is a feature that contributes to the efficient solution of the reduced problems. The $\gamma$ values
employed for the reduced problems for the particular example are listed on the right-most column of the top row 
in Table \ref{table:rate_conv}; much smaller values of $\gamma$ are used compared with the full problem
for which the bound that follows from (\ref{eq:matrix_forgamma}) is 
$\gamma = 2 \| b c^T \|_2^2 = 2 \| b \|_2^2 \| c \|_2^2 = 72836$.

\begin{figure}[h]
		\hskip -2ex
		\includegraphics[width=\textwidth]{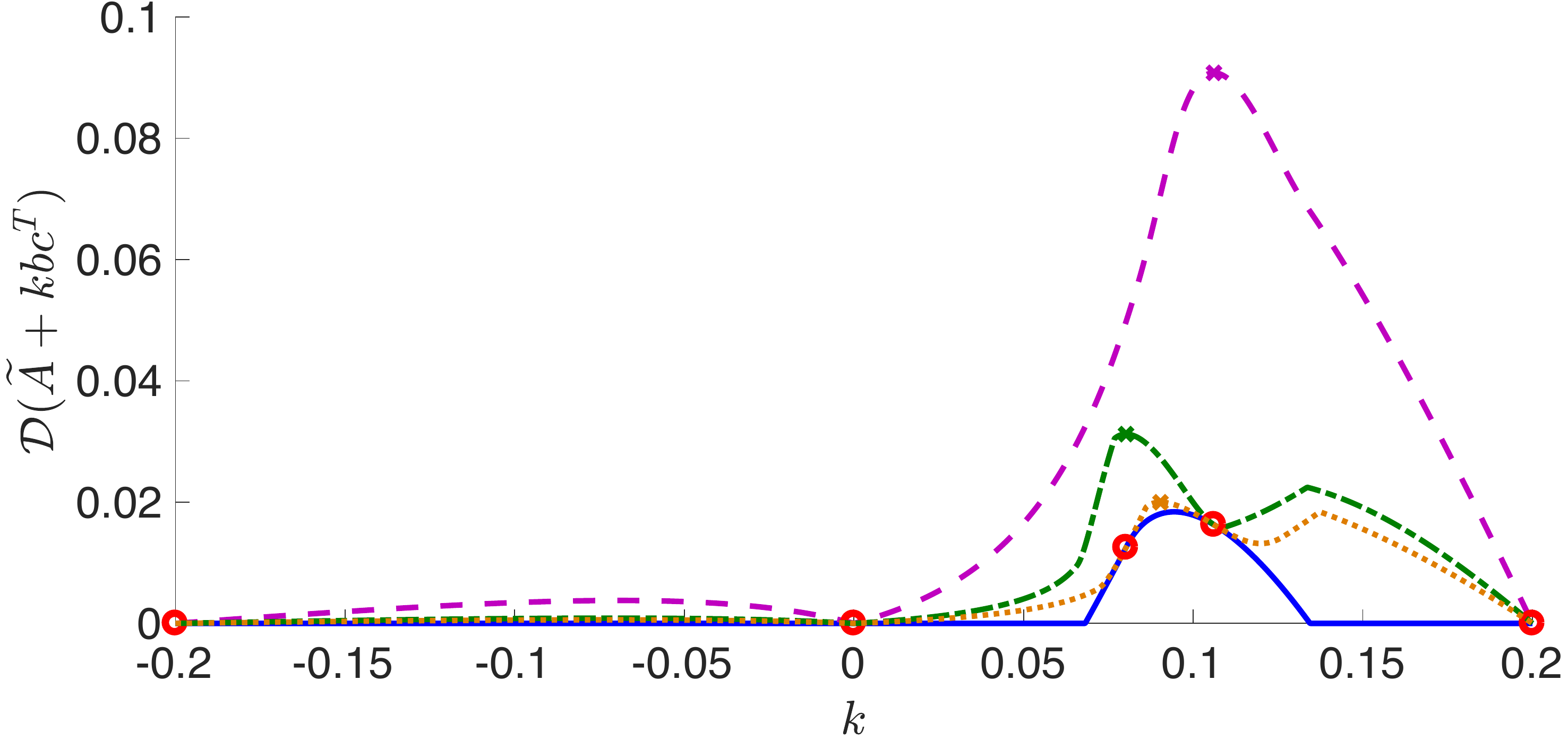} 	
	  \caption{ 
	  		This figure displays the progress of Algorithm \ref{alg} on a variation of the example of Figure 
			\ref{fig:small_random2}, which concerns the robust stability of a random matrix 
			$A \in {\mathbb R}^{200\times 200}$ subject to rank one updates $k b c^T$ 
			for given random $b, c \in {\mathbb R}^{200}$. Here we maximize 
			${\mathcal D}(\widetilde{A} + k b c^T)$ over $k \in [-0.2,0.2]$ where
			$\widetilde{A} := A - 0.08 b c^T$.
			The solid curve is a plot of the full distance function ${\mathcal D}(\widetilde{A} + k b c^T)$ 
			with respect to $k$, whereas
			the dashed, dashed-dotted, dotted curves are reduced functions operated 
			on at step $k$ of Algorithm \ref{alg} with a $k$ dimensional subspace for $k = 3, 4, 5$, respectively. 
			The crosses mark the global maxima of the reduced functions, 
			whereas the circles mark the points where the full distance function is Hermite-interpolated.
	   	     }
	     \label{fig:support}
\end{figure}

\begin{table}[h]
\begin{center}
\begin{tabular}{|c|rccc|}
	\hline
	$\ell$	&		$ x^{(\ell+1)} \quad$ 	&	$\quad | x^{(\ell+1)} - x_\ast | \quad$	 &	$\arg\min_{z\in {\mathbb C}^+} \sigma^{{\mathcal V}_\ell} (x^{(\ell+1)},z)$	&	$\gamma$ 	 \\
	\hline
	\hline
	1	&	 -0.2000000000		&	0.2943923006	&	 0.348268		&	164 			\\
	2	&	 0.2000000000		&	0.1056076994	&	0				&	948 			\\
	3	&	 0.1061223556		&	0.0117300550	&	 0.367974$\ri$		&	983 			\\
	4	&	 0.0801946598		&	0.0141976408	&	-0.249139$\ri$		&	1346 		\\
	5	&	 0.0903625728		&	0.0040297278	&	-0.294728$\ri$		&	1465  		\\
	6	&	 0.0944455439		&	0.0000532433	&	-0.323530$\ri$		&	2201 		\\
	7	&	 0.0943923006		&	0.0000000143	&	-0.323085$\ri$		&	2374			\\
	\hline
\end{tabular}
\end{center}

\vskip 1ex

\begin{center}
\begin{tabular}{|c|cc|}
	\hline
	$\ell$		&	${\quad \mathcal D}^{{\mathcal V}_\ell}(A(x^{(\ell+1)}))$		&		$\quad {\mathcal D}^{{\mathcal V}_\ell}(A(x^{(\ell+1)})) - {\mathcal D}_\ast \quad$  \\
	\hline
	\hline
	1	&	1.8127144436		&	1.7943202283		 \\
	2	&	0.7324009260		&	0.7140067108		 \\
	3	&	0.0907849141		&	0.0723906988		 \\
	4	&	0.0312089149		&	0.0128146997		 \\
	5	&	0.0199709057		&	0.0015766904		 \\
	6	&	0.0184001783		&	0.0000059630	 	\\
	7	&	0.0183942153		&	0.0000000000		\\
	\hline
\end{tabular}
\end{center}


\caption{ The table concerns an application of Algorithm \ref{alg} to the example of Figure \ref{fig:support}
on the interval $[-0.2,0.2]$.
It lists the iterates generated by Algorithm \ref{alg}, as well as their errors, 
the points in ${\mathbb C}^+$ where ${\mathcal D}^{{\mathcal V}_\ell}(A(x^{(\ell+1)}))$ are attained 
and the $\gamma$ values used
for the reduced problems. The notations ${\mathcal D}_\ast$ and $x_\ast$ represent the maximal value of
${\mathcal D}(\widetilde{A} + k b c^T)$ over $k \in [-0.2,0.2]$ and the corresponding global maximizer, 
respectively.
}
\label{table:rate_conv}
\end{table}

Algorithm \ref{alg_uniform_stab} which performs the inner minimization over the imaginary axis
rather than the right-half of the complex plane fails to converge, indeed stagnates at an unstable
point, on this example on the interval $[-0.2, 0.2]$. Similar phenomenon occurs even on smaller 
intervals, such as $[0.05,0.15]$, as long as the interval contains unstable points. This algorithm
however converges to the correct global maximizer at a superlinear rate on the interval $[0.07,0.13]$ where
uniform stability holds; this is consistent with what is expected in theory, in particular with Theorems
\ref{thm:global_conv_uniform_stab} and \ref{thm:convergence_rate_uniform_stab}. 
The maximal values of the reduced distance functions by Algorithm \ref{alg_uniform_stab}
on this example are listed in Table \ref{table:rate_conv_uniform}.

\begin{table}
\begin{center}
\begin{tabular}{|c|cc|}
	\hline
	$\ell$		&	${\quad \widetilde{\mathcal D}}^{{\mathcal V}_\ell}(A(x^{(\ell+1)}))$		&		$\quad {\widetilde{\mathcal D}}^{{\mathcal V}_\ell}(A(x^{(\ell+1)})) - {\mathcal D}_\ast \quad$  \\
	\hline
	\hline
	1	&	0.4386346102		&	0.4202403949		 \\
	2	&	0.2183039380		&	0.1999097228		 \\
	3	&	0.0188802168		&	0.0004860015		 \\
	4	&	0.0183970429		&	0.0000028276		 \\
	5	&	0.0183942153		&	0.0000000000		 \\
	\hline
\end{tabular}
\end{center}

\caption{ This table concerns an application of Algorithm \ref{alg_uniform_stab} to the 
example of Figure \ref{fig:support} on the interval $[0.07,0.13]$ where uniform stability holds.
It lists the maximal values of the reduced distance functions and their errors
with respect to the subspace dimension.
}
\label{table:rate_conv_uniform}
\end{table}

\subsection{Results on Large Random Examples}\label{sec:numexp_large}
We test the performance of Algorithm \ref{alg} once again to maximize ${\mathcal D}(A + k b c^T)$
over $k \in [-3,3]$ with random $A \in {\mathbb R}^{n\times n}$, $b, c \in {\mathbb R}^n$
for $n = 400, 800, 1200, 1600, 2000$. In each of these examples, $A$ is shifted by a multiple of the identity
and normalized so that $\| A \|_2 = 10$ and its spectral abscissa lies in $[-0.31, -0.23]$, and the vectors 
$b, c$ are normalized so that $\| b \|_2 = \| c \|_2 = \sqrt{50}$. 
The exact data is available on the 
web\footnote{\url{http://home.ku.edu.tr/~emengi/software/max_di/Data_&_Updates.html}}.

Convergence to a global maximizer at a superlinear rate is realized in practice for each one of these examples.
Figure \ref{fig:large_dist} confirms the correctness of the computed maximizers for $n = 400$ (dotted curve),
$n = 800$ (dashed curve), $n = 1200$ (solid curve); the distance to instability functions are displayed in these
figures along with the computed global maxima marked with circles.

The main purpose of these experiments is to illustrate the factors determining the efficiency of Algorithm \ref{alg} 
as the sizes of the matrices increase. According to Table \ref{table:large_runtimes} the main factor that dominates 
the run-times as $n$ increases is the computation of the distance to instability for the full problem a few times,
i.e., the solution of the large-scale singular value minimization problems in lines \ref{large_dinstab} and \ref{large_dinstab2}.
The number of large-scale distance to instability computations is at most one more than the number of subspace iterations,
which varies between 6-10 for these examples. The solution of the 
reduced problems take significant portion of the run-time for small $n$ (e.g., $n = 400$), but its effect on the total
runtime diminishes as $n$  increases. Indeed the total time consumed for the solution of the reduced problems appear 
more or less independent of $n$.

For all of these examples we would be using $\gamma = 2 \| b c^T \|_2^2 = 5000$, had we rely on formula (\ref{eq:matrix_forgamma})
and had we been performing optimization directly on the full problem. But as in the previous subsection, we
benefit from projections and use $\gamma = 2 \| V_\ell^T b (V_\ell^T c)^T \|_2^2$ for the reduced problem at
step $\ell$, which turns out to be much smaller than the one for the full problem. Indeed the $\gamma$ values chosen
this way for the reduced problems never exceed 400 for the particular examples; this is evident in 
Figure \ref{fig:gamma_values}, where the $\gamma$ values are plotted for the reduced problems with respect to 
the iteration number.

\begin{figure}
		\includegraphics[width=\textwidth]{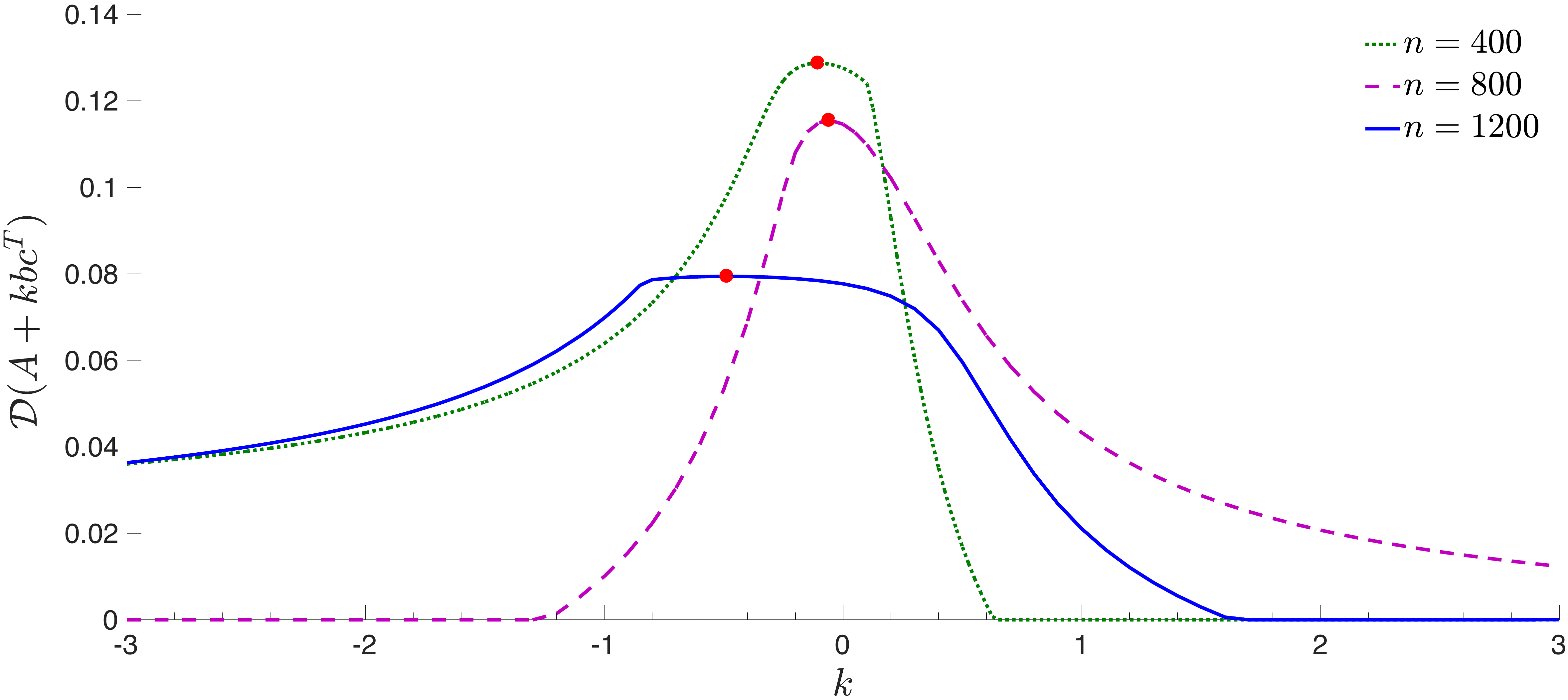} 	
	  \caption{ 
	  		The plots of ${\mathcal D}(A + k b c^T)$ as a function of $k \in [-3,3]$
			for the examples of Section \ref{sec:numexp_large} for $n = 400$ (dotted curve),
			$n$ = 800 (dashed curve), $n$ = 1200 (solid curve). The circle marks the computed 
			global maximum for each one of the three cases.
	   	     }
	     \label{fig:large_dist}
\end{figure}

\begin{table}
\begin{center}
\begin{tabular}{|c|cccccc|}
	\hline
	$n$		&		$\#$ iter 	& time	 &	reduced	&	dist. instab.	&  $x_\ast$		&	${\mathcal D}_{\ast}$  \\
	\hline
	\hline
	400		&	7			&	154	 &	123		&	30		&     -0.10565 	&	0.12870882		\\
	800		&	6			&	200	 &	72		&	124		&     -0.05943	&	0.11545563		\\
	1200		&	6			&	632	 &	245		&	372		&     -0.48694	&	0.07941192		\\
	1600		&	9			&	1646  &	366		&	1228	 	&     -0.05009	&	0.07829585		\\
	2000		&	10			&	4288 &	279		&	3896		&      $\:$ 0.04762	&	0.08436380 	\\
	\hline
\end{tabular}
\end{center}

\caption{The number of subspace iterations (2nd column), run-times (3rd-5th columns), computed globally maximal value of the 
distance ${\mathcal D}(A + k b c^T)$ (7th column) and the corresponding global maximizer (6th column) are listed for each 
one of the examples in Section \ref{sec:numexp_large} with respect to $n$. The 3rd, 4th, 5th columns
provide the total run-time, total time for the reduced problems, total time for large-scale distance to instability computations
in seconds. }
\label{table:large_runtimes}
\end{table}

\begin{figure}
		\includegraphics[width=\textwidth]{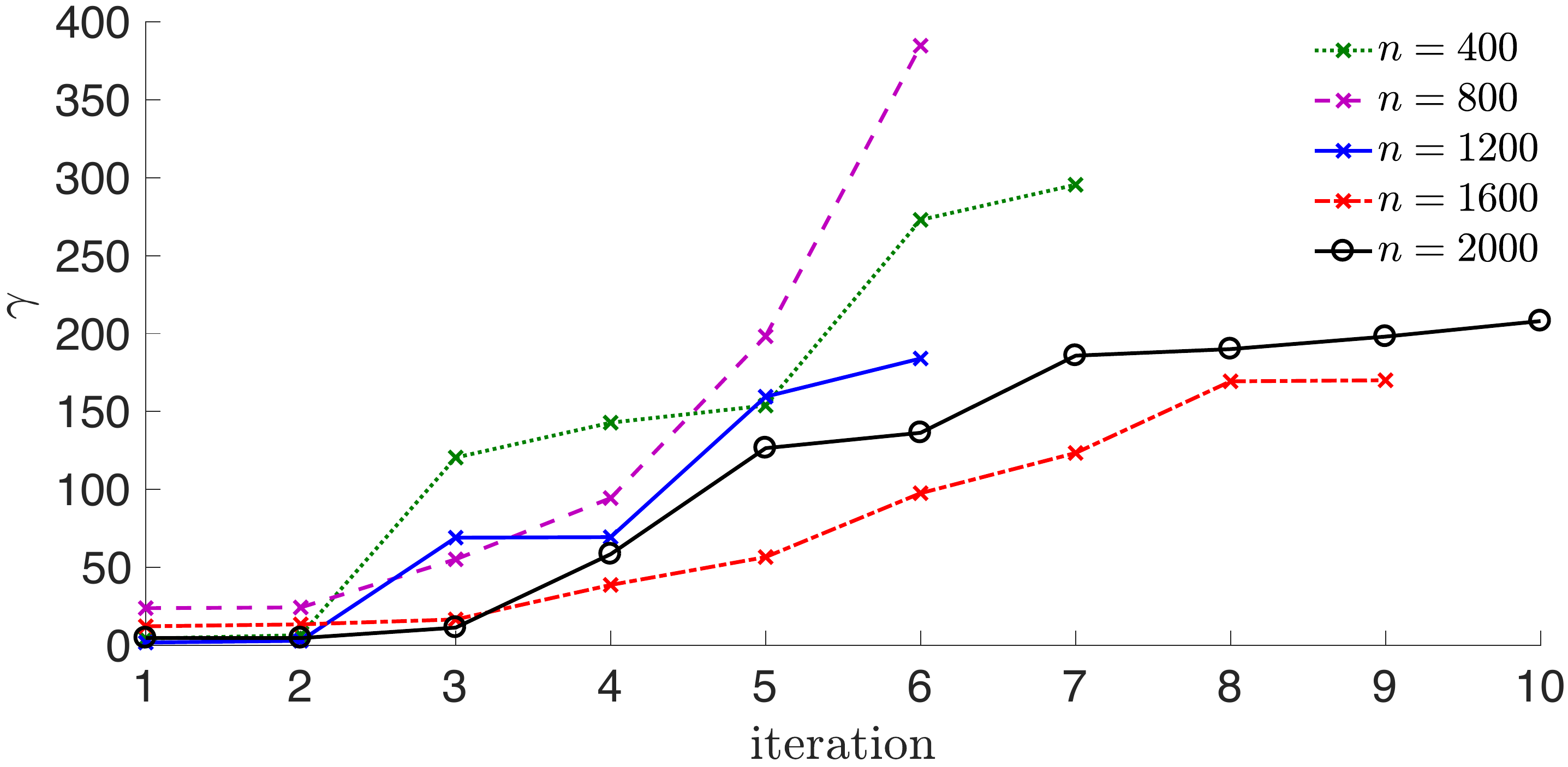} 	
	  \caption{ 
	  		A plot of the $\gamma$ values employed for the reduced problems for the examples
			of Section \ref{sec:numexp_large} with respect to the iteration number. The 
			dotted, dashed, solid, dotted-dashed curves marked with crosses correspond
			to the problems with $n = 400$, $n = 800$, $n = 1200$, $n = 1600$, respectively.
			 The solid curve marked with circles corresponds to the problem with $n = 2000$.
	   	     }
	     \label{fig:gamma_values}
\end{figure}

\section{Concluding Remarks}
We have focused on the maximization of the distance to instability of a matrix dependent on several
parameters. In the systems setting this is motivated by maximizing the robust stability and minimizing the
transient behavior of the associated autonomous control system. Existing approaches including BFGS, gradient 
sampling, and bundle methods converge to a locally optimal solution. This local convergence attribute is problematic
especially if the starting point is unstable; then all of these methods would stagnate at the unstable point.
Unlike these existing approaches in the literature, we have described approaches that converge to a
globally optimal point. In particular, even if our approaches are initiated with unstable points, they are capable
of locating stable points, and beyond robustly stable points where distance to instability is maximized.

The problems where the parameter-dependent matrix is small are dealt by means of a support function 
based approach, an adaptation of the approach in \cite{Mengi2014} that approximates the distance function 
with a piece-wise quadratic model which also lies above the distance function globally. Arguments are
presented in support of its global convergence. 

The problems where the parameter-dependent matrix is large are dealt by means of a subspace framework.
Here the matrices are restricted to subspaces from the right-hand side leading to rectangular problems. At every iteration
such a restricted problem is solved efficiently by employing the support function based approach for small 
problems. Then the subspace is expanded carefully in a way so that Hermite interpolation properties hold 
between the full problem and the restricted problem at the optimal point of the restricted problem.
We have established formally that the Hermite interpolation property along with a monotonicity property
give rise to global convergence at a superlinear rate with respect to the subspace dimension.
The proposed approaches for both the small-scale and the large-scale problems are implemented in Matlab
and made available on the internet \cite{Mengi2018}.

The present work extend the approaches in \cite{Mengi2014} and \cite{Kangal2015} to a more challenging
setting; whereas the approaches in those previous works concern the minimization of the $J$th largest eigenvalue,
the approaches here are tailored for a maximin optimization problem with a smallest eigenvalue as the
objective. There are several other eigenvalue optimization problems of similar spirit with a maximin or a minimax 
structure, e.g., minimization of the numerical radius, minimization of the ${\mathcal H}_\infty$-norm 
of a parameter dependent control system, minimization of the pseudospectral abscissa to name a few.
We believe that the ideas and approaches developed in this paper set an example, and are likely to be 
applicable to those contexts as well.

\smallskip

\noindent
\textbf{Acknowledgements.} Parts of this work have been carried out during author's visits to
\'{E}cole Polytechnique F\'{e}d\'{e}rale de Lausanne (EPFL) and Courant Institute of New York University (NYU).
The author is grateful to Daniel Kressner and Michael Overton for hosting him at EPFL and NYU, respectively,
as well as for fruitful discussions.

\bibliography{optimize_distinstab}

\end{document}